\documentclass[11pt,longtable,letterpaper,oneside,reqno]{amsart}
\usepackage{amsmath,amsfonts,amsthm,amssymb,amscd,cancel,hyperref,stmaryrd,subfig,todonotes,url}
\usepackage{algpseudocode}
\usepackage{anysize,placeins,enumitem}
\usepackage[toc,page]{appendix}
\usepackage{bbm}
\usepackage{bbold}
\usepackage{booktabs}
\usepackage{cancel}
\usepackage{caption}
\usepackage{cases}
\usepackage{color}
\usepackage{epsfig}
\usepackage{float,mathtools}
\usepackage{graphicx}
\usepackage{latexsym}
\usepackage{longtable,tabularx,tabulary}
\usepackage{mathrsfs}
\usepackage{mathtools}
\usepackage[singlelinecheck=false]{caption}
\usepackage{subcaption}
\usepackage{textcomp}
\usepackage{textgreek}
\usepackage{titletoc}
\usepackage{todonotes}
\usepackage{ulem}
\usepackage{soul}
\usepackage{verbatim}
\usepackage{xcolor}

\newcommand{\R}{{\mathbb{R}}}

\newcommand{\Q}{{\mathbb{Q}}}

\renewcommand{\Re}{{\mathfrak{Re}}}
\renewcommand{\Im}{{\mathfrak{Im}}}
\DeclareMathOperator{\Norm}{N}

\makeatletter
\let\@wraptoccontribs\wraptoccontribs
\makeatother

\newcommand{\s}{\sigma}

\numberwithin{equation}{section}
\newtheorem{theorem}{Theorem}
 \newtheorem{corollary}{Corollary}[theorem]
 \newtheorem{lemma}[theorem]{Lemma}

 \newtheorem{proposition}[theorem]{Proposition}
  \theoremstyle{remark}
 \newtheorem{remark}{Remark}

\newcommand*{\thmref}[1]{Theorem~\ref{#1}}
\newcommand*{\lmaref}[1]{Lemma~\ref{#1}}

\newcommand*{\propref}[1]{Proposition~\ref{#1}}
\newcommand*{\corref}[1]{Corollary~\ref{#1}}

\reversemarginpar

\begin{document}

\title[An effective version of Chebotarev's density theorem]{An effective version of Chebotarev's density theorem}

\author[S. Das]{Sourabhashis Das}
\address{Department of Pure Mathematics\\
University of Waterloo\\
200 University Avenue West\\
Waterloo, Ontario\\
N2L 3G1 Canada}
\email{s57das@uwaterloo.ca}

\author[H. Kadiri]{Habiba Kadiri}
\address{Department of Mathematics and Computer Science\\
University of Lethbridge\\
4401 University Drive\\
Lethbridge, Alberta\\
T1K 3M4 Canada}
\email{habiba.kadiri@uleth.ca}

\author[N. Ng]{Nathan Ng}
\address{Department of Mathematics and Computer Science\\
University of Lethbridge\\
4401 University Drive\\
Lethbridge, Alberta\\
T1K 3M4 Canada}
\email{nathan.ng@uleth.ca}

\begin{abstract}
Chebotarev's density theorem asserts that the prime ideals are equidistributed among the conjugacy classes of the Galois group of any normal extension of number fields. An effective version of this theorem was first established by Lagarias and Odlyzko in 1977. In this article, we present an explicit refinement of their statement that applies to all non-rational fields, with every implicit constant expressed explicitly in terms of the field invariants. Additionally, we provide a sharper bound for extensions of sufficiently small degree.
Our approach begins by proving an explicit formula for a smoothed prime ideal counting function. This relies on recent zero-free regions for Dedekind zeta functions, improved estimates on the number of low-lying zeros, and precise bounds for sums over the non-trivial zeros of the Dedekind $\zeta$-function.
\end{abstract}



\maketitle

\newcommand{\sourabh}[1]{{\color{violet} \sf  Sourabh: #1}}
\newcommand{\nathan}[1]{{\color{purple} \sf  Nathan: #1}}
\newcommand{\habiba}[1]{{\color{brown} \sf  Habiba: #1}}

\section{Introduction}

Chebotarev's density has many number theoretic applications including to modular forms, Galois representations, and to binary quadratic forms (see \cite{mm}, \cite{js} for many examples). As Chebotarev's density theorem is phrased in the language of algebraic number theory, we now introduce the necessary notation.

Let $L/K$ be a normal extension of number fields with Galois group $\operatorname{Gal}(L/K) = G$.  
We let $\Norm I := \Norm_{K/\mathbb{Q}} I$ denote the absolute norm of an ideal $I$ in $\mathcal{O}_L$, the ring of integers of $L$.
For each prime ideal $\mathfrak{p}$ of $\mathcal{O}_K$, we denote the Artin symbol at $\mathfrak{p}$ by $\sigma_\mathfrak{p}$ which determines the splitting of $\mathfrak{p} \mathcal{O}_L$ in the larger ring $\mathcal{O}_L$. 
For each fixed conjugacy class $C$ of $G$, we denote 
\begin{equation*}
 \label{piCx}
\pi_C(x) = \# \{ \mathfrak{p} \subset \mathcal{O}_K \ | \ \mathfrak{p} \text{ is unramified in } L \text{ with } \Norm \mathfrak{p} \leq x \textnormal{ and } \sigma_\mathfrak{p} = C \}.
\end{equation*}
In his 1922 Ph.D. thesis, Chebotarev  \cite{tsch} proved a weighted version of the asymptotic formula
\begin{equation}\label{cdt}
\pi_C(x) \ 
 \sim \ \frac{|C|}{|G|} \operatorname{Li}(x)  \textnormal{ as } x \rightarrow \infty. 
\end{equation} 
where $\operatorname{Li}(x)=\int_{2}^{x} \frac{dt}{\log t}$. 
This essentially asserts that, for any given normal extension of number fields, the prime ideals in the upper field are equidistributed among the conjugacy classes of Galois group of this extension. 

The trivial case $L=K=\mathbb{Q}$, \eqref{cdt} corresponds to  the prime number theorem.
In the cases where $L$  is a cyclotomic extension of $K=\Q$, \eqref{cdt} reduces to the prime number theorem in arithmetic progressions and includes Dirichlet's theorem as a special case. 

We recall weighted prime counting function 
\begin{equation*}\label{def-PsiC}
\psi_C(x) = \sum_{\overset{\mathfrak{p} \ \textnormal{unramified}}{\Norm \mathfrak{p}^m \leq x,\ \sigma_\mathfrak{p}^m = C}} \log (\Norm \mathfrak{p} ).
\end{equation*}
The Chebotarev Density Theorem is equivalent to the statement 
\begin{equation}\label{CDT-PsiC}
 \psi_C(x) \sim \frac{|C|}{|G|} x \ \textnormal{ as } x \rightarrow \infty.
\end{equation}

This article concerns an effective version of Chebotarev's density theorem. Namely, a version of \eqref{CDT-PsiC} with an error term which depends on $x$ and the field invariants of $L$ and $K$. In future work \cite{DKNpiC}, we will give the error term for the \eqref{cdt} version. 

The first effective version of such a theorem was established by Lagarias and Odlyzko in 1977 \cite{lo}.  Just as the prime number theorem is connected to the distribution of zeros of the Riemann zeta function, Chebotarev's theorem is intimately related to the zeros of Artin $L$-functions.  However, a major obstacle in the theory is that it is unknown whether Artin $L$-functions are holomorphic or not. Some special cases are known (eg: abelian extensions) but in general, it is  a wide-standing open problem. Due to this gap of knowledge, effective versions of Chebotarev's densisty theorem are related to the zeros of Dedekind zeta functions of the upper field $L$. 

We now introduce some important field invariants. 
For a number field $L$, $d_L$ is the absolute value of the discriminant of $L$  and  $n_L= [L : \mathbb{Q}]$ is the degree of extension. 
We introduce the notation
\begin{equation}\label{def-DeltaL}
\Delta_L  = d_L^{1/n_L}.
\end{equation}
We recall that $n_L$ and $d_L$ verify the Minkowski bound: 
\begin{equation}\label{bnd-Minkovski}
   \frac{n_L}{\log d_L} \leq \mathscr{M} ,
\end{equation}
for some absolute positive constant $\mathscr{M} $.
The results that follow will depend on these quantities. Throughout this article we assume $L\neq \Q$, so $n_L\geq 2$. 
Associated to $L$ is its Dedekind zeta function $\zeta_L(s)$ (see  \cite[Chapter 7, Section 5]{jn} for its properties). 
The Generalized Riemann Hypothesis (GRH) is the assertion that within the critical strip $0 < \Re(s) < 1$,
 $\zeta_L(s)$ only vanishes on the vertical line $\Re(s) =1/2$.
Unconditionally, Stark \cite{hs} proved a first so-called ``zero-free region" for $\zeta_L(s)$, that is $\zeta_L(s)$ has at most one zero in the region 
\begin{equation*}
  \label{smallregion}
 1-\frac{1}{4\log d_L} \leq \Re(s) \leq 1 \text{ and } |\Im(s)| \leq \frac1{4\log d_L} .
 \end{equation*}
If such a zero exists, it is real and simple. For the rest of the article we will denote this possible exceptional zero by $\beta_0$. 
We denote $R_1$ and $R_2$ some absolute constants such that $\zeta_L(s)$ does not vanish in 
\begin{equation}\label{zfr1}
\Re s \geq 1 - \frac1{R_{1} n_L \log(4 \Delta_L))} \ \textnormal{ when } \ |\Im s| \leq  2, 
\end{equation}
except possibly at $\beta_0$,
and in 
\begin{equation}\label{zfr2}
\Re s \geq 1 -  \frac1{R_{2} n_L \log( |\Im s| \Delta_L )} \ \textnormal{ when } \ |\Im s| >  2.
\end{equation}
Lagarias and Odlyzko \cite[Theorem 9.2]{lo}, proved that there exist absolute effectively computable constants $a,b,c$ such that, if $d_L$ is sufficiently large, then 
\begin{equation}\label{locdt}
    \bigg| \psi_C(x) - \frac{|C|}{|G|} x \bigg| \leq \frac{|C|}{|G|} \frac{x^{\beta_0}}{\beta_0} + a x e^{ - b \sqrt{ \frac{\log x}{n_L}  }}
    \text{ for all } \log x \geq c n_L (\log d_L)^2,
\end{equation} 
where the $\beta_0$ term occurs only if $\zeta_L(s)$ has an exceptional zero $\beta_0$.  This convention shall be used throughout this article.
In the original result, $c=4$ was arbitrarily chosen, and $a,b$ were not explicitly computed. 
Subsequently, they deduced bounds via partial summation for $\pi_C(x)$ in \cite[Theorem 1.3]{lo}.
In addition, they established versions under the assumption of the Generalized Riemann Hypothesis (GRH) (see also \cite[Theorem 4]{js}). 
More recently, Greni\'e and Molteni \cite{GM19} have made this GRH-dependent version fully explicit.
The past fifteen years has seen an emergence of new results and a flurry of activity on Chebotarev's density theorem \cite{ptw}, \cite{ltz}, \cite{tz2}, the least prime in Chebotarev \cite{az}, \cite{tz}, \cite{knw},  \cite{ak2}, \cite{kapj}, and on applications of such result to algebraic number theory, including to the $\ell$-torsion of class groups \cite{ptw2}, \cite{ltz}, and to the Lang-Trotter conjecture \cite{his}. The present article complements this body of work and builds upon recent results on zeros of Dedekind zeta functions such as \cite{ak}, \cite{el}, \cite{kapj} or \cite{dglsw}. Here we establish formulae and bounds for $\psi_C(x)$, giving computable formulas for the constants $a,b,c$ in \eqref{locdt} in a fully effective way. 
In addition, our work can be considered as a non-abelian generalization of Bennett et al.'s \cite{mb} article which provides explicit bounds for primes in arithmetic progressions. There is a great utility to have completely explicit results for the prime counting function $\psi_C(x)$ and we expect this work will be used in other number theoretic applications.  As a follow up work, we will provide bounds for $\pi_C(x)$ in \cite{DKNpiC}.
For the rest of the article, we denote
\begin{equation}
\label{def-EC}
E_C(x) = \frac{ \left| \psi_C(x) - \frac{|C|}{|G|} x\right|}{\frac{|C|}{|G|} x}.
\end{equation}
\begin{theorem}
\label{main-thm-psi}
Let $n_L \geq 2$. Let $\Delta_L$ and $\mathscr{M}$ be defined as in \eqref{def-DeltaL} and \eqref{bnd-Minkovski}. 
Let $\zeta_L(s)$ be the associated Dedekind zeta function and let $\beta_0$ be its possible exceptional real zero with respect to the zero-free regions \eqref{zfr1} and \eqref{zfr2}, with associated constants $R_1$ and $R_2$. 
Let $t_0\geq 1$ be chosen as in \lmaref{mlu}.
We define 
\begin{equation*}
    \label{Mthm-cond-alpha-M}
\alpha  
= \max \left(  4\frac{ R_1^2}{R_2}  \left( (\log 4)\mathscr{M} + 1 \right)^2 , 4  R_2  \left( (\log t_0)\mathscr{M} +1\right)^2 \right). 
\end{equation*}
For all $x$ satisfying 
\begin{equation}\label{Mthm-Strong-logxcondition}
\log x 
\geq \alpha n_L (\log \Delta_L )^2, 
\end{equation}
we have
\begin{equation*}\label{Mthm-main_error}
E_C(x)  
  \leq  \frac{x^{\beta_0 -1}}{\beta_0} + \mathscr{E}_L(\beta_0, x),
\end{equation*}
where 
\begin{equation}\label{Mthm-def-eps-L-m-x}
\mathscr{E}_L(\beta_0, x) = 
\max \left(\mathscr{E}_1(\beta_0)   , \mathscr{E}_2(\beta_0)  \right) 
\sqrt{\lambda_L} \sqrt{ n_L } \,
\sqrt{ \log x } \,
e^{-\frac{1}{\sqrt{R_2}}  \sqrt{\frac{\log x}{n_L}} },
\end{equation}
where $\mathscr{E}_1(\beta_0) ,\mathscr{E}_2(\beta_0) $ are defined in \eqref{def-eps1-L-m-x}, \eqref{def-eps2-L-m-x} respectively, and where 
\begin{equation}
    \label{Thm3-def-lambdaL}
\lambda_L  
= \max \left(  (\log \Delta_L)^2  n_L^2 ,  (\log \Delta_L) \Delta_L  \sqrt{ n_L} \right).
\end{equation}
\newline
In addition, if 
\begin{equation}
    \label{cond-nL}
n_L  \leq 
\mathcal{N}_0,
\end{equation}
where $\mathcal{N}_0$ is defined in \eqref{def-curvyN0}, 
then we can replace \eqref{Mthm-def-eps-L-m-x} with 
\begin{equation}
\label{def2-eps-L-m-x}
\mathscr{E}_L(\beta_0,x) =
\mathscr{E}_3(\beta_0) \lambda_L^{\frac{1}{2}} (\log x)^{\frac{1}{2}}
e^{-\frac{1}{\sqrt{R_2}} \sqrt{\frac{\log x}{n_L}} },  
\end{equation}
where $ \mathscr{E}_3(\beta_0) $ is defined in \eqref{def-epsilon3}.
\end{theorem}

Explicit numerical values for $\max (\mathscr{E}_1,\mathscr{E}_2)$ and $\mathscr{E}_3$ may be found in  Table \ref{beta0-present-cor1.1}  in Appendix \ref{appendixtables} (depending respectively whether $\beta_0$ exists or not). 
A more detailed version of \thmref{main-thm-psi} is given in \thmref{thm-psi} and the associated proof can be found in Section \ref{Section73}.

\begin{remark}
The condition \eqref{Mthm-Strong-logxcondition} can be rewritten as $\log x \geq \alpha \frac{(\log d_L )^2}{n_L}$. 
Thus, we extend the admissible range for $\log x$ by a factor of $n_L^2$, now requiring only $\log x \geq C n_L^{-1} (\log d_L)^2$ for some computable constant $C$. 
This general result enlarges the range from $\log x \gg n_L (\log d_L)^2$ given by Lagarias and Odlyzko \cite{lo}, and by Winckler \cite{bw} to  $\log x \gg n_L^{-1} (\log d_L)^2$.   This strengthens and makes explicit results due to V.K. Murty \cite{vkm}.
In addition, we improve Winckler's constants in front of the error term by a factor of $10^{-14}$, as the following corollaries will clarify. 
Note that work of  Thorner and Zaman
 \cite{tz2}  extends 
 the range in \eqref{locdt} to $\log x \gg (\log d_L)(\log \log d_L) +n_L (\log n_L)^2 $, by making use of zero-density results from Weiss \cite{weiss}. It would be a significant undertaking to make their work completely explicit. Though this remains an interesting avenue for future investigations. 

Finally, we present a refined error term for small degrees $n_L$. Specifically, for degrees up to $\mathcal{N}_0$, where we demonstrate in the following corollaries that $\mathcal{N}_0=519$ is admissible, we are able to further improve the bound by effectively reducing the field coefficient by a factor of $\frac12$ in the exponent. To the best of our knowledge, this refinement is original.
\end{remark}
As a numerical corollary to Theorem \ref{main-thm-psi} and Table \ref{beta0-present-cor1.1}, we prove in Section \ref{proof-cor1.2} that:
\begin{corollary}\label{largelogxcase1}
For all $n_L \geq 2$ and all $x$ satisfying 
\begin{equation}
    \label{cor11psidbda}
\log x 
\geq 2915\frac{(\log d_L )^2}{n_L}, 
 \end{equation}
we have
\begin{equation}
  \label{cor11psidbdb}
E_C(x)  
  \leq  \frac{x^{\beta_0 -1}}{\beta_0} +  2.714 \cdot10^{-1} \, \lambda_L  \sqrt{n_L } \, \sqrt{ \log x} \, e^{- 0.285\sqrt{\frac{\log x}{n_L}} } ,
 \end{equation}
where $\lambda_L$ is defined in \eqref{Thm3-def-lambdaL}.
In addition, if $ 2 \leq n_L  \leq 519$, 
then 
\begin{equation}
  \label{cor11psidbdc}
E_C(x)  
  \leq  \frac{x^{\beta_0 -1}}{\beta_0} + 4.452 \cdot10^{-1} \,  \sqrt{ \lambda_L }  \, \sqrt{ \log x } \,
e^{- 0.285\sqrt{\frac{\log x}{n_L}} }.
\end{equation}
\end{corollary}
In \thmref{form_logx^k} (see Section \ref{section-bnd-psi-log} and the associated Table~\ref{beta0-present-logx-cor1.2}), we express the error term for $\psi_C(x)$ in the form $O(\frac{1}{\log x})$ where the implied constant is effectively computable in terms of the field invariants.
We now present a numerical version of this result. (See the end of Section \ref{section-bnd-psi-log} for the proof).
\begin{corollary}\label{form_logx}
For all $n_L \geq 2$ and all $x$ satisfying 
$$        \log x \geq  2915\frac{(\log d_L)^2}{n_L},
$$
then 
\begin{equation}
    \label{cor12psidbdb}
E_C(x) 
    \leq 
    \frac{x^{\beta_0 -1}}{\beta_0} + \frac{1.475 \lambda_L n_L^2}{\log x} ,
\end{equation}
where $\lambda_L$ is defined in \eqref{Thm3-def-lambdaL}.
In addition, if $ 2 \leq n_L  \leq 519$, 
then 
\begin{equation}
    \label{cor12psidbdc}
E_C(x)  
    \leq 
    \frac{x^{\beta_0 -1}}{\beta_0} + \frac{2.419 \sqrt{ \lambda_L }n_L^{3/2}}{\log x}  .
\end{equation}
\end{corollary}
In \thmref{mainthm2} (see Section \ref{Section75} and the associated Table \ref{beta0-present-no-dL-cor1.2}), we express the coefficient only in terms of the degree $n_L$. Here is a numerical version of this result. (See the proof in Section \ref{proofCor13}). 
\begin{corollary}\label{largelogxcase2}
For all $n_L \geq 2$ and all $x$ satisfying 
$$
\log x 
\geq 2915\frac{(\log d_L )^2}{n_L}, 
$$
we have
\begin{equation}
  \label{cor13psidbda}
E_C(x)  
  \leq  \frac{x^{\beta_0 -1}}{\beta_0} + 1.952 \cdot10^{-3}  \,  n_L^{ 2}  \,  (\log x)  \, e^{- 0.267 \sqrt{\frac{\log x}{n_L}} } .
\end{equation}
In addition, if $ 2\leq n_L  \leq 519,$ 
then 
\begin{equation}
   \label{cor13psidbdb}
E_C(x)  
  \leq  \frac{x^{\beta_0 -1}}{\beta_0} + 3.674 \cdot 10^{-2}  \,  n_L^{\frac{3}{4} } \, (\log x)^{\frac{3}{4}}  \,  e^{- 0.258 \sqrt{\frac{\log x}{n_L}} } .
\end{equation}
\end{corollary}
Finally, in \corref{corpi-main2-gen} (see Section \ref{cor35.1} and associated Tables \ref{beta0-present-cor2.4-for-nL>=n0} and \ref{beta0-present-cor2.4}), we establish a result in which the error term takes the form \eqref{locdt}, as originally stated in Lagarias and Odlyzko's \cite[Theorem 9.2]{lo}.
We state here a simplified version (see proof in Section \ref{proofCor14}).
\begin{corollary}\label{corpsi-main2}
For all $n_L \geq 2$ and all $x$ satisfying 
$$
\log x \geq 729 n_L (\log d_L )^2 , $$
we have
\begin{equation}
  \label{cor14psidbda}
E_C(x)  \leq  \frac{x^{\beta_0 -1}}{\beta_0} + 175   \, e^{- 0.23  \sqrt{\frac{\log x}{n_L}} } .
\end{equation}
In addition, if $ 2\leq n_L  \leq 519, $
then 
\begin{equation}
  \label{cor14psidbdb}
E_C(x)  
  \leq  \frac{x^{\beta_0 -1}}{\beta_0} + 18458  \, e^{- 0.25 \sqrt{\frac{\log x}{n_L}} } .
\end{equation}
\end{corollary}
\begin{remark}
An unpublished explicit version was given by Winckler in \cite[Theorem 8.2]{bw}: 
he proved \corref{corpsi-main2} with a bound of 
\begin{equation}\label{compare1}
1.51 \cdot  10^{12} \, e^{- 0.014 \sqrt{\frac{\log x}{n_L}}},\ \text{for all}\  \log x \geq 1\,545 n_L (\log d_L)^2.
\end{equation}
In comparison, for a similar range for $\log x$ and constant in the exponential, the first author reworked his arguments to improve the constant factor in \cite[Theorem 1.12]{sd}:
$$ 5.805 \cdot  10^{3} e^{- 0.014 \sqrt{\frac{\log x}{n_L}}},\ \text{for all}\  \log x \geq 1\,845 n_L (\log d_L)^2.$$
Here, \corref{corpsi-main2} improves on all absolute constants, giving a wider range for $(\log x)$, 
and our approach yields a significantly smaller negative
exponential term than that of \cite{bw}, thereby improving the overall error bound. Our approach gives
 an exponent factor closer to $\frac1{\sqrt{R_2}}$ (here at value $0.285$), whereas
 Winckler's proof used a zero-free region giving $R_2 \simeq 306$ and $\frac1{\sqrt{R_2}}=0.057\ldots$.
 Consequently, the error term in our version of the Chebotarev Density Theorem is comparable to that in the classical cases of primes and primes in arithmetic progressions (see \cite{mb} and its statement in \eqref{psixqabd} below).
The order of magnitude of our theorem in comparison with what Lagarias-Odlyzko/Winckler's method suggests a factor of $10^{-14}$ as we can see with comparing \eqref{compare1} with the following other numerical version of \corref{corpsi-main2}:
\begin{equation*}\label{compare2}
1.210 \cdot  10^{-2}  e^{- 0.014 \sqrt{\frac{\log x}{n_L}}},\ \text{for all}\  \log x \geq 729 n_L (\log d_L)^2.
\end{equation*}
\end{remark}

\begin{remark}
Note that in Corollary \ref{largelogxcase2} the constants in the error term involve field invariants $n_L^2$, $n_L^{\frac{3}{4}}$.  However, when the range of $\log x$ is increased to the size $c n_L (\log d_L)^2$ for $c >0$ as in Corollary \ref{corpsi-main2}, the field invariants are absorbed  into the exponentially decreasing term, allowing them to be bounded by absolute constants. This explains that the constants in the exponentials change from $0.267$ and $0.258$ (see \eqref{cor13psidbda} and \eqref{cor13psidbdb}) and from $0.230$ and $0.250$ (see \eqref{cor14psidbda} and \eqref{cor14psidbdb}). 
\end{remark}

\begin{remark}
Note that all values given in Corollaries \ref{largelogxcase1} - \ref{corpsi-main2}, the constants given for $\mathcal{N}_0$ and in the error term can all be improved in the case where there is no exceptional zeros.
Namely, we can replace $519$ with $654$ for the $\mathcal{N}_0$ values, and the other error term constants $( 2.714\cdot10^{-1}, 4.452\cdot10^{-1}, 1.475 , 2.419, 1.952\cdot10^{-3}, 3.674\cdot10^{-2}, 175, 18458)$ with $( 1.9203\cdot10^{-1}, 3.1501\cdot10^{-1}, 1.0435, 1.740, 1.382\cdot10^{-3}, 2.600\cdot10^{-2}, 124, 13052)$.
\end{remark}

\subsection{A history of effective results}
The theorems and corollaries above provide bounds of a similar nature to the current best explicit estimates for the error terms in the prime counting functions
$$
\psi(x) =\sum_{n \leq x} \Lambda(n)
\text{ and }
\psi(x;q,a) = \sum_{\substack{n \leq x \\ n \equiv a \bmod{q}}} \Lambda(n)
$$
where $\Lambda(n)$ is the von-Mangoldt function.
Early explicit bounds for $\psi(x)$ include work of Rosser and Schoenfeld (see \cite{RS62}) 
and the latest include Fiori, Kadiri, Swindisky \cite{fks} who prove:
\begin{equation}
  \label{psibd}
  \frac{| \psi(x) - x|}{x} \leq 9.23  \big(  \log x  \big)^{\frac{3}{2}}\,
e^{ -2 \sqrt{ \frac{\log x}{R_0} }} 
\text{ for } x \geq 2 
\end{equation}
where $R_0= 5.5666305$ is the constant appearing in the classical form of the zero-free region for the Riemann zeta function \cite{MTY}: $\Re s \geq 1 -  \frac1{R_0 \log |\Im s|} \ \textnormal{ when } \ |\Im s| >  2$. 
Note that $2/\sqrt{R_0}= 0.8476\ldots$. In addition, Johnston and Yang \cite{JY} established a Korobov-Vinogradov form that becomes sharper for large values of $\log x$.  

Regarding primes in arithmetic progressions, Bennett et al. \cite{mb} proved that, for $q \geq  10^5$ and $\log x \geq  4 R_0' (\log q)^2$,
\begin{equation} 
 \label{psixqabd}
\frac{ |\psi(x;a,q) -  \frac{x}{\phi(q)}|}{\frac{x}{\phi(q)}} \leq 1.012 x^{\beta_0-1} + 1.4579 \ \phi(q)
      \sqrt{ \frac{\log x}{R_0'} }  \, e^{ - \sqrt{ \frac{\log x}{R_0'} }}  .
 \end{equation} 
where the first term on the right-hand side is present only if some Dirichlet $L$-function (mod $q$) has an exceptional zero $\beta_0$, which is real, with respect to the zero free region $\Re s \geq 1 -  \frac1{R_0' q\max(\log |\Im s|,1) }$. Here $R_0' = 9.645908801$ is given by \cite{km2}, and $1/\sqrt{R_0'} = 0.3219\ldots$. Note that in \eqref{psibd} there is an extra factor of $2$ which arises from the use of an explicit zero-density result for $\zeta(s)$, which is currently not available for Dirichlet $L$-functions modulo $q$.

\subsection{Remarks about the Methodology}

\begin{remark}
\textbf{Explicit formulae} relate prime counting functions to sums over the zeros of associated $ L $-functions: the Riemann zeta function $ \zeta(s) $ for $ \psi(x) $, and Dirichlet $ L $-functions for $ \psi(x; q, a) $.
The classical approach, as used by Lagarias and Odlyzko \cite{lo}, and later in \cite{bw}, relies on a truncated version of  Perron's formula and follows the traditional proof of the prime number theorem in arithmetic progressions, as presented in \cite{hd}. A more recent development by Cully-Hugill and Johnston \cite{CHJ} provides an explicit version of Perron's formula that yields improved error terms for $ \psi(x) $, achieving a savings of roughly a factor of $ \log x $ over the classical method. This makes their approach competitive with other techniques discussed here.

Another classic method consists in comparing $\psi(x)$ to its averages around $x$. Historically, tight explicit bounds were first obtained this way as in the works of Rosser and Schoenfeld \cite{ro,RS62,RS75} for $\psi(x)$. The results on primes in arithmetic progressions by McCurley \cite{km}, and more recently by Bennett, Martin, O'Bryant, and Rechnitzer \cite{mb}, follow their approach. 

The resulting explicit formula in  \cite{ro,RS62,RS75} was later reframed using a smoothing argument by Faber and the second author \cite{lh}. There, the Mellin transform was applied to derive an explicit formula in the spirit of the Guinand-Weil formulas \cite{guinand, weil}, where the choice of smoothing weight significantly affects the size of the error term given as a constant. This technique was extended to $ \psi(x; q, a) $ in preliminary work by Kadiri and Lumley \cite{al2}.
Subsequently, B\"uthe \cite{But16} used the Guinand-Weil formula with a specific Logan weight function and its Fourier transform to obtain significantly improved bounds for $ \psi(x) $ when $ \log x $ is not too large. 

Each of the above methods is best suited depending on the size of $\log x$, particularly in comparison to $(\log q)^2$ in the context of arithmetic progressions, and on the desired form of the error term. Consequently, all approaches remain highly relevant and deserve to be further explored in broader contexts of prime number theorems.

The first contribution of this article (see Section \ref{explicit-psiL}) is to use a smoothing argument following \cite{lh, al2}. We refer to Section \ref{introsmooth} for the choice of the smooth weight and how we approximate $\psi_C(x)$ by its smoothed version $\widetilde{\psi}_C(x)$.  
By applying the inverse Mellin transform, we derive an explicit formula that connects $ \psi_C(x) $ to Artin $ L $-functions.
Since the holomorphicity of the Artin $L$-functions is currently unknown, we use a classical algebraic argument due to Deuring to express $\widetilde{\psi}_C(x)$ as a sum over zeros of holomorphic Hecke $L$-functions of an intermediate field (see \propref{formulaI}) as well as a residual sum over ramified primes (see \lmaref{epstilbo}). Ultimately, \thmref{bnd-tildeECx} gives an explicit inequality for the error term $E_C(x)=\left|\psi_C(x)-\frac{|C|}{|G|}x\right|\frac{|G|}{|C|x}$. 

We note that for the particular shape of the error term considered here, our choice of smoothing weight corresponds to the one used in \cite{RS75,mb}. This approach is one of the main reasons our bounds outperform those obtained via the classical Perron's formula. 

Estimating the resulting sum over the {\bf zeros of Dedekind zeta function $\zeta_L(s)$} is central to bounding the error terms, and involves controlling three key aspects: the size of the transform inside the critical strip, the location of the real parts of the zeros, and the density of zeros in bounded regions of the strip. We briefly discuss the two later here. 
\end{remark}
\begin{remark}
 In the case of the Riemann zeta function $\zeta(s)$ and Dirichlet $L$-functions $L(s,\chi)$, {\bf partial verifications of the Generalized Riemann Hypothesis} are very useful. We refer to the work of Platt \cite{Platt18} and Platt and Trudgian \cite{PlaTruH2020}. Unfortunately, we do not have access to similar calculations for the Dedekind zeta function for extensive families of number fields $L$.
For Artin $L$-functions first partial verifications of the GRH were done by Lagarias-Odlyzko \cite{lo2}. 
 Tollis \cite{to} verified GRH for a small number of cubic, quartic, and quintic number fields
 and Booker \cite{bo} verified GRH for  a few fields of small discriminant.  
 These partial GRH verifications are not employed in this article.  
\end{remark}
\begin{remark}
{\bf Zero-free regions} play a fundamental role and those in the context of number field have only recently been developed since the early 2010's.
The first zero of $\zeta(s)$ occurs on the critical line $\Re(s) = 1/2$, with an imaginary part exceeding $14$. However, for $L(s,\chi)$, zeros may occur much closer to the real axis and can lie extremely close to $\Re(s) = 1$. It is possible to construct zero-free regions containing at most one real, simple zero, known as an exceptional zero, which can only occur for a real Dirichlet character and appears at most once among all characters modulo $q$. 
The current best-known zero-free regions for $\zeta(s)$ are found in \cite{Bellotti24, MTY, Yang24}, those for Dirichlet $L$-functions appear in \cite{km, hk3, hkphd, Khale}, and in  \cite{ak2, el, dglsw} for Dedekind zeta functions $\zeta_L(s)$.
In Section~\ref{ZFR-zetaL}, we summarize the zero-free regions that are relevant to the main results of this paper.
\end{remark}
\begin{remark}
{\bf Explicit zero-counting results} for the number of zeros inside the region $0<\Re s<1$ and $0\leq \Im s\leq T$ trace back to the work of Von Mangoldt in 1905, followed, in 1918, by important contributions from Backlund. These have inspired consecutive work such as \cite{HSW} by Hasanalizade, Shen, and Wong and \cite{bewo} by Belotti and Wong in the case of the Riemann zeta function, and \cite{mb2} by Bennett et al. for Dirichlet $L$-functions. 
For Dedekind zeta functions $\zeta_L(s)$, early bounds on the number of zeros $N_L(T)$ were provided in \cite{kn, tt}. In this work, we utilize the latest bounds due to Hasanalizade, Shen, and Wong \cite{HSW2}. Details are provided in Section~\ref{counting-zeros-zetaL}, where we also present an improved bound for $N_L(1)$ (see Theorem~\ref{thm-bnd-nchiNL}) and a new result estimating the number of non-trivial zeros of Hecke $L$-functions within small circles, generalizing work of Fiorilli–Martin \cite{mf}.

We anticipate that future {\bf explicit zero density-estimates}, especially bounds on the number of zeros in a narrower region $\sigma < \Re(s) < 1$ and $0\leq |\Im s| \leq T$, will lead to further refinements of our results. 
For explicit bounds on $N(\sigma, T)$, one can refer to \cite{kln, b, s}, and to \cite{FKS3} for applications to estimates on $\psi(x)$. 
\end{remark}
\begin{remark}
We split the resulting {\bf sum over the zeros of $\zeta_L(s)$} into four distinct regions, as defined in equations \eqref{def-J3}, \eqref{def-J4}, \eqref{def-J5}, and \eqref{def-J6}, and we bound each contribution separately. This approach differs from that of Lagarias and Odlyzko \cite{lo}, who considered only three regions. In particular, we introduce an additional region to more precisely estimate the contribution from zeros lying close to the real axis. Since the Generalized Riemann Hypothesis (GRH) has not been verified for number fields, we are more cautious when handling the sums involving low-lying zeros of Hecke $L$-functions. 
To bound the contribution from the non-trivial zeros of $\zeta_L(s)$ with large imaginary parts, we encounter certain incomplete Bessel-type integrals, as defined in \eqref{inm} and \eqref{def-K_nu}. These integrals were originally introduced by Rosser and Schoenfeld \cite{RS62}, and were more recently used in \cite[Prop. 4.7, p. 461]{mb} to obtain bounds for $\psi(x; q, a)$. 
\end{remark}
\begin{remark}
{\bf Choice of parameters and final bounds:}
In Sections \ref{Section71} to \ref{Section73}, our aim is to obtain an error term of the form $\mathcal{O}_L(x\exp(-c \sqrt{ \tfrac{\log x}{n_L}}))$ (see \propref{prop-bnd-error-delta-L-m-T-x}). By selecting an appropriate cutoff parameter $T$, as defined in \eqref{chooseT}, we derive this error term explicitly in Proposition~\ref{prop-bnd-error-delta-L-m-x0-x}. We then determine the final values of the parameters defining our smooth weight function, most notably the parameter $\delta$, which governs the length of the support, in a way that optimally balances the contribution of each term. This yields the bounds stated in Theorem~\ref{thm-psi}, presented in a simplified form in Theorem~\ref{main-thm-psi}, along with Corollary~\ref{largelogxcase1}, and numerical computations detailed in Table~\ref{beta0-present-cor1.1}.
The analysis of the optimal choice of $\delta$ also highlights certain small-degree cases in which sharper bounds can be achieved. Our selection of the parameter $T$ differs from previous approaches by Lagarias–Odlyzko \cite{lo} and Winckler \cite{bw}, yielding a more refined final error term. A key complication arises from the dependence of various terms on field invariants associated with $L$. We note that the calculation is complicated by the many terms being estimated and the dependence on field invariants.  
In Section~\ref{section-bnd-psi-log}, we complement our main theorem with estimates of the shape $1/(\log x)^k$ (see Theorem~\ref{form_logx^k}, Corollary~\ref{form_logx}, and Table~\ref{beta0-present-logx-cor1.2}. 
Finally, in Section~\ref{Section75}, we deduce Theorem~\ref{mainthm2}, which provides an alternative error term formulation directly comparable to \cite{lo,bw} (see also Corollaries~\ref{largelogxcase2} and~\ref{corpsi-main2}, and Table~\ref{beta0-present-no-dL-cor1.2}.
\end{remark}
\begin{remark}
On the {\bf computational} side, we take advantage of Fiori's results on minimal discriminants \cite{kapj}, which allow us to improve upon the classical Minkowski bound, namely $\frac{n_L}{\log d_L} \leq \frac{2}{\log 3}$, by up to a factor of $4$ for most degrees $n_L$.
\end{remark}

\section*{Acknowledgements}
Sourabhashis Das was supported by a University of Waterloo graduate fellowship. 
Habiba Kadiri was supported by the NSERC discovery grants RGPIN-2020-06731 and the Pacific Institute for Mathematical Sciences Europe Fellowship 2023.
Nathan Ng was supported by the NSERC discovery grant RGPIN-2020-06032. 
The research environment also benefited from the PIMS Collaborative Research Group $2022-2025$ {\it $L$-functions in Analytic Number Theory}. 
The authors gratefully acknowledge valuable discussions with Andrew Fiori, Allysa Lumley, and Asif Zaman.

\section{Notation and parameters} 
\label{notation}
\subsection{Notation for field invariants}
We write $L/E/K$ to denote a tower of normal extensions of number fields.
We recall the quantities depending on the field $L$:
$n_L=[L:\Q]$ its degree, $d_L$ its absolute discriminant, and 
$$
 \Delta_L   =  d_L^{1/n_L}.
$$
Let $\chi$ be a character of $\text{Gal}(L/E)$.
We denote
\begin{equation*}
\sum_{\chi} c_{\chi}  =  \sum_{\chi \in \text{Irr(Gal}(L/E))} c_{\chi} \label{fact5}
\end{equation*}
where $c_\chi$ is some real number.
We use the same indexation for $\bigcup_{\chi}$.
We denote 
\begin{equation}
 \label{deltachi}
  \mathbb{1}(\chi) 
    = \begin{cases}
    1 & \text{ if } \chi  \text{ is principal}, \\
    0 & \text{ otherwise}.
    \end{cases}
\end{equation}
Thus
\begin{equation}
 \sum_\chi \mathbb{1}(\chi)  =1 \label{fact3}.
\end{equation}
Let 
$F(\chi)$ the Artin conductor of $\chi$, 
and 
\begin{equation}
 \label{Achi}
    A(\chi) = d_E \Norm_{E/\mathbb{Q}} (F(\chi)) ,
\end{equation}
where $\Norm_{E/\Q}$ denotes the absolute norm of an ideal in the ring of integers $\mathcal{O}_E$.
We have
\begin{equation}
 \sum_{\chi }  \log (A(\chi))  = \log d_L, \label{fact1}
\end{equation}
and
\begin{equation}
 \sum_\chi n_E  = n_L\label{fact2} .
\end{equation}

\subsection{Explicit Minkowski bound}
We shall recall Minkowski bound \eqref{bnd-Minkovski}: there exists an absolute positive constant $\mathscr{M}$ such that, for all number fields $L$, we have
$$\frac{n_L}{\log d_L} \leq \mathscr{M},
$$
and it holds for $\mathscr{M}=\frac{2}{\log 3} = 1.8204 \ldots  $ (see \cite[Page 1421]{ak2}). 
We refine this bound:
\begin{equation}\label{def-Minkowski}
\text{for each }\ n_L\geq n_0\geq 2, \ \text{ there exists }\ d_0 \geq 3 \ \text{ such that }\ d_L\geq d_0,\ \text{ and }\ \log \Delta_L \geq \frac1{\mathscr{M} }.
\end{equation}
A table of values of $(n_0, d_0, \mathscr{M})$ is given in Table \ref{n0d0} and was compiled by Andrew Fiori in \cite[Appendix]{kapj}. 
\begin{table}[h!]
\centering
\caption{Table of $n_0$, $d_0$, $\mathscr{M}$ as given in \eqref{def-Minkowski}}
\label{n0d0}
\begin{tabular}{| l | l | l || l | l | l |}
\hline
$n_0$ & $d_0$ & $\mathscr{M}$ & $n_0$ & $d_0$ & $\mathscr{M}$\\
\hline
2&3 &1.82048 & 12& $2.74\cdot10^{10}$ &0.499297\cr
3&23 &0.956787 & 13& $7.56\cdot10^{11}$ &0.475297\cr
4&117 &0.839953 & 14& $5.43\cdot10^{12}$ &0.477442\cr
5&1\,609 &0.677198 & 15& $1.61\cdot10^{14}$ &0.458541\cr
6&9\,747 &0.653259 & 16& $1.17\cdot10^{15}$ &0.461151\cr
7&184\,607 &0.577273 & 17& $3.70\cdot10^{16}$ &0.445613\cr
8&1\,257\,728&0.569605 & 18& $2.73\cdot10^{17}$ &0.448338\cr
9& $2.29\cdot10^7$ &0.531078 & 19& $9.03\cdot10^{18}$ &0.435310\cr
10&$1.56 \cdot10^8$ &0.530072 & 20& $6.74\cdot10^{19}$ &0.438047\cr
11& $3.91 \cdot10^9$ &0.498035 & $21^+$ & $10^{n_0}$ & 0.434294\cr
\hline
\end{tabular}
\end{table}
\ \newline
Note that, for $n_L\geq n_0\geq 21$, $d_L\geq 10^{n_0}$, and  $\mathscr{M} =\frac1{\log10} $. 
These bounds shall be used a number of times throughout this article.
\subsection{Notation for $L$-functions}
Let $\chi$ be a character of $\text{Gal}(L/E)$.  
We denote the associated Artin $L$-function $L(s,\chi,L/E)$, 
$Z(\chi)$ the set of its non-trivial zeros, and $Z(\zeta_L)$ the set of zeros of the Dedekind zeta function $\zeta_L(s)$.
We have
\begin{equation}    \bigcup_{\chi } Z(\chi)  = Z(\zeta_L). \label{fact4} \end{equation}
Throughout this article, non-trivial zeros of $\zeta_L(s)$ and $L(s,\chi, L/E)$ shall be denoted by $\varrho$.  We shall use the convention
$\varrho=\beta+i \gamma$ with $\beta \in (0,1)$ and $\gamma \in \R$, as well as $\beta_0$ the potential real exceptional zero of $\zeta_L(s)$. 
We introduce $a_{\beta_0}$ as the following:
 \begin{equation}\label{abeta0}
     a_{\beta_0} = \begin{cases}
     1 & \textnormal{ if } \beta_0 \text{ exists}, \\
     2 & \text{ otherwise}. 
     \end{cases}
 \end{equation}
\subsection{Introducing a smooth weight}\label{introsmooth}
In this section, we introduce a family of smooth functions that will be used to smooth the prime sum in Chebotarev's theorem. 
Let $0 < \delta < 1, \alpha  \in \{ 1 - \delta,  1 \}$ and $m \in \Norm$. We define a function $h$, depending on $\delta$ and $m$, on $[0,\infty)$ by 
\begin{equation}\label{defh}
h(x) = 
\begin{cases}
1 & \textnormal{if } 0 \leq x \leq \alpha,\\
g ( \frac{x- \alpha}{\delta} ) & \textnormal{if } \alpha \leq x \leq \alpha + \delta,\\
0 & \textnormal{if } x \geq  \alpha + \delta,
\end{cases}
\end{equation}
where $g$ is a function defined on $[0,1]$ satisfying
\begin{enumerate}
    \item (Condition 1) $0 \leq g(x) \leq 1$  for  $0 \leq x \leq 1$,
    \item (Condition 2) $g$ is an $m$-times differentiable function on $(0,1)$ such that for all $k = 1,...,m$,
    $$g^{(k)} (0) = g^{(k)}(1) = 0,$$
    and there exist positive constants $a_k$ such that
    $$|g^{(k)}(x)| \leq a_k \ \textrm{  for all  } \  0 < x< 1.$$
\end{enumerate}
The Mellin transform of $h$ is given by
\begin{equation}\label{def-Mellin}
     H(s) = \int_0^\infty h(t) t^{s-1} dt.
\end{equation}
If $H$ is analytic in $\Re(s) > 0$, the inverse Mellin transform formula is given by
$$h(t) \ = \ \frac{1}{2 \pi i} \int_{2- i \infty}^{2+ i \infty} H(s) t^{-s} ds.$$
For any non-negative integer $m$, we define $M(\delta,m)$ to be a function satisfying
$$\max_{\alpha \in \{ 1- \delta,  1 \}} \int_\alpha^{\alpha + \delta} | h^{(m+1)} (t) | t^{m+1} dt \leq \frac{1}{\delta^m} M(\delta,m).$$
After choosing our smooth weight, we will provide an expression for $M(\delta,m)$ at the end of this section.
\begin{lemma}
Let $0 < \delta < 1$, $\alpha  \in \{ 1 - \delta,  1 \}$, $m \in \Norm$ and $m \geq  1$. We obtain: 
\begin{enumerate}[label=(\alph*)]
    \item The Mellin transform $H$ of $h$ has a single pole at $s = 0$ with residue $1$ and is analytic everywhere else.
    \item Let $ s \in \mathbb{C}$ such that $0 < \Re (s) \leq 1$. Then $H$ satisfies
    \begin{align}\label{def-H(1)}
        H(1) & = \alpha + \delta \int_0^1 g(u) du, \\
        \label{def-H(s)} 
        | H(s) | & \leq \frac{M(\delta,k)}{\delta^k |s|^{k+1}}, \  \textnormal{for all} \ k = 0,1,...,m.
    \end{align}
    \item Let $ s \in \mathbb{C}$ such that $\Re (s) \leq 0$. Then
    \begin{equation}\label{bnd-H}
    |H(s)| \leq \frac{(1- \delta)^{\Re(s)}}{|s|} \int_0^1 |g'(u)| \ du. 
    \end{equation}
\end{enumerate}
\end{lemma}
\begin{proof}
In  \cite[Lemma 2.2]{lh}, the lemma is stated for $m \geq 2$, however an examination of the proof
reveals that it also holds for $m=1$. 
Note that $(a)$ and $(b)$ follow from \cite[Lemma 2.2]{lh}.
 For $(c)$, we integrate by parts \eqref{def-Mellin} to obtain
\begin{equation*}\label{Huse1}
|H(s)| = \Big| \int_0^{1} h(t) t^{s-1} \ dt \Big| = \Big| - \frac{1}{s} \int_{\alpha}^{\alpha+ \delta} h'(t) t^s \ dt \Big| \leq \frac{1}{|s|} \int_{\alpha}^{\alpha+ \delta} |h'(t)| t^{\Re(s)} \ dt.
\end{equation*} 
For $\alpha = 1$, we use $t^{\Re(s)} \leq 1$ and for $\alpha = 1 - \delta$, we use $t^{\Re(s)} \leq (1-\delta)^{\Re(s)}$. Moreover, $1 \leq (1-\delta)^{\Re(s)}$ for $\Re(s) \leq 0$. Using $h(t) = g (\tfrac{t-\alpha}{\delta})$ and making a variable change of $\tfrac{t-\alpha}{\delta} = u$, we obtain identity \eqref{bnd-H}.
\end{proof}
Next, we define our choice of smooth weight, $g$.
For each positive integer $m$ we define
\begin{equation*}\label{grosser}
        g(x) = \frac{1}{m!} \sum_{j=0}^{m} (-1)^{j+m} \binom{m}{j} \bigg( \frac{(1 + (\delta + 2(\alpha-1))j/m)- (\delta x + \alpha)}{(\delta + 2(\alpha-1))/m} \bigg)^m \mathbbm{1} \bigg( \frac{\delta x + \alpha}{1 + (\delta + 2(\alpha-1))j/m} \bigg), 
    \end{equation*}
    where $\mathbbm{1}$ is the indicator function on $(0,1)$.
    It was demonstrated in \cite[Section 3.4]{lh} that this choice of smooth function corresponds to Rosser's approach \cite{ro} of a multivariable averaging using the first mean value theorem for integrals.
 From this, we deduce
    \begin{equation*}\label{hrosser}
        h(x) = \frac{1}{m!} \sum_{j=0}^{m} (-1)^{j+m} \binom{m}{j} \bigg( \frac{(1 + (\delta + 2(\alpha-1))j/m)- x}{(\delta + 2(\alpha-1))/m} \bigg)^m \mathbbm{1} \bigg( \frac{x}{1 + (\delta + 2(\alpha-1))j/m} \bigg).
    \end{equation*}
It can be shown that 
    \begin{equation}\label{Hrosser}
        H(s) = \frac{\sum_{j=0}^m (-1)^{j+m+1} \binom{m}{j} (1 + (\delta + 2(\alpha-1))j/m)^{m+s}}{(\delta/m)^m s (s+1) \cdot \cdot \cdot (s +m)},
    \end{equation}    
and
    $$
    \int_0^{1} g(x) \ dx = \frac{1}{2}, \ \int_\alpha^{\alpha + \delta} h(x) \ dx = \frac{\delta}{2}, \ \int_0^1 |g'(x)| \ dx = 1.$$
Further, it follows from  \cite[Theorem 15]{ro} that the Mellin transform given in \eqref{Hrosser} satisfies \eqref{def-H(s)} with
    \begin{equation}\label{Mrosser}
        M(\delta,m) = \begin{cases} 
        (m((1+\delta/m)^{m+1}+1))^m & \textnormal{ for } m \geq  1, \textnormal{ and } \\
        1+\frac{\delta}{2} & \textnormal{ for } m =0.
        \end{cases}
    \end{equation}

\begin{remark}
Faber and Kadiri in \cite{lh} propose an alternative smooth weight function. However, we observe that our choice of $g$ yields a better bound for $M(\delta, m)$ than theirs for $m \leq 5$. As we will show later, the optimal value of $m$ for the form of the error term we are establishing is $m = 1$.
\end{remark}

\section{An explicit inequality for $\psi_C(x)$}
\label{explicit-psiL}

\subsection{Introducing a smoothed version of $\psi_C(x)$}
We introduce 
\begin{equation*}\label{psiCtilde}
 \widetilde{\psi}_C(x) = \sum_{\overset{\mathfrak{p} \ \textrm{unramified }}{\sigma_{\mathfrak{p}}^m = C}} \sum_{m \geq  1} (\log (\Norm \mathfrak{p})) h \Big( \frac{\Norm \mathfrak{p}^m}{x} \Big),
\end{equation*}
where $\mathfrak{p}$ runs over all the prime ideals of $K$ and $h$ is defined in \eqref{defh}. 
We denote the normalized error terms for respectively $\psi$ and $\widetilde{\psi}$ by 
\begin{equation}
\label{Eh}
    E_C (x) = \frac{\left| \psi_C(x) - \frac{|C|}{|G|} x \right|}{\frac{|C|}{|G|} x}
    \text{ and }
\tilde{E}_C(x) =  \frac{\left|\widetilde{\psi}_C(x) - \frac{|C|}{|G|} x \right|}{\frac{|C|}{|G|} x}.
\end{equation}
We define $\psi_C^-$ and $\psi_C^+$ as the sums $\widetilde{\psi}_C$ associated to the weights $h$ defined by $\alpha = 1 - \delta$ and $\alpha = 1$ respectively. We also denote $E_C ^-$ and $E_C ^+$ the respective error terms. Observe that 
\begin{equation*}\label{psi-+}
    \psi_C^-(x) \leq \psi_C(x) \leq \psi_C^+(x),
\end{equation*}
and
\begin{equation}\label{E-+}
    E_C (x) \leq \max ( E_C ^-(x), E_C ^+(x)).
\end{equation} 
We write
\begin{equation}\label{divideI}
    \widetilde{\psi}_C(x) = I_{L/K}(x) - I^{\text{ram}}_{L/K}(x),
\end{equation}
with 
\begin{equation}\label{newI}
    I_{L/K}(x) = \sum_{\mathfrak{p}} \sum_{m \geq  1} \theta(\mathfrak{p}^m) (\log (\Norm\mathfrak{p})) h \Big( \frac{\Norm\mathfrak{p}^m}{x} \Big),
\end{equation}
and 
\begin{equation*}\label{Itilde}
 I^{\text{ram}}_{L/K}(x) = \sum_{\mathfrak{p} \ \textrm{ramified}} \sum_{m \geq  1} \theta(\mathfrak{p}^m) (\log (\Norm \mathfrak{p})) h \Big( \frac{\Norm\mathfrak{p}^m}{x} \Big),
\end{equation*}
where $\theta$ is the indicator function characterizing the Artin symbol at $\mathfrak{p}$ coinciding with the conjugacy class $C$. More specifically, for prime ideals of $\mathcal{O}_K$, $\mathfrak{p}$ unramified in $L$, we have
\begin{equation}\label{deftheta}
\theta(\mathfrak{p}^m) = \begin{cases}
1 & \textnormal{ if } \sigma_\mathfrak{p}^m = C, \\
0 & \textnormal{ otherwise},
\end{cases}
\end{equation}
and $|\theta(\mathfrak{p}^m)| \leq 1$ if $\mathfrak{p}$ ramifies in $L$.
It follows from \eqref{Eh}, \eqref{E-+}, and \eqref{divideI} that 
\begin{equation}\label{use1psiC}
E_{\widetilde{\psi}}(x) 
    \leq  \frac{|I_{L/K}(x) - \frac{|C|}{|G|} x|}{\frac{|C|}{|G|} x} + \frac{|I^{\text{ram}}_{L/K}(x)|}{\frac{|C|}{|G|} x}.
\end{equation}
\subsection{Controlling the smoothed sum over ramified prime ideals}
\begin{lemma}\label{epstilbo}
Let $C$ be a fixed conjugacy class of $G = \operatorname{Gal}(L/K)$. For $x \geq  x_0 \geq  2$, $0 \leq \delta \leq \delta_0 < 1$, we have
$$ \frac{|I^{\text{ram}}_{L/K}(x)|}{\frac{|C|}{|G|} x}  \leq \ell_0 (\log d_L) \frac{\log x}{x}$$ 
where
\begin{equation}\label{def-el0}
\ell_0 = \ell_0(\delta_0,x_0) = \frac{2}{\log 2} \Big( 1 + \frac{\log (1+\delta_0)}{\log x_0} \Big).
\end{equation}
\end{lemma}
\begin{proof}
By definition of $\theta$ in \eqref{deftheta} and $h$ in \eqref{defh}, we obtain
\begin{align}\label{use101}
|I^{\text{ram}}_{L/K}(x)| & \leq \sum_{\mathfrak{p} \ \textrm{ramified}} \sum_{\overset{m \geq  1}{\Norm \mathfrak{p}^m < x (\alpha +\delta)}}  (\log (\Norm\mathfrak{p})) \leq \sum_{\mathfrak{p} \textnormal{ ramified}} \log (\Norm\mathfrak{p}) \sum_{\overset{m \geq  1}{\Norm(\mathfrak{p}^m) < x(\alpha +\delta)}} 1 .
\end{align}
For each prime ideal $\mathfrak{p}$, $\Norm\mathfrak{p} \geq  2$, and thus we have  
\begin{equation}\label{use102_2}
    \sum_{\overset{m \geq  1}{\Norm(\mathfrak{p}^m) < x(\alpha +\delta)}} 1 \leq  \frac{\log (x(\alpha +\delta))}{\log 2}.
\end{equation}
Serre \cite[Proposition 5]{js} proved
\begin{equation}\label{use102}
    \sum_{\mathfrak{p} \textnormal{ ramified}} \log (\Norm\mathfrak{p}) \leq \frac{2}{|G|} \log d_L.
\end{equation}
Putting together \eqref{use101}, \eqref{use102} and \eqref{use102_2}, we obtain
\begin{equation}\label{Itildebound}
    |I^{\text{ram}}_{L/K}(x)| \leq \frac{2}{\log 2} \frac{(\log d_L) (\log ( x(\alpha + \delta)))}{|G|}.
\end{equation}
Using $\delta \leq \delta_0$, $\alpha \leq 1$ and $|C| \geq  1$ in \eqref{Itildebound}, we complete the proof.
\end{proof}
\subsection{An explicit formula for smoothed sum over all prime ideals ($I_{L/K}$)}
To obtain explicit formula for the prime counting function, $\psi_C(x)$, Lagarias and Odlyzko in \cite{lo} and Winckler in \cite{bw} use the classical method of Perron's formula. We instead use inverse Mellin transform by generalizing the approach for Riemann $\zeta$-function taken by Faber and Kadiri in \cite{lh}, and for Dirichlet $L$-functions taken by Kadiri and Lumley in \cite{al2}.
\subsubsection{Expressing \texorpdfstring{$I_{L/K}$}{} in terms of Hecke $L$-functions}
\begin{lemma}\label{Ilk}
Let $g \in C$, $G_0 = \langle g \rangle$ be the cyclic group generated by $g$, $E$ be the fixed field of $G_0$, and $\chi$ runs through the irreducible characters of $G_0$. Let $H$ be defined in \eqref{def-Mellin}. Then
$$I_{L/K}(x)  = \frac{|C|}{|G|} \sum_{\chi} \bar{\chi}(g) \Big( \frac{1}{2 \pi i} \int_{2-i \infty}^{2 + i \infty} H(s) x^s \Big( - \frac{L'}{L} (s, \chi, L/E) \Big) ds \Big).$$
\end{lemma}
\begin{proof}
Let $\phi$ be an irreducible character of $G = \text{Gal}(L/K)$. Let us define
\begin{equation*}\label{phiK}
\phi_K (\mathfrak{p}^m) = \frac{1}{|I_0|} \sum_{\alpha \in I_0} \phi (\tau^m \alpha),
\end{equation*}
where $I_0$ is the inertia group of $\mathfrak{q}$, one of the prime ideal factors of $\mathfrak{p}$, and $\tau$ is one of the Frobenius automorphism corresponding to $\mathfrak{p}$. If $L(s,\phi,L/K)$ is the Artin $L$-series associated to $\phi$, then from \cite[(3.2)]{lo}, we get that for $\Re(s) > 1$,
$$ -\frac{L'}{L} (s,\phi,L/K) = \sum_{\mathfrak{p},m} \phi_K(\mathfrak{p}^m) \log (\Norm \mathfrak{p})(\Norm \mathfrak{p})^{-m s},$$
where the sum is over all the prime ideals of $K$. Using \cite[(3.1),(3.2),(3.5),(3.6)]{lo}, $\theta$ defined in \eqref{deftheta} can be redefined in a general way as
\begin{equation}\label{gentheta}
    \theta(\mathfrak{p}^m) = \frac{|C|}{|I_0||G|} \sum_{\phi, \alpha \in I_0} \bar{\phi}(g) \phi(\tau^m \alpha).
\end{equation}
Using \eqref{gentheta} and the inverse Mellin transform of $h$ in \eqref{newI}, we obtain
\begin{equation*}\label{integral1ofI}
    I_{L/K}(x) =  \frac{|C|}{|G|} \sum_{\phi \in G} \bar{\phi}(g) \Big( \frac{1}{2 \pi i} \int_{2-i \infty}^{2 + i \infty} H(s) x^s \Big( - \frac{L'}{L} (s, \phi, L/K) \Big) ds \Big).
\end{equation*}
Deuring reduction as shown in \cite[Lemma 4.1]{lo} is the process of reduction of Artin $L$-functions, $L(s,\phi,L/K)$ to the case of Hecke $L$-functions, $L(s,\chi,L/E)$ where intermediate field extension $L/E$ has a cyclic Galois group. Hecke $L$-functions have been proven to be holomorphic on the entire complex plane whereas the same has not been proven for the Artin $L$-functions. Following Deuring reduction, we obtain, $\sum_{\chi} \bar{\chi}(g) \chi^* = \sum_{\phi} \bar{\phi}(g) \phi$ where $\chi^*$ is the character of $G$ induced by $G_0$. Also \cite[Theorem 2.3.2(d)]{nn} gives us $L(s,\chi^*,L/K) = L (s, \chi , L/E)$. Therefore,
\begin{equation*}\label{integral2ofI}
    I_{L/K}(x) = \frac{|C|}{|G|} \sum_{\chi} \bar{\chi}(g) \Big( \frac{1}{2 \pi i} \int_{2-i \infty}^{2 + i \infty} H(s) x^s \Big( - \frac{L'}{L} (s, \chi, L/E) \Big) ds \Big),
\end{equation*}
which holds for $\Re(s) >1$, and hence by analytic continuation, for all $s$. 
\end{proof}
\subsubsection{Expressing \texorpdfstring{$I_{L/K}(x)$}{} in terms of zeros of Hecke $L$-functions}
From this point on, we will abbreviate $L(s,\chi,L/E)$ to $L(s,\chi)$. 
Note that  from \cite[(5.4)-(5.7)]{lo}, we have for each $\chi$, there exist non-negative integers $a(\chi)$ and $b(\chi)$ such that $a(\chi) + b(\chi) = n_E$.
Furthermore, we define
\begin{equation}\label{def-gammachis}
\gamma_\chi(s) = \Big( \pi^{-\frac{s+1}{2}} \Gamma\Big( \frac{s+1}{2} \Big) \Big)^{b(\chi)} \Big( \pi^{-\frac{s}{2}} \Gamma\Big( \frac{s}{2} \Big) \Big)^{a(\chi)}
\end{equation}
and
\begin{equation*}\label{xis}
    \xi(s,\chi) = (s (s-1))^{\mathbb{1}(\chi)} A(\chi)^{\frac{s}{2}} \gamma_\chi(s) L(s,\chi),
\end{equation*}
where $\mathbb{1}(\chi)$ and $A(\chi)$ are given in \eqref{deltachi} and \eqref{Achi}.
Note that $\xi(s,\chi)$ satisfies the functional equation
\begin{equation*}\label{xifun}
\xi(1-s,\bar{\chi}) = W(\chi) \xi(s,\chi),
\end{equation*}
where $W(\chi)$ is the root number. In \cite[(5.9)]{lo}, it is shown that, for all complex numbers $s$,
\begin{equation}\label{formula-L'/L}
    \frac{L'}{L}(s,\chi) =  B(\chi) + \sum_{\varrho \in Z(\chi)} \Big( \frac{1}{s -\varrho} + \frac{1}{\varrho} \Big) - \mathbb{1}(\chi) \Big( \frac{1}{s} + \frac{1}{s-1} \Big) - \frac{1}{2} \log (A(\chi)) - \frac{\gamma_\chi'}{\gamma_\chi} (s),
\end{equation}
where $B(\chi)$ is some constant.  Lagarias and Odlyzko in \cite[(5.10)]{lo} showed that 
\begin{equation}\label{eqbchi}
\Re B(\chi) = - \Re \sum_{\varrho \in Z(\chi)} \frac{1}{\varrho} .
\end{equation}
Also, from \cite[7.1]{lo},  $r(\chi)$ is  defined as
\begin{equation}\label{def-r(chi)}
r(\chi) = B(\chi) - \frac{1}{2} (\log A(\chi)) + \frac{n_E}{2} (\log \pi) + \mathbb{1}(\chi) - \frac{b(\chi)}{2} \frac{\Gamma'}{\Gamma} \Big( \frac{1}{2} \Big) - \frac{a(\chi)}{2} \frac{\Gamma'}{\Gamma} (1).
\end{equation}
In order to establish an explicit formula for $I_{L/K}$ we require the following facts. 

\begin{lemma}\label{lagold}
Let $\chi$ be a character of $G_0$ and $s = \sigma + i t$.
\begin{enumerate}
    \item Let $m$ be a non-negative integer. If $\sigma \leq -1/4$ and $|s+m| \geq  1/4$, then
    \begin{equation}\label{L'/LschibigO}
        \frac{L'}{L}(s,\chi) \ll \log A(\chi) + n_E \log (|s| + 2).
    \end{equation}
    \item If $-1/2 \leq \sigma \leq 3$ and $|s| \geq  1/8$, then
    \begin{equation}\label{use3v3chi}
        \Big| \frac{L'}{L} (s,\chi) + \frac{\mathbb{1}(\chi)}{s-1}  - \sum_{\overset{\varrho \in Z(\chi)}{|\gamma -  t| \leq 1}} \frac{1}{s- \varrho} \Big| \ll \log A(\chi) + n_E \log(|t| + 2).
    \end{equation}
    \item Let $t\geq 2, x\geq 2$ and $1 < \sigma_1 \leq 3$. Let $\varrho' = \beta' + i \gamma' \in Z(\chi) $ such that  $\gamma' \neq t$. Then 
    \begin{equation}\label{use10}
       \int_{-1/4}^{\sigma_1} \frac{x^{\sigma + it}}{(\sigma +i t)(\sigma + i t - \varrho')} d \sigma \ll |t|^{-1} x^{\sigma_1} (\sigma_1 - \beta')^{-1}.
    \end{equation}
    \item  Let $a\geq 0$ and $T \geq 0$. We define 
    \begin{equation}
  \label{nchiaT}
  n_{\chi,a}(T) = \# \{  \varrho \in Z(\chi) :  \ 0 < \beta < 1 \ \text{and} \ |\gamma - T| \leq a \}.
\end{equation} 
We have
    \begin{equation}\label{def-nchit}
        n_{\chi,1}(T) \ll \log A(\chi) + n_E \log (|T| + 2).
    \end{equation}
\end{enumerate}
\end{lemma}
\begin{proof}
See \cite[Lemma 6.2, Lemma 5.6, Lemma 6.3, Lemma 5.4]{lo}.
\end{proof}

\begin{proposition}\label{formulaI}
Let $C$ be a conjugacy class of $G = \operatorname{Gal}(L/K)$, $g \in C$, $G_0 = \langle g \rangle$ be the cyclic group generated by $g$, $E$ be the fixed field of $G_0$, and $\chi$ runs through the irreducible characters of $G_0$. Let $Z(\chi)$ denote the set of non-trivial zeros of the Artin $L$-function $L(s,\chi)$. Let $a(\chi), b(\chi), \mathbb{1}(\chi), r(\chi)$ be defined as in \eqref{def-gammachis}, \eqref{deltachi}, \eqref{def-r(chi)} respectively. Let $ 0 < \delta < 1$, $\alpha  \in \{ 1 - \delta,  1 \}$, $h$ be defined in \eqref{defh} and $H$ be defined in \eqref{def-Mellin}. Let $x \geq  2$. Then
\begin{align}\label{explicitI}
    I_{L/K}(x) & =\frac{|C|}{|G|} x H(1) + \frac{|C|}{|G|} \sum_{\chi} \bar{\chi}(g) \Big( - r(\chi) - (a(\chi) -\mathbb{1}(\chi)) \Big( \log (\alpha x)+ \int_\alpha^{\alpha +\delta} \frac{h(t)}{t} dt \Big) \notag \\
    & \hspace{.5cm} - \sum_{\varrho \in Z(\chi)} x^\varrho H(\varrho) - b(\chi) \sum_{m \geq  1} x^{-(2m-1)}H(1-2m) - a(\chi) \sum_{m \geq  1} x^{-2m}H(-2m) \Big).
\end{align}
\end{proposition}
\begin{proof}
Let $x \geq  2$ and $T \geq  2$ be such that it does not equal the ordinate of any zero of any of the $L(s,\chi)$. We rewrite \eqref{integral2ofI} as
\begin{equation}\label{IwithJ}
I_{L/K}(x) = \frac{|C|}{|G|} \sum_{\chi} \bar{\chi}(g) \Big( \lim_{T \rightarrow \infty} \frac{1}{2 \pi i} \int_{2-i T}^{2 + i T} Y_\chi(s) ds \Big),
\end{equation}
where
\begin{equation*}\label{defY}
    Y_\chi(s) = H(s)x^s \Big( - \frac{L'}{L} (s,\chi) \Big).
\end{equation*}
Let $U = j + \frac{1}{2}$ for some non-negative integer $j$ and $B_{T,U}$ be the positively oriented rectangle with vertices at $2 -i T$, $2 + i T$, $-U + i T$ and $-U - i T$. We define
\begin{equation*}\label{JchixTU}
    J_\chi(x,T,U) = \frac{1}{2 \pi i} \int_{B_{T,U}} Y_\chi(s) ds.
\end{equation*}
The next step is to study the poles for $Y_\chi(s)$, followed by using Cauchy's theorem and then take the limit as $T,U \rightarrow \infty$. 
The poles of $Y_\chi(s)$ inside $B_{T,U}$ are at $s=0, 1, $ at all the trivial and non-trivial zeros of $L(s,\chi)$.
We now determine their contributions.

The Laurent series expansion of $\frac{L'}{L} (s,\chi)$ about $s=0$ as defined in \cite[Page 448, (7.1)]{lo} shows that 
$$ \frac{L'}{L}(s,\chi) = \frac{a(\chi) - \mathbb{1}(\chi)}{s} + r(\chi) + s f(s,\chi),$$
where $f(s,\chi)$ is a function that is analytic at $s = 0$. Also, \cite[(3.5.3) and (3.5.4)]{al2} establishes that $H(s) x^s$ has a simple pole at $s=0$ and that its Laurent series expansion at this point is 
\begin{equation*}\label{G'0}
H(s) x^s = \frac{1}{s} \left(1 + (\log x + G'(0))s + \mathcal{O}(s^2) \right) \textnormal{ with } G'(0) = \log \alpha + \int_\alpha^{\alpha +\delta} \frac{h(t)}{t} dt.    
\end{equation*}
Therefore the residue of $Y_\chi(s)$ at $s=0$ is 
$$- r(\chi) - (a(\chi) -\mathbb{1}(\chi)) \Big( \log( \alpha x)+ \int_\alpha^{\alpha +\delta} \frac{h(t)}{t} dt \Big).$$

We know that $H(s)x^s$ is analytic at $s=1$, as well as $L(s,\chi)$ unless $\chi = \chi_1$, the principal character, in which case $\frac{L'}{L} (s,\chi)$ has a first order pole of residue $-1$ at $s = 1$. Hence, the residue of $Y_\chi(s)$ at $s =1$ is $$\mathbb{1}(\chi) x H(1).$$

We have that $\frac{L'}{L} (s,\chi)$ has first order poles at the so-called trivial zeros, i.e., at $s = -(2m -1)$, $m = 1,2,...$ with residue $b(\chi)$, and at $s = -2m$, $m = 0,1,2,...$ with residue $a(\chi)$. Therefore, the residue of $Y_\chi(s)$ at trivial zeros with $\Re(s) < 0$ is:
$$ - b(\chi) \sum_{1 \leq m \leq \frac{U+1}{2}} x^{-(2m-1)}H(-(2m-1)) - a(\chi) \sum_{1 \leq m \leq \frac{U}{2}} x^{-2m}H(-2m).$$

Finally, $\frac{L'}{L} (s,\chi)$ has a first order pole with residues $1$ at each non-trivial zero $\varrho$ of $L(s,\chi)$ (counted with multiplicity). Also $F(s)x^s$ is analytic at such points. Hence, the residue of $Y_\chi(s)$ at $Z(\chi)$ is 
$$- \sum_{\overset{\varrho \in Z(\chi)}{|\gamma| < T}} x^\varrho H(\varrho).$$

Now using Cauchy's theorem on $J_\chi(x,T,U)$, we obtain
\begin{multline}\label{cauchyonJ}
J_\chi(x,T,U) = V_{1,\chi} + V_{2,\chi} + V_{3,\chi}  - r(\chi) - (a(\chi) -\mathbb{1}(\chi)) \Big( \log (\alpha x)+ \int_\alpha^{\alpha +\delta} \frac{h(t)}{t} dt \Big)
+ \mathbb{1}(\chi) x H(1) 
\\ - \sum_{\overset{\varrho \in Z(\chi)}{|\gamma| < T}} x^\varrho H(\varrho) - b(\chi) \sum_{1 \leq m \leq \frac{U+1}{2}} x^{-(2m-1)}H(1-2m) - a(\chi) \sum_{1 \leq m \leq \frac{U}{2}} x^{-2m}H(-2m), 
\end{multline}
where 
\begin{align*}
& 
    V_{1,\chi} = V_{1,\chi} (x,T,U) = - \frac{1}{2 \pi} \int_{-T}^T Y_\chi(-U+it)  \ dt,
\\& 
   V_{2,\chi} = V_{2,\chi} (x,T,U) = \frac{1}{2 \pi i} \int_{-U}^{-1/4} \left( Y_\chi(\sigma-i T) - Y_\chi(\sigma+i T) \right) \ d\sigma,
\\& 
   V_{3,\chi} =  V_{3,\chi} (x,T) = \frac{1}{2 \pi i} \int_{-1/4}^{2}  (Y_\chi(\sigma-i T) - Y_\chi(\sigma+i T)) \ d\sigma.
\end{align*}
We use the bounds \eqref{L'/LschibigO} for the $L$-function and \eqref{def-H(s)} for the Mellin transform $H$.
For $V_{1,\chi}$, we use 
$$  Y_{\chi}(-U \pm i t ) \ll  \begin{cases}
\frac{\log U}{U^{m+1}} x^{-U} &  \text{when}\  0< |t|<\min(U,T),\\
\frac{\log |t|}{|t|^{m+1}} x^{-U} &  \text{when}\  U< \s<T,
\end{cases}
$$ 
giving 
\begin{equation}\label{V1chi0} 
V_{1,\chi} \ll   \frac{\log U}{U^{m}} x^{-U} + \frac{\log T}{T^{m}} x^{-U} .
\end{equation}
For $V_{2,\chi}$, we use 
$$  Y_{\chi} (-\sigma \pm i T) \ll  \begin{cases}
\frac{\log T}{T^{m+1}} x^{-\sigma} &  \text{when}\  1/4<\s< \min(U, T),\\
\frac{\log \s}{\s^{m+1}} x^{-\sigma} &  \text{when}\  T< \s<U,
\end{cases}
$$ 
giving 
\begin{equation}\label{V2chi0}
 V_{2,\chi} \ll \frac{\log T}{T^{m+1}} \frac{x^{-1/4}}{\log x} + \frac{x^{-T}}{T^{m-1}}.
\end{equation}
For $V_{3,\chi}$, we use \eqref{use3v3chi}: 
since
$\displaystyle  \frac{L'}{L} (s,\chi)  - \sum_{\overset{\varrho \in Z(\chi)}{|\gamma \mp T| \leq 1}} \frac{1}{s- \varrho} \ll \log T$, 
then
\begin{equation}\label{v3chi1} 
\int_{-1/4}^{2} x^{\sigma \pm i T} H(\sigma \pm i T)  \Big(  \frac{L'}{L} (\sigma \pm i T,\chi)  - \sum_{\overset{\varrho \in Z(\chi)}{|\gamma \mp T| \leq 1}} \frac{1}{\sigma \pm i T - \varrho} \Big) d\s 
\ll 
\frac{\log T}{T^{m+1}} \frac{x^2}{\log x}. 
\end{equation}
In addition, it follows from \eqref{use10} that 
$\displaystyle 
   \int_{-1/4}^{2} \frac{x^{\sigma \pm i T} }{(\sigma \pm i T)(\sigma \pm i t - \varrho)} d \sigma \ll  \frac{x^2}{T},
$
so that, together with \eqref{def-nchit}, we obtain
\begin{equation}\label{v3chi2}
    \int_{-1/4}^{2} x^{\sigma \pm i T} H(\sigma \pm i T) \Big( \sum_{\overset{\varrho \in Z(\chi)}{|\gamma \mp T| \leq 1}} \frac{1}{\sigma \pm i T - \varrho} \Big) d \sigma 
    \ll \frac{x^2}{T^{m+1}} n_\chi(\pm T) 
    \ll x^2 \frac{\log T}{T^{m+1}}.
\end{equation}
Adding \eqref{v3chi1} and \eqref{v3chi2} gives
\begin{equation}\label{V3chi0}
V_{3,\chi} \ll x^2 \frac{\log T}{T^{m+1}}.
\end{equation}
Finally, adding \eqref{V1chi0}, \eqref{V2chi0}, and \eqref{V3chi0}
and taking the limit as $T,U \rightarrow \infty$ gives
\begin{equation}\label{v123}
    \lim_{T,U \rightarrow \infty} (V_{1,\chi}(x,T,U) + V_{2,\chi}(x,T,U) + V_{3,\chi}(x,T)) = 0.
\end{equation}
We conclude to the announced identity \eqref{explicitI} by combining \eqref{IwithJ}, \eqref{cauchyonJ}, \eqref{v123} and \eqref{deltachi}. 
\end{proof}
Rewriting
$$r(\chi) + \sum_{\varrho \in Z(\chi)} x^\varrho H(\varrho) = r(\chi) + \sum_{\overset{\varrho \in Z(\chi)}{|\varrho| < \frac{1}{2}}} \frac{1}{\varrho} -  \sum_{\overset{\varrho \in Z(\chi)}{|\varrho| < \frac{1}{2}}} \frac{1}{\varrho} + \sum_{\varrho \in Z(\chi)} x^\varrho H(\varrho) $$
we obtain
\begin{equation*}
\begin{split}
 I_{L/K}(x) -  \frac{|C|}{|G|}   x H(1)
 = & \frac{|C|}{|G|} \sum_{\chi} \bar{\chi}(g) \Big(   - (a(\chi) -\mathbb{1}(\chi)) \Big( \log (\alpha x)+ \int_\alpha^{\alpha +\delta} \frac{h(t)}{t} dt \Big) 
 \\ & 
 -  r(\chi) 
 - \sum_{\overset{\varrho \in Z(\chi)}{|\varrho| < \frac{1}{2}}} \frac{1}{\varrho} 
 +  \sum_{\overset{\varrho \in Z(\chi)}{|\varrho| < \frac{1}{2}}} \frac{1}{\varrho} 
 - \sum_{\varrho \in Z(\chi)} x^\varrho H(\varrho) 
 \Big.\\ \Big.&  - b(\chi)  \sum_{m \geq  1} x^{-(2m-1)}H(1-2m) - a(\chi) \sum_{m \geq  1} x^{-2m}H(-2m) \Big).
\end{split}
\end{equation*}
Using $\sum_\chi |a(\chi) - \mathbb{1}(\chi)| \leq \sum_\chi n_E = n_L$ and \eqref{fact4} in \eqref{explicitI}, we obtain:
\begin{corollary}\label{defji}
Under the assumption in \propref{formulaI} and for any real number $T > 2$, we have
\begin{equation}
\begin{split}
\label{def-|I|modified}
    \bigg|I_{L/K}(x) - \frac{|C|}{|G|} x H(1) \bigg|  \le
    & \frac{|C|}{|G|} \Big(x^{\beta_0} H(\beta_0) + n_L \Big| (\log \alpha x)+ \int_\alpha^{\alpha +\delta} \frac{h(t)}{t} dt \Big| + J^{(1)}(x)  + J^{(2)}(x) \\
     & + J^{(3)}(x) + J^{(4)}(x) + J^{(5)}(x,T) + J^{(6)}(x,T) \Big),
\end{split}
\end{equation}
where
\begin{align}\label{def-J1a}
    J^{(1)}(x)  = &\sum_\chi \Big( b(\chi) \sum_{m \geq  1} x^{-(2m-1)}|H(1-2m)| + a(\chi) \sum_{m \geq  1} x^{-2m}|H(-2m)| \Big) \\ \label{def-J1b} 
    & + \sum_\chi \Big| r(\chi) + \sum_{|\varrho|<\frac{1}{2}} \frac{1}{\varrho} \Big|,\\ \label{def-J2}
    J^{(2)}(x)  = & x^{1-\beta_0} H(1- \beta_0) - \frac{1}{1-\beta_0}, \\\label{def-J3}
    J^{(3)}(x) = &  \sum_{\varrho \neq 1- \beta_0, |\varrho| < \frac{1}{2}} \Big| x^\varrho H(\varrho) - \frac{1}{\varrho} \Big|,\\ \label{def-J4}
    J^{(4)}(x) = &  \sum_{\varrho \neq \beta_0, |\varrho| \geq  \frac{1}{2}, |\gamma| \leq 2} x^{\beta} |H(\varrho)|,\\ \label{def-J5}
    J^{(5)}(x,T) = &  \sum_{2 < |\gamma| < T} x^{\beta} |H(\varrho)|,\\ \label{def-J6}
    J^{(6)}(x,T) = &  \sum_{|\gamma| \geq  T} x^{\beta} |H(\varrho)|,
\end{align}
with $\alpha_4$ defined in \eqref{ahneq1}, $\beta_0$ being the possible real exceptional zero of $\zeta_L(s)$, $J^{(1)}$ summing over the non-trivial zeros, $\varrho$ of $L(s,\chi)$ and $J^{(3)}, J^{(4)}, J^{(5)}, J^{(6)}$ summing over the non-trivial zeros, $\varrho$ of the Dedekind $\zeta$-function, $\zeta_L(s)$. 
\end{corollary}
We recall that 
$E_C(x) \leq \max(E_C^+(x), E_C^-(x))$, where $E_C^+(x), E_C^-(x)$ respectively are $$E_C^{\pm}(x)=  \frac{\left|\psi_C^{\pm}(x) - \frac{|C|}{|G|} x \right|}{\frac{|C|}{|G|} x}$$ 
where we recall that $\psi_C^{\pm}(x)$ are defined just after \eqref{Eh}. We state the explicit inequality for $E_C^{\pm}(x)$:
\begin{theorem}\label{bnd-tildeECx}
Under the assumption in \propref{formulaI} and for any real number $T > 2$, we have
\begin{align}\label{def-|I|modified}
E_C^{\pm}(x)   & 
\leq \frac{|C|}{|G|} \Big(x^{\beta_0} H(\beta_0) + n_L \Big| (\log \alpha x)+ \int_\alpha^{\alpha +\delta} \frac{h(t)}{t} dt \Big| + J^{(1)}(x)  + J^{(2)}(x)\notag \\
     & + J^{(3)}(x) + J^{(4)}(x) + J^{(5)}(x,T) + J^{(6)}(x,T) \Big) + \ell_0 (\log d_L) \frac{\log x}{x},
\end{align}
where $\ell_0$ is defined in \eqref{def-el0}, 
$J^{(1)}(x)$, $J^{(2)}(x)$,
     $J^{(3)}(x)$, $J^{(4)}(x)$, $J^{(5)}(x,T)$,  $J^{(6)}(x,T)$ 
     are defined in \eqref{def-J1a}, \eqref{def-J2}, \eqref{def-J3}, \eqref{def-J4}, \eqref{def-J5}, \eqref{def-J6}
     respectively. 
\end{theorem}
\section{Preliminary results}\label{prelims}
Recall that $\chi$ is a character of the cyclic group Gal($L/E$) and $L(s,\chi)$ is it's associated Artin $L$-function. Also recall that $\gamma_\chi(s)$ defined in \eqref{def-gammachis} is the gamma function associated to $\chi$.
In this section, we write $s=\sigma+it$. 
\subsection{Bounds for $\frac{L'}{L}(s,\chi)$ and $\frac{\gamma_\chi'}{\gamma_\chi} (s)$}
We state \cite[Lemma 4.3]{bw} and \cite[Lemma 4.1] {bw} which were proven by  Winckler in his Thesis, and that we also independently verified. 
\begin{lemma}\label{Lgamma}
If $\Re(s) > 1$, then 
\begin{equation*}\label{eql'l}
\Big| \frac{L'}{L}(s,\chi) \Big| \leq \frac{n_E}{\Re(s) -1}.
\end{equation*}
\end{lemma}
\begin{lemma}\label{gamchi}
We have the following bounds for $\frac{\gamma_\chi'}{\gamma_\chi} (s)$:
\begin{enumerate}
    \item If $\Re(s) > -\frac{1}{2}$ and $|s| \geq  \frac{1}{8}$, then   
    \begin{equation*}\label{gm'gm}
    \Big| \frac{\gamma_\chi'}{\gamma_\chi} (s) \Big| \leq \frac{n_E}{2} \Big( \log (1 +|s|) + \frac{164}{7} \Big).
    \end{equation*}
    \item If $s = \sigma + i t$ with $\sigma \geq  1$, then
\begin{equation*}\label{gm'gm2}
 \Big| \frac{\gamma_\chi'}{\gamma_\chi} (s) \Big| \leq \frac{n_E}{2} \Big( \log (|t| + \sigma + 1) + \frac{539}{134} \Big). 
\end{equation*}
\item If $s = \sigma + i t$ with $\sigma \geq  2$, then
\begin{equation*}\label{gm'gm2-2}
 \Big| \frac{\gamma_\chi'}{\gamma_\chi} (s) \Big| \leq \frac{n_E}{2} \Big( \log (|t| + \sigma + 1) + \frac{405}{134} \Big). 
\end{equation*}
\end{enumerate}
\end{lemma} 
\subsection{Zero-free regions of Dedekind zeta function}
\label{ZFR-zetaL}

The first theorem describes a region free of zeros, with at most one possible exception, which is entirely explicit and valid for all field extensions.  
\begin{theorem}\label{thmR} 
Let $L$ be a number field with $n_L \geq  2$. Let $\varrho = \beta + i \gamma$ be a non-trivial zero of $\zeta_L(s)$ with $\varrho \neq \beta_0$. Then there exists $R_1, R_2 > 0$ such that
\begin{equation}\label{ahneq2}
\beta < 1 - \frac1{R_{1,L} \log(4 \Delta_L))} \ \textnormal{ when } \ |\gamma| \leq  2, 
\end{equation}
and
\begin{equation}\label{lee}
\beta < 1 -  \frac1{R_{2,L} \log(\Delta_L  |\gamma|)} \ \textnormal{ when } \ |\gamma| >  2,
\end{equation}
where $R_{1,L} = R_1 n_L$, $R_{2,L} = R_2 n_L$ and $\Delta_L  = d_L^{1/n_L}$.  In particular, 
\begin{equation}\label{akR}
R_1 = 20 \ \textnormal{ and }  \ R_2 = 12.2411
\end{equation}
are admissible constants constants. 
\end{theorem}

\begin{proof}
First, we  show \eqref{lee} holds with $R_2=12.2411$. Lee in \cite[Theorem 1]{el} proved that $\zeta_L(s)$ has no zeros in the region:
\begin{equation}\label{zerofree-lee}
    \Re(s) \geq  1 - \frac{1}{C_1 \log d_L + C_2 n_L (\log |\Im (s)|) + C_3 n_L + C_4} \ \textnormal{ and } \ |\Im(s)| \geq  1,
\end{equation}
where 
$C_1= 12.2411, C_2=9.5347, C_3=0.05017$ and $C_4=2.2692$.
For $n_L \geq  2$ and $|\Im(s)| \geq  2$,
$$ C_1 \log d_L + C_2 n_L (\log |\Im (s)|) + C_3 n_L + C_4 \leq C_1 n_L \bigg(\log (d_L^{1/n_L}) + \frac{C_2}{C_1} \log |\Im (s)| + \frac{C_3}{C_1} + \frac{C_4}{2C_1} \bigg),$$
with
$$\frac{C_2}{C_1} \log |\Im (s)| + \frac{C_3}{C_1} + \frac{C_4}{2C_1} \leq \log |\Im(s)|.$$
Thus $R_2 = C_1 = 12.2411$ satisfies \eqref{lee} as required. 

Next, we examine the $R_1$ value. 
In an upcoming article \cite[Theorem 1.2]{dglsw}, the first author in collaboration with four other researchers proved that if $\varrho = \beta + i \gamma$ is a non-trivial zero of $\zeta_L(s)$ with $\varrho \neq \beta_0$, then $\varrho$ satisfies
\begin{align}
\label{region-b/w-0-1-}
    \beta & < 1 - \frac1{19.55293 n_L \log(\Delta_L)} \quad \text{when} \quad |\gamma| \leq 1.
\end{align}
Next, we shall establish  
\begin{align} 
\label{region-b/w-1-2-c}
    \beta & < 1 - \frac1{12.2411 n_L \log(\Delta_L  (|\gamma|+0.102))} \quad \text{when} \quad 1 < |\gamma| \leq 2. 
\end{align}
 Let $C_0' > 0$ be a constant such that
\begin{align}\label{C_i-bound5}
    C_2 n_L (\log |\Im (s)|) + C_3 n_L + C_4 & \leq C_1 n_L (\log (|\Im (s)| + C_0')),
    \end{align}
    which is equivalent to
    \begin{align}\label{C_i-bound6}
        C_2 (\log |\Im (s)|) + C_3  + \frac{C_4}{n_L} & \leq C_1 (\log (|\Im (s)| + C_0')).
    \end{align}
    Since $n_L \geq 2$, \eqref{C_i-bound6} holds in the region $1 < |\Im(s)| \leq 2$ if
    \begin{equation*}\label{C_i-bound7}
        g(C_0') = \min_{|\Im(s)| \in (1,2]} C_1 (\log (|\Im (s)| + C_0')) - C_2 (\log |\Im (s)|) - C_3  -  \frac{C_4}{2} \geq 0.
    \end{equation*}
    Using Maple, we confirm that $g(0.102) > 0$ and $g(0.101) < 0$. Thus using $C_0' = 0.102$, \eqref{C_i-bound5} in \eqref{zerofree-lee} establishes \eqref{region-b/w-1-2-c}. 
Finally, combining \eqref{region-b/w-0-1-} and \eqref{region-b/w-1-2-c} with $\Delta_L \leq 4 \Delta_L$, $\Delta_L(|\gamma| + 0.102) \leq 4 \Delta_L$ for $1 < |\gamma| \leq 2$, and $\max \{12.2411, 19.55293\} \leq 20$ establishes 
    \eqref{ahneq2} with $R_1=20$. 
    This completes the proof.
\end{proof}

The next theorem provides a wider zero-free region closer to the real axis:
\begin{theorem}\label{thmal4}
Let $L$ be a number field with $n_L \geq  2$. Then there exists $\alpha_4 > 0$ such that the Dedekind zeta function, $\zeta_L(s)$ has at most one zero $\varrho = \beta + i \gamma$ with 
\begin{equation}\label{ahneq1}
\beta > 1 - \frac{1}{\alpha_4 \log d_L} \ \textnormal{  and  } \ |\gamma| < \frac{1}{\alpha_4 \log d_L}.
\end{equation}
This zero, if it exists, is real and simple, and denoted as $\beta_0$.
Note that admissible constants for $\alpha_4$ are given by
\begin{equation}\label{akalpha4}
\alpha_4 = 
\begin{cases}
    2\ & \textrm{if there are no exceptional zeros,}\\
    1.7\ & \textrm{if there is an exceptional zero.}
\end{cases}
\end{equation}
This zero, if it exists, is real and simple, and denoted as $\beta_0$.
\end{theorem}
The value $\alpha_4=2$ is a direct consequence of Ahn and Kwon \cite[Theorem 1]{ak} 
and the value $\alpha_4=1.7$ is a direct consequence of Kadiri and Wong \cite[Proposition 6]{kapj}.
\subsection{Counting zeros of $L(s,\chi)$ and $\zeta_L(s)$}
\label{counting-zeros-zetaL}
Let $T \geq  1$.
We introduce the zero-counting function 
\begin{equation*}
  \label{NLT} 
   N_L(T) = \# \{ \varrho=\beta+i\gamma  \ | \ \varrho \in Z(\zeta_L), \ 0 < \beta < 1, \ |\gamma| \leq T \}. 
\end{equation*}
and recall the zero-counting function   $n_{\chi,a}(T)$ previously introduced in \eqref{nchiaT}. 
Lagarias and Odlyzko in \cite[Lemma 5.4]{lo} proved a non-explicit bound for $n_{\chi,1}(T)$ and Winckler in \cite[Lemma 4.6]{bw} made their result explicit. We shall prove a more general bound for $n_{\chi,a}(T)$. This is achieved by generalizing an argument for Dirichlet $L$-functions by Fiorilli and Martin in \cite[Section 5]{mf}. For a non-principal Dirichlet character $\chi$ modulo $q$ and for any real number $T$, they proved that for the Dirichlet $L$-function $L(s,\chi)$, we have:
$$ 
\char"0023 \{\varrho = \beta+i\gamma : L(\beta+i\gamma,\chi)=0,\ 0<\beta <1\ \text{and}\  |\gamma -T| \leq 2 \} \leq 4 \log(0.609 q (|T| +5)). 
$$ 
To prove this, in \cite[Lemma 5.3]{mf}, they used $2 + i T$ in their sum to obtain their result. We, on the other hand, use $1+\varepsilon + i T$ instead to get closer to the $\Re(s) =1$ line. First, we use the explicit formula for $\frac{L'}{L}(s,\chi)$ in \eqref{formula-L'/L} to prove the following result and then use this result to prove a bound for $n_{\chi,a}(T)$. Then we derive a bound for $N_L(T)$ which is more relevant for small values of $T$.
\begin{lemma}\label{b1+e}
Let $\varepsilon > 0$. Let $\mathbb{1}(\chi)$ and $A(\chi)$ be as in \eqref{deltachi} and \eqref{Achi}. Let $T$ be any real number. Then 
\begin{align*}
\label{real1+e}
 \sum_{\varrho \in Z(\chi)} \Re\Big( \frac{1}{1+\varepsilon + i T - \varrho} \Big) 
  \leq & n_E \Big( \frac{1}{2} \log (2 + \varepsilon +|T|) + \Big( \frac{1}{\varepsilon} + \frac{539}{268} \Big) \Big)
 + \frac{\log A(\chi)}{2} 
 \\ &
 + \left( \frac{1+\varepsilon}{\sqrt{(1+\varepsilon)^2+T^2}} + \frac{\varepsilon}{\sqrt{\varepsilon^2+T^2}} \right)
\mathbb{1}(\chi).
\end{align*}
\end{lemma}
\begin{proof}
By the classical explicit formula for $\frac{L'}{L}(s,\chi)$ given in \eqref{formula-L'/L} and \eqref{eqbchi}, we have 
\begin{align*}
 \sum_{\varrho \in Z(\chi)} \Re\Big( \frac{1}{1+\varepsilon + i T - \varrho} \Big)
  = & \Re \frac{L'}{L} (1 + \varepsilon + i T, \chi) + \frac{\log A(\chi)}{2} 
  \\ &
 + \mathbb{1}(\chi) \Re \Big( \frac{1}{1+\varepsilon + i T} + \frac{1}{\varepsilon + i T} \Big) + \Re \frac{\gamma_\chi ' }{\gamma_\chi} (1 + \varepsilon + i T) .
\end{align*}
Notice that
$$\Re \Big( \frac{1}{1+\varepsilon + i T} + \frac{1}{\varepsilon + i T} \Big) =  \frac{1+\varepsilon}{\sqrt{(1+\varepsilon)^2+T^2}} + \frac{\varepsilon}{\sqrt{\varepsilon^2+T^2}}
.$$
Thus using \lmaref{Lgamma} and \lmaref{gamchi} for $\Re(s) \geq  1$, we obtain the required result. 
\end{proof}
\begin{proposition}\label{nchia}
Under the assumption in \lmaref{b1+e} and for any real numbers $a > 0$, we have
$$ 
    n_{\chi,a}(T) \leq c_1(a,\varepsilon) \log A(\chi) + c_2 (a,\varepsilon,T) n_E +  c_3(a,\varepsilon,T) \mathbb{1}(\chi),
$$ 
where
\begin{equation*}
\begin{split}
\label{def-nL-c1-c2}
    c_1(a, \varepsilon) & = \frac{(1+\varepsilon)^2 + a^2}{2 \varepsilon}, \\
    c_2(a, \varepsilon, T) & = c_1(a, \varepsilon) \log (2 + \varepsilon +|T|) + 2 c_1(a,\varepsilon) \Big( \frac{1}{\varepsilon} + \frac{539}{268} \Big),  \\
    \text{and } \quad c_3(a, \varepsilon, T) & = 2 c_1(a,\varepsilon) \Big( \frac{1+\varepsilon}{\sqrt{(1+\varepsilon)^2+T^2}} + \frac{\varepsilon}{\sqrt{\varepsilon^2+T^2}} \Big).
\end{split}
\end{equation*}
\end{proposition} 
\begin{proof}
Since $0 < \beta < 1$, therefore $\varepsilon < 1 + \varepsilon - \beta < 1 + \varepsilon$.
As a result
\begin{equation*}
    \sum_{\overset{\varrho \in Z(\chi)}{|T - \gamma| \leq a}} 1 
     \leq \frac{(1+\varepsilon)^2 + a^2}{\varepsilon} \sum_{\varrho \in Z(\chi)} \frac{1+\varepsilon - \beta}{(1+\varepsilon-\beta)^2 + (T -\gamma)^2}
      = \frac{(1+\varepsilon)^2 + a^2}{\varepsilon} \sum_{\varrho \in Z(\chi)} \Re \Big( \frac{1}{1+ \varepsilon +i T - \varrho} \Big).
\end{equation*}
Applying \lmaref{b1+e} completes the proof.
\end{proof}
\begin{theorem}
    \label{thm-bnd-nchiNL}
For $n_L\leq \mathscr{M}\log d_L$ and for any $T > 0$, we have 
\begin{equation}
\label{bnd-NLT}
N_L(T) \leq  \alpha_0(T) \log d_L ,
\end{equation}
and for any $T \geq 3$,
\begin{equation*}
N_L(T+1)-N_L(T-1) \leq b_1 n_L (\log T) + b_2 n_L  + b_3 \log d_L + b_4. 
\end{equation*}
where $\alpha_0(T)$  and the $b_i$'s are defined in \eqref{def-alpha0(T)} and \eqref{def-bi} respectively. In addition, their values in terms of $\mathscr{M}$ are listed in Table \ref{Table-NL-1-2} in Appendix \ref{appendixtables}.
\end{theorem}
\begin{proof}
Combining \propref{nchia} for $a>0$ and $T=0$ with \eqref{fact1}, \eqref{fact2} and \eqref{fact3}, we obtain
$$ 
N_L(a) = 
\sum_{\chi} n_{\chi,a}(0) \leq 
c_1(a,\varepsilon) \log d_L + c_2 (a,\varepsilon,0) n_L +  c_3(a,\varepsilon,0).
$$  
Together with \eqref{def-Minkowski} $n_L\leq \mathscr{M}( \log d_L)$ and $d_L \geq d_0$, we deduce
$$N_L(a) \leq B(a,\epsilon) \log d_L$$
where
$$ 
B(a,\epsilon) = c_1(a,\varepsilon)  + c_2(a,\varepsilon,0) \mathscr{M}  +  \frac{ c_3(a,\varepsilon,0)}{\log d_0}.
$$  
For each entry $(\mathscr{M},d_0)$ in Table \ref{n0d0}, we define 
\begin{align}
\label{def-alpha0(T)}
    & \alpha_0(T) = \min_{\epsilon>0} B(T,\epsilon).
\end{align}
As a result, we obtain
$$N_L(T) \leq \alpha_0(T) \log d_L
.$$
Let $T\geq T_0 \geq 3$. Since 
$$ 
N_L(T+1) - N_L(T-1) 
\leq \sum_{\chi} {(n_{\chi,1}(T)+n_{\chi,1}(-T) )}, 
$$ 

we can apply \propref{nchia} for $a=1$ with \eqref{fact1}, \eqref{fact2} and \eqref{fact3}:
\begin{align*}
&
\sum_{\chi} n_{\chi,1}(\pm T) \leq 
c_1(a,\varepsilon) \log d_L + c_2 (a,\varepsilon,T) n_L +  c_3(a,\varepsilon,T) 
\\ &
\leq 
n_L \Big( c_1(1,\varepsilon) (\log T)
\Big( 
1 + \frac{ \log (1+\frac{2+\epsilon}{T_0}) 
}{\log T_0} 
\Big) +  2 c_1(a,\varepsilon) \Big( \frac{1}{\varepsilon} + \frac{539}{268} \Big) \Big) + c_1(1,\varepsilon) \log d_L + c_3(1,\varepsilon,T_0)
\\ &
\leq 
c_1(1,\varepsilon) 
\Big( 
1 + \frac{ \log (1+\frac{2+\epsilon}{3}) 
}{\log 3} 
\Big) n_L (\log T) + 2 c_1(a,\varepsilon) \Big( \frac{1}{\varepsilon} + \frac{539}{268} \Big) n_L  + c_1(1,\varepsilon) \log d_L + c_3(1,\varepsilon,3).
\end{align*}
We conclude by choosing $\epsilon_0= 1.1814$ which minimizes $c_1(1,\epsilon)\Big( 1 + \frac{ \log (1+\frac{2+\epsilon}{3}) }{\log 3} \Big) $.
Rounding up the constants to the 4th decimal, we define for this value of $\varepsilon$:
\begin{equation}
    \begin{split}
\label{def-bi}
& b_1 = 2c_1(1,\epsilon_0) 
\Big( 
1 + \frac{ \log (1+\frac{2+\epsilon_0}{3}) 
}{\log 3} 
\Big) = 8.0818,
\  b_2 = 4 c_1(a,\varepsilon_0) \Big( \frac{1}{\varepsilon_0} + \frac{539}{268} \Big) = 27.8581,
\\& b_3 = 2c_1(1,\epsilon_0) = 4.8743,
\  b_4 = 2 c_3(1,\epsilon_0,3) = 9.3052.
    \end{split}
\end{equation}
\end{proof}
We state here the bound for $N_L(T)$ as proven by Hasanalizade, Shen, and Wong in \cite[Corollary 1.2]{HSW2}:
\begin{theorem}\label{theohsw}
Let $T \geq 1$ be any real number and $N_L(T)$ be the number of zeros $\varrho = \beta + i \gamma$ of $\zeta_L(s)$ in the region $0 < \beta < 1$ and $|\gamma| \leq T$. 
Then, we have
\begin{equation*}\label{ttequation}
    | N_L(T) - P_L(T) | \leq E_L(T),
\end{equation*}
where
\begin{equation}\label{tteq}
    P_L(T) = \frac{T}{\pi} \log \bigg( d_L \bigg( \frac{T}{2 \pi e} \bigg)^{n_L} \bigg) \ \textnormal{ and } \ E_L(T) = \alpha_1 ( \log d_L + n_L \log T) + \alpha_2 n_L + \alpha_3,
\end{equation}
and 
\begin{equation}\label{valalphai}
\alpha_1 = 0.228, \ \alpha_2 = 23.108, \ \textnormal{ and } \ \alpha_3 = 4.520.
\end{equation}
\end{theorem}
\begin{corollary}\label{cor-NL(T)}
    Using \eqref{def-Minkowski} $n_L \leq \mathscr{M} (\log d_L)$, for $T \geq 1$, we obtain
    $$N_L(T) \leq \alpha_0'(T) \log d_L,$$
    where $\alpha_0(T)$ is given as
    \begin{align}\label{alpha_0'-N_L}
        \alpha_0'(T) = \frac{T}{\pi} + \alpha_1 + \mathscr{M} \left( \frac{T}{\pi} \log \bigg( \frac{T}{2 \pi e} \bigg) + \alpha_1 \log T + \alpha_2 \right) + \frac{\alpha_3}{\log d_0},
    \end{align}
    and its values at $T=1$ and $2$ are explicitly given in Table \ref{Table-NL-1-2} in Appendix \ref{appendixtables}.
\end{corollary}
\begin{remark}\label{bound-on-N_L}
\begin{enumerate}
    \item For $T =1$, we note that \thmref{thm-bnd-nchiNL} provides better bounds than \thmref{theohsw}.
For $\mathscr{M} = \frac{2}{\log 3}$, 
the first gives
$ 
N_L(1) \leq  40.1778 \log d_L    
$
while the latter gives
$ N_L(1) 
\leq  45.0838 \log d_L$. For $\mathscr{M} = \frac{1}{\log 10} \approx 0.434294$, the bounds are $9.52550 \log d_L$ and $10.2832 \log d_L$ respectively. 
\item For $T \geq 2$, \thmref{theohsw} provides better bounds for $N_L(T)$ than \thmref{thm-bnd-nchiNL}. For instance, for $\mathscr{M} = \frac{2}{\log 3}$, the former gives $N_L(2) \leq 44.8486 \log d_L$ and the latter gives $N_L(2) \leq 52.4347 \log d_L$ . For $\mathscr{M} = \frac{1}{\log 10}$, the bounds are $10.4692 \log d_L$ and $12.4297 \log d_L$ respectively. Therefore, we will use the values for $N_L(2)$ given by \thmref{theohsw} in our calculations.
\end{enumerate}
\end{remark}
\section{Bounding the $J^{(i)}$'s}
In this section, we study the sums over zeros given in \corref{defji} and provide explicit bounds for such sums.

\subsection{Bounding $J^{(1)}$}
Recall that $B(\chi)$ is the undefined constant in the expression for $\frac{L'}{L}(s,\chi)$ given in \eqref{formula-L'/L}. Lagarias and Odlyzko in \cite[Lemma 5.5]{lo} proved that, for any $\varepsilon$ with $0 < \varepsilon \leq 1$, we have
\begin{equation*}\label{bigobchi}
    B(\chi) + \sum_{\overset{\varrho \in Z(\chi)}{|\varrho| < \varepsilon}} \frac{1}{\varrho}\ll \frac{\log A(\chi) + n_E}{\varepsilon},
\end{equation*}
and Winckler in \cite[Lemma 4.7]{bw} made their result explicit. We prove :

\begin{lemma}\label{newbchi}
For any $\varepsilon \in (0,1]$, we have
$$ \sum_\chi \Big| B(\chi) + \sum_{\overset{\varrho \in Z(\chi)}{|\varrho| < \varepsilon}} \frac{1}{\varrho} \Big| \leq \Big( 2.6430 +
\Big( 1 + \frac{1}{\varepsilon} \Big) \alpha_0(1) 
\Big) \log d_L + 46.8427 n_L + 10.54, $$ 
where $\alpha_0$ is defined in \eqref{def-alpha0(T)} (with values in Table \ref{Table-NL-1-2}).
\end{lemma}
\begin{remark}
In the proof of \lmaref{newbchi}, we can use $s = r$ with any $r > 1$ instead of $s=2$. However, we found that $s=2$ gives the best results. 
\end{remark}
\begin{proof}
We split the sum over the non-trivial zeros of $L(s,\chi)$ as follows:
\begin{equation*}\label{split}
    \sum_{\overset{\varrho \in Z(\chi)}{|\varrho| < \varepsilon}} \frac{1}{\varrho} = \sum_{\varrho \in Z(\chi)} \Big( \frac{1}{\varrho} + \frac{1}{2 - \varrho} \Big) - \sum_{\overset{\varrho \in Z(\chi)}{|\varrho| \geq  1}} \Big( \frac{1}{\varrho} + \frac{1}{2 - \varrho} \Big) - \sum_{\overset{\varrho \in Z(\chi)}{|\varrho| < 1}} \frac{1}{2 - \varrho} - \sum_{\overset{\varrho \in Z(\chi)}{\varepsilon \leq |\varrho| < 1}} \frac{1}{\varrho}. 
\end{equation*} 
Therefore
\begin{multline}
    \label{bchiuse1}
   \sum_\chi \Big| B(\chi) + \sum_{\overset{\varrho \in Z(\chi)}{|\varrho| < \varepsilon}} \frac{1}{\varrho} \Big| 
 \leq \sum_\chi \Big| B(\chi) + \sum_{\varrho \in Z(\chi)} \Big( \frac{1}{\varrho} + \frac{1}{2 - \varrho} \Big) \Big|  
   + \sum_\chi \Big| \sum_{\overset{\varrho \in Z(\chi)}{|\varrho| \geq  1}} \Big( \frac{1}{\varrho} + \frac{1}{2 - \varrho} \Big) \Big| 
 \\
 + \sum_\chi \Big|  \sum_{\overset{\varrho \in Z(\chi)}{|\varrho| < 1}} \frac{1}{2-\varrho} \Big| + \sum_\chi \Big|  \sum_{\overset{\varrho \in Z(\chi)}{\varepsilon \leq |\varrho| < 1}} \frac{1}{\varrho} \Big|.
\end{multline}
In the second sum on the right hand side of \eqref{bchiuse1}, using $\big| \frac{1}{2 - \varrho} + \frac{1}{\varrho} \big| = \frac{2}{|(2-\varrho)\varrho|} \leq \frac{2}{|\varrho|^2}$, we obtain
\begin{equation}\label{chizeta1}
\sum_\chi \Big| \sum_{\overset{\varrho \in Z(\chi)}{|\varrho| \geq  1}} \big( \frac{1}{2 - \varrho} + \frac{1}{\varrho} \big) \Big| 
\leq  \sum_{\overset{\varrho \in Z(\zeta_L)}{|\varrho| \geq  1}} \frac{2}{|\varrho|^2} 
\leq 2   N_L(1) +  2 \sum_{k=1}^\infty \frac{N_L(k +1) - N_L(k)}{k^2} 
= 2 \sum_{k=2}^\infty  \frac{(2k-1)N_L(k)}{(k(k-1))^2}.
\end{equation}
Moreover, \thmref{theohsw} gives 
\begin{equation}\label{NLuse2}
    N_L(k) \leq \frac{k}{\pi} \log \Big( \frac{d_L}{(2\pi e)^{n_L}} \Big)  + \frac{n_L}{\pi} (k \log k) + \alpha_1 n_L \log k + \alpha_1 \log d_L + \alpha_2 n_L + \alpha_3. 
\end{equation}
Together with the calculations 
\begin{equation*}
\label{sumbounds}
\begin{split}
&
2.645 \leq \sum_{k=2}^\infty \frac{k(2k-1)}{(k(k-1))^2} \leq 2.65, 
\ \sum_{k=2}^\infty \frac{k(\log k)(2k-1)}{(k(k-1))^2} \leq 3.06,
\\  & 
\sum_{k=2}^\infty \frac{2k-1}{(k(k-1))^2} = 1, \ 
\textnormal{and } \sum_{k=2}^\infty \frac{(\log k)(2k-1)}{(k(k-1))^2} \leq 0.87, 
\end{split}
\end{equation*}
we deduce an explicit bound for \eqref{chizeta1}:
\begin{equation}\label{part1}
    \sum_\chi \Big| \sum_{\overset{\varrho \in Z(\chi)}{|\varrho| \geq  1}} \Big( \frac{1}{2-\varrho} + \frac{1}{\varrho} \Big) \Big| \leq  \frac{5.3}{\pi} \log d_L - \frac{5.29 \log (2 \pi e)}{\pi} n_L + \frac{6.12}{\pi} n_L + 1.74 \alpha_1 n_L + 2 \alpha_1 \log d_L + 2 \alpha_2 n_L + 2 \alpha_3.
\end{equation}
We now bound the third and fourth sums on the right hand side of \eqref{bchiuse1}. Noting that when $|\varrho| < 1$, then $|2 -\varrho| > 1$, and using \eqref{bnd-NLT}, we deduce
\begin{equation}\label{part2}
    \sum_\chi \Big|  \sum_{\overset{\varrho \in Z(\chi)}{|\varrho| < 1}} \frac{1}{2-\varrho} \Big| + \sum_\chi \Big|  \sum_{\overset{\varrho \in Z(\chi)}{\varepsilon \leq |\varrho| < 1}}  \frac{1}{\varrho} \Big| \leq \Big( 1 + \frac{1}{\varepsilon} \Big) N_L(1) 
    \leq \Big( 1 + \frac{1}{\varepsilon} \Big) \alpha_0(1) (\log d_L).
\end{equation} 
For the remaining sum on the right hand side of \eqref{bchiuse1}, we apply \eqref{formula-L'/L} with $s =2$. Note that
$$ 
 \sum_\chi \Big| B(\chi) + \sum_{\varrho \in Z(\chi)} \Big( \frac{1}{2 - \varrho} + \frac{1}{\varrho} \Big) \Big|  
 = \sum_\chi \Big| \frac{L'}{L}(2,\chi) + \frac{1}{2} \log (A(\chi)) + \mathbb{1}(\chi) \Big( \frac{1}{2} + \frac{1}{2-1} \Big) + \frac{\gamma_\chi'}{\gamma_\chi}(2) \Big|.
$$ 
By \lmaref{Lgamma}, we have $\big| \frac{L'}{L}(2,\chi) \big| \leq n_E$ and by part 3 of \lmaref{gamchi}, we have $\big| \frac{\gamma_\chi'}{\gamma_\chi}(2) \big| \leq \frac{n_E}{2} (\log 3 + \frac{405}{134})$.
Thus, using \eqref{fact3}, \eqref{fact1}, \eqref{fact2} for the sums over $\chi$, we obtain 
\begin{equation}\label{part3}
    \sum_\chi \Big| B(\chi) +\sum_{\varrho \in Z(\chi)} \Big( \frac{1}{2 - \varrho} + \frac{1}{\varrho} \Big) \Big| \leq \Big( \frac{\log 3}{2} + \frac{673}{268} \Big) n_L + 0.5 \log d_L + \frac{3}{2}. 
\end{equation}
Finally, we insert \eqref{part1}, \eqref{part2} and \eqref{part3} into \eqref{bchiuse1}, and obtain
\begin{multline}
    \label{finalbchi}
   \sum_\chi \Big| B(\chi) + \sum_{\overset{\varrho \in Z(\chi)}{|\varrho| < \varepsilon}} \frac{1}{\varrho} \Big| 
  \leq 
 \Big( 0.5 + \frac{5.3}{\pi} + 2 \alpha_1 + \Big( 1 + \frac{1}{\varepsilon} \Big) \alpha_0(1) \Big) (\log d_L)
 \\
 + 
 \Big( \frac{\log 3}{2} + \frac{673}{268} 
 - \frac{5.29 \log (2 \pi e)}{\pi} 
 + \frac{6.12}{\pi} 
  + 1.74 \alpha_1  
 + 2 \alpha_2 \Big) n_L  
 + \frac{3}{2}
 + 2 \alpha_3.
 \end{multline}
We shall apply \thmref{theohsw} with the admissible values $(\alpha_1,\alpha_2,\alpha_3) = (0.228,23.108,4.520)$ in \eqref{finalbchi} to obtain the required result.
\end{proof}
\begin{lemma}\label{bouJ0J1}
Let $n_L \geq n_0 \geq 2$, $x_0\geq  2$, and $\delta \leq \delta_0 < 1 - \frac{\sqrt{2}}{x_0}$. For all $x \geq x_0$, we have 
\begin{equation*} 
\Big| \log (\alpha x)+ \int_\alpha^{\alpha +\delta} \frac{h(t)}{t} dt \Big| n_L + J^{(1)}(x) \leq \ell_1 (\log d_L) (\log x),
\end{equation*}
where 
\begin{multline}
\label{def-el1}
\ell_1 = \ell_1(\delta_0, \mathscr{M}, n_0, x_0) 
= \frac{3.1430 + 3\alpha_0(1)}{\log x_0} 
\\ + \mathscr{M} \Big( 1 + \frac{\delta_0}{2(1 - \delta_0) (\log x_0)} + \frac{1}{((1-\delta_0)x_0)^2 \log x_0} + \frac{48.3969}{\log x_0} + \frac{11.54}{n_0 (\log x_0)} \Big),
\end{multline}
where $\alpha_0$ is defined in \eqref{def-alpha0(T)} (with values in Table \ref{Table-NL-1-2}).
\end{lemma}
\begin{proof}
Using $\alpha = \{ 1-\delta, 1 \}$ with $\delta \leq \delta_0 < 1 - \frac{\sqrt{2}}{x_0}$ yields $\log (\alpha x) \geq 0$ for all $x \geq x_0$. Thus using $|\log (\alpha x)| = \log (\alpha x) \leq \log x$ and $\int_\alpha^{\alpha+\delta} h(t) \ dt = \delta/2$ together with $\alpha \geq  1 - \delta$, we obtain
\begin{equation}\label{boumodpart}
\Big| \log (\alpha x)+ \int_\alpha^{\alpha +\delta} \frac{h(t)}{t} \ dt \Big| \leq \log x + \frac{\delta}{2(1-\delta)}.
\end{equation}
Let $J^{(1)}(x) = J^{(1a)}(x)+J^{(1b)}$ where $J^{(1a)}(x)$ is the sum in \eqref{def-J1a} and $J^{(1b)}$ is the sum in \eqref{def-J1b}. 
Using \eqref{bnd-H} with $\int_0^1 |g'(u)| \ du = 1$, we have $|H(s)| \leq \frac{(1-\delta)^{\Re(s)}}{|s|}$ for $\Re(s) \leq 0$ which we insert in \eqref{def-J1a} to obtain
$$ 
    J^{(1a)}(x) \leq \sum_\chi \Big( b(\chi) \sum_{m \geq  1} \frac{((1-\delta)x)^{-(2m-1)}}{2m-1} + a(\chi) \sum_{m \geq  1} \frac{((1-\delta)x)^{-2m}}{2m} \Big).
$$ 
We combine 
$$  \sum_{m \geq  1} \frac{((1-\delta)x)^{-(2m-1)}}{2m-1} = \sum_{m \geq  1} \frac{((1-\delta)x)^{-m}}{m} - \sum_{m \geq  1} \frac{((1-\delta)x)^{-2m}}{2m},
$$ 
with 
$b(\chi) = n_E - a(\chi)$, 
$a(\chi) \leq n_E$, 
and obtain
\begin{align*}
 &   b(\chi) \sum_{m \geq  1} \frac{((1-\delta)x)^{-(2m-1)}}{2m-1} + a(\chi) \sum_{m \geq  1} \frac{((1-\delta)x)^{-2m}}{2m}
 \\& = \frac{n_E}{2} \log \bigg( \frac{ (1-\delta)x +1}{ (1-\delta)x -1}\bigg) + a(\chi) \log \bigg( \frac{(1-\delta)x}{ (1-\delta)x +1} \bigg) 
\leq \frac{n_E}{2} \log \bigg( \frac{((1-\delta)x)^2}{((1-\delta)x)^2 - 1} \bigg) \leq \frac{ n_E}{((1-\delta)x)^2}, 
\end{align*}
where the last inequality is valid for $((1-\delta)x) \geq ((1-\delta_0)x_0) > \sqrt{2}$. 
We conclude by summing over $\chi$ \eqref{fact2}: 
\begin{equation}\label{boundJ0}
   J^{(1a)}(x) \leq 
   \sum_\chi   \frac{n_E}{((1-\delta)x)^2}
   = \frac{n_L}{((1-\delta)x)^2}.
\end{equation}
Inserting $\frac{\Gamma'}{\Gamma} ( \frac{1}{2} ) = -2 (\log 2) - \gamma$, $\frac{\Gamma'}{\Gamma}(1) = - \gamma$ and $a(\chi) + b(\chi) = n_E$ into \eqref{def-r(chi)}, we obtain
$$
\bigg| r(\chi) + \sum_{\overset{\varrho \in Z(\chi)}{|\varrho| < \frac{1}{2}}} \frac{1}{\varrho} \bigg| 
\leq \bigg| B(\chi) + \sum_{\overset{\varrho \in Z(\chi)}{|\varrho|<\frac{1}{2}}} \frac{1}{\varrho} \bigg|+ \frac{\log A(\chi)}{2} + \mathbb{1}(\chi) + n_E\bigg( \frac{\log \pi}{2} + \frac{\gamma}{2} + \log 2 \bigg).
$$
By using \eqref{fact3}, \eqref{fact1}, and \eqref{fact2} for sums over $\chi$ and applying \lmaref{newbchi} with $\varepsilon = \frac{1}{2}$: 

\begin{equation}\label{boundJ1}
    J^{(1b)} 
   \leq \Big( 3.1430 +
3\alpha_0(1)  \Big)
\log d_L + 48.3969 n_L +  11.54.
\end{equation}
Combining \eqref{boumodpart}, \eqref{boundJ0}, and \eqref{boundJ1}, we have
\begin{multline*}
  \Big( \log (\alpha x)+ \int_\alpha^{\alpha +\delta} \frac{h(t)}{t} dt \Big) n_L + J^{(1)}(x) \\
  \leq 
  \Big( 3.1430 +
3\alpha_0(1)  \Big)\log d_L + \Big( \log x + \frac{\delta}{2(1-\delta)}  + \frac{1}{((1-\delta)x)^2} + 48.3969 \Big) n_L + 11.54,
\end{multline*}
and we conclude with the assumptions $x \geq x_0$, $\delta \leq \delta_0$ and $n_0 \leq n_L \leq (\log d_L) \mathscr{M}$.
\end{proof}

\subsection{Bounding $J^{(2)}$}
\begin{lemma}\label{bouJ2}
Let $x_0\geq 2$ and $\delta \leq \delta_0 <1 - \frac{1}{x_0}$. For all $x \geq x_0$, we have 
\begin{equation*}
  J^{(2)}(x) \leq \ell_2 (\log x) x^\frac{1}{2},
\end{equation*}
where
\begin{equation}\label{def-el2}
\ell_2 = \ell_2(\delta_0,x_0) = 1 + \frac{\delta_0}{2 (1 - \delta_0) \log x_0}.
\end{equation}
\end{lemma}
\begin{proof}
Recall from \eqref{def-J2}  that  $ J^{(2)}(x) = x^{1-\beta_0} H(1-\beta_0)-\frac{1}{1-\beta_0}$.
Also,  from \eqref{defh} $h(t) =1$ for $0 \leq t \leq \alpha$, $h(t)=0$ for $t > \alpha+\delta$, and it follows that
 \eqref{def-Mellin}, we find
\begin{equation}\label{use1J3}
    x^s H(s) - \frac{1}{s} = \frac{1}{s} ( (\alpha x)^s -1 ) + x^s \int_{\alpha}^{\alpha+\delta} h(t) t^{s-1} \ dt.
\end{equation}
Since $\alpha  \in \{ 1 - \delta,  1 \}$ and by our assumption, $\delta < 1 - \frac{1}{x}$, we obtain $\alpha x  > 1$. Thus, for $s = 1 - \beta_0 \leq \frac{1}{2}$, the Mean Value Theorem gives
\begin{equation}\label{use2J3}
    \frac{(\alpha x)^{1-\beta_0} -1}{1-\beta_0} \leq x^{1 - \beta_0} \log (\alpha x) \leq \sqrt{x} \log x.
\end{equation}
Since $\int_\alpha^{\alpha + \delta} h(t) \ dt = \frac{\delta}{2}$ and $\beta_0 \in (\tfrac{1}{2},1)$, we have  
\begin{equation}\label{use3J3}
    x^{1-\beta_0} \int_\alpha^{\alpha + \delta} h(t) t^{-\beta_0} \ dt 
    \leq \frac{\sqrt{x}}{\alpha} \frac{\delta}{2} \leq \frac{\sqrt{x}}{1-\delta} \frac{\delta}{2}.
\end{equation}
We conclude by combining \eqref{use1J3}, \eqref{use2J3}, \eqref{use3J3}, and the conditions $x \geq x_0$ and $\delta \leq \delta_0$.
\end{proof}
\subsection{Bounding $J^{(3)}$}
To provide bounds for $J^{(3)}(x)$ defined in \eqref{def-J3}, we use the zero-free region of $\zeta_L(s)$ as described in \thmref{thmal4}.
\begin{lemma}\label{boundingJ3}
Let $\delta \leq \delta_0 < 1,$ and $x_0\geq 2$. For all $x \geq x_0$, we have 
\begin{equation*}
J^{(3)}(x) \leq 
\ell_3 (\log d_L)^2 x^{\frac{1}{2}}
\end{equation*}
with 
\begin{equation}\label{def-el3}
\ell_3 = \ell_3(\delta_0) = \alpha_4 \Big(\frac{2 + \delta_0}{2} + \frac{1}{\sqrt{x_0}} \Big) \frac{\alpha_0(1/2)}{2}.
\end{equation}
where 
$\alpha_0$ is given in \eqref{def-alpha0(T)} (with values in Table \ref{Table-NL-1-2}) and 
$\alpha_4$ is defined in \eqref{akalpha4}.
\end{lemma}
\begin{proof}
The equations \eqref{def-H(s)} and \eqref{Mrosser} gives $|H(\varrho)| \leq \frac{2 + \delta}{2} \frac{1}{|\varrho|}$. Using this in \eqref{def-J3}:
\begin{equation}\label{j3usem2}
  J^{(3)}(x) \leq \sum_{\overset{\varrho \in Z(\zeta_L), \varrho \neq 1- \beta_0}{|\varrho| < \frac{1}{2}}} \Big( \frac{2 + \delta}{2} \Big| \frac{x^\varrho}{\varrho} \Big| + \Big| \frac{1}{\varrho} \Big| \Big) \leq \Big( \frac{2 + \delta}{2} \sqrt{x} + 1 \Big)  \sum_{\overset{\varrho \in Z(\zeta_L), \varrho \neq 1- \beta_0}{|\varrho| < \frac{1}{2}}}  \Big| \frac{1}{\varrho} \Big|,
\end{equation}
since $|\varrho| < \frac{1}{2}$ implies 
$ \Re(\varrho) < \frac{1}{2}$.
In the region $|\varrho| < \frac{1}{2}$ with $\varrho \neq 1 - \beta_0$, the zero-free region \eqref{ahneq1} shows that either $|\gamma| > \frac{1}{\alpha_4 \log d_L}$ or $\beta > \frac{1}{\alpha_4 \log d_L}$, which implies that $|\varrho| > \frac{1}{\alpha_4 \log d_L}$. Thus, using the symmetry of zeros of $\zeta_L(s)$ about the critical line $\Re(s) = 1/2$, we obtain
\begin{equation}\label{j3u3}
\sum_{\overset{\varrho \in Z(\zeta_L), \varrho \neq 1- \beta_0}{|\varrho| < \frac{1}{2}}}  \bigg| \frac{1}{\varrho} \bigg| \leq \alpha_4 (\log d_L) \sum_{\overset{\varrho \in Z(\zeta_L), \varrho \neq 1- \beta_0}{|\varrho| < \frac{1}{2}}}  1 \leq \alpha_4 (\log d_L)  \sum_{\overset{\varrho \in Z(\zeta_L), \varrho = \beta + i \gamma}{|\gamma| < \frac{1}{2}, \ 0 < \beta < \frac12}}  1  \leq \alpha_4 (\log d_L) \frac{ N_L ( 1/2 )}{2},
\end{equation} 
where the second inequality follows from the fact that the zeros $|\varrho| < 1/2$ are contained in the region $0 \leq \Re(s) < 1/2$.

We apply \thmref{thm-bnd-nchiNL} to bound $N_L(1/2)$:
\begin{equation}\label{j3u4}
    N_L(1/2) \leq 
     \alpha_0(1/2)  \log d_L .
\end{equation}
Combining \eqref{j3usem2}, \eqref{j3u3}, \eqref{j3u4} with the assumptions $x \geq x_0$ ad $\delta \leq \delta_0$ completes the proof.
\end{proof}
\subsection{Bounding $J^{(4)}$}
To provide bounds for $J^{(4)}(x)$ defined in \eqref{def-J4}, we use the explicit zero free region for $\zeta_L(s)$ as stated in Theorem \ref{thmR}. 
\begin{lemma}\label{boundingJ4}
Let $n_L \geq n_0 \geq 2$, $\delta \leq \delta_0 < 1,$ and $x_0\geq 2$. For all $x \geq x_0$, we have 
\begin{equation*}
J^{(4)}(x) \leq \ell_4 (\log d_L) x^{1 -\frac{1}{R_{1}n_L \log (4 \Delta_L )}},
\end{equation*}
where $R_1$ is defined in Theorem \ref{thmR}, $\Delta_L$ is given in \eqref{def-DeltaL}, 
\begin{equation}\label{def-el4}
\ell_4 = \ell_4(\delta_0, \mathscr{M}, n_0, R_1, x_0) 
= \left( \frac{2 + \delta_0}{2} \right) \left( 1 + x_0^{-1 + \frac{2}{R_{1} n_0 ( (1/\mathscr{M}) + \log 4)}} \right) \frac{\alpha_0(1)+\alpha_0'(2)}{2},
\end{equation}
and $\alpha_0(1)$ and $\alpha_0'(2)$ are given in \eqref{def-alpha0(T)} and \eqref{alpha_0'-N_L} respectively (with values in Tables \ref{Table-NL-1-2}).
\end{lemma}
\begin{proof}
Using $|H(\varrho)| \leq \frac{2 + \delta}{2} \frac{1}{|\varrho|}$ in \eqref{def-J4}, we have
\begin{equation}\label{use1J4}
    J^{(4)}(x) \leq \frac{2 + \delta}{2} \sum_{\overset{\varrho \in Z(\zeta_L), \varrho \neq \beta_0, |\varrho| \geq  \frac{1}{2}}{|\gamma| \leq 2}} \frac{x^{\beta}}{|\varrho|}
    . 
\end{equation}
Thus using $|\varrho| \geq  1/2$ for $|\gamma| \leq 1$ and $|\varrho| \geq  1$ for $1 < |\gamma| \leq 2$, we have
\begin{align}\label{j4u1-}
\sum_{\overset{\varrho \in Z(\zeta_L), \varrho \neq \beta_0, |\varrho| \geq  \frac{1}{2}}{|\gamma| \leq 2}} \frac{x^{\beta}}{|\varrho|} 
& \leq 2 \sum_{\overset{ \varrho \in Z(\zeta_L), \beta \notin \{ \beta_0,1-\beta_0 \}}{|\gamma| \leq 1}} x^\beta + \sum_{ \overset{\varrho \in Z(\zeta_L)}{1 < |\gamma| \leq 2}} x^\beta.  
\end{align}

For $\varrho = \beta + i \gamma$ with $\varrho \neq \beta_0$, \eqref{ahneq2} implies $\frac{1}{R_{1,L} \log (4 \Delta_L )} \leq \beta \leq 1 - \frac{1}{R_{1,L} \log (4 \Delta_L )}$ when $|\gamma| \leq 2$. Next, we use  the symmetry of zeros of $\zeta_L(s)$ about the critical line $\Re(s)=\frac{1}{2}$ and the fact that $x^\beta + x^{1 - \beta}$ takes its maximum value at its extremal values.
Thus 

\begin{align}
\notag
  \sum_{\overset{\varrho \in Z(\zeta_L), \varrho \neq \beta_0, |\varrho| \geq  \frac{1}{2}}{|\gamma| \leq 2}} \frac{x^{\beta}}{|\varrho|} 
  & \leq 2 \sum_{\overset{ \varrho \in Z(\zeta_L), \beta \notin \{ \beta_0,1-\beta_0 \}}{|\gamma| \leq 1}} \frac{x^\beta+x^{1-\beta}}{2} + \sum_{ \varrho \in Z(\zeta_L), 1 < |\gamma| \leq 2} \frac{x^\beta+x^{1-\beta}}{2} 
    \\ \notag & 
    \leq \left( x^{1 - \frac{1}{R_{1,L} \log (4 \Delta_L )}} + x^{\frac{1}{R_{1,L} \log (4 \Delta_L )}} \right) \left( N_L(1) + \frac{N_L(2) - N_L(1)}{2} \right) 
        \\ & 
    \leq x^{1 - \frac{1}{R_{1,L} \log (4 \Delta_L )}} \left( 1 + x^{-1 + \frac{2}{R_{1,L} \log (4 \Delta_L )}} \right) \left( \frac{N_L(1) + N_L(2)}{2} \right).
\label{j4u1}
\end{align}

Using the bounds for $N_L(1)$ given in \thmref{thm-bnd-nchiNL} and for $N_L(2)$ given in \corref{cor-NL(T)}, we obtain
\begin{equation}\label{use2J4}
    \frac{N_L(1) + N_L(2)}{2} \leq 
    \frac{\alpha_0(1)+\alpha_0'(2)}{2} (\log d_L).
\end{equation}
We conclude by combining \eqref{use1J4}, \eqref{j4u1}, \eqref{use2J4} with the assumptions $x \geq x_0$, $\delta \leq \delta_0$, $n_L \geq n_0$ and $\log \Delta_L \geq \frac{1}{\mathscr{M}}$.
\end{proof}
\subsection{Bounding $J^{(5)}$}
To provide bounds for $J^{(5)}(x,T)$ defined in \eqref{def-J5} and $J^{(6)}(x,T)$ defined in \eqref{def-J6}, we use the explicit zero free region for $\zeta_L(s)$ as shown in \eqref{lee} with $R_2$ defined in \thmref{thmR}. 
\begin{lemma}\label{boundingJ5}
Let $n_L \geq n_0 \geq 2, \delta \leq \delta_0 < 1, T_0 > 4$, and $x_0\geq  2$. For all $T \geq T_0$ and $x \geq x_0$, we have 
\begin{equation}
J^{(5)}(x,T) \leq \ell_5 (\log d_L) (\log T)^2 x^{1 -\frac{1}{R_{2}n_L \log (\Delta_L  T)}},
\end{equation}
where $R_{2}$ is satisfying Theorem \ref{thmR}, $\Delta_L$ is given in \eqref{def-DeltaL}, and 
\begin{multline}
\label{def-el5}
\ell_5 
= \ell_5(\delta_0,\mathscr{M}, n_0, R_2, T_0, x_0) 
= 
\frac{2 + \delta_0}{4}  \left( 1 + x_0^{-1 + \frac{2}{R_{2} n_0 ((1/\mathscr{M}) + \log T_0)}} \right) 
\bigg[ 
\frac{(\log(T_0 - 1)) - 1}{\pi ( \log T_0)^2} 
\\
+  \frac{\alpha_1 T_0}{(\log T_0)^2(T_0 -1)} 
+ \frac{T_0 +1}{\pi(T_0 -1) (\log T_0)^2}  
+ \mathscr{M} \bigg( 
\frac{1}{2\pi} 
+ \frac{(T_0 +1) \log (T_0 +1)}{\pi (T_0 -1) (\log T_0)^2} 
+ \frac{\alpha_1 \log (T_0+1)}{(T_0-1) (\log T_0)^2}
\\ + \frac{1}{n_0 (\log T_0)^2} \frac{\alpha_3 T_0}{T_0 -1}
+ \frac{1}{(\log T_0)^2} \Big( \frac{0.683}{\pi} 
+ 0.92 \alpha_1
- \frac{\log(2 \pi e)}{\pi} \Big( \frac{T_0+1}{T_0-1} - \log 2 \Big) 
+ \frac{\log(\pi e)}{\pi} 
\\+ \frac{\alpha_1 (\log 2)}{2} 
+  \frac{\alpha_2 T_0}{T_0-1} 
 \Big) 
  \bigg) \bigg],
\end{multline}
with $\alpha_1, \alpha_2, \alpha_3$ defined in \eqref{valalphai}. 
\end{lemma}
\begin{proof}
Using $|H(\varrho)| \leq \frac{2 + \delta}{2} \frac{1}{|\varrho|}$ in \eqref{def-J5}, we obtain
\begin{equation}\label{use1J5-}
  J^{(5)}(x,T) \leq \frac{2 + \delta}{2} \sum_{\overset{\varrho \in Z(\zeta_L)}{2 < |\gamma| < T}} \frac{x^\beta}{|\varrho|} \leq \frac{2 + \delta}{2} \sum_{\overset{\varrho \in Z(\zeta_L)}{2 < |\gamma| < T}} \frac{x^\beta}{|\gamma|}.
\end{equation}
For $\varrho = \beta + i \gamma$ with $\varrho \neq \beta_0$, \eqref{lee} implies $\frac{1}{R_{2,L} \log (\Delta_L T )} \leq \beta \leq 1 - \frac{1}{R_{2,L} \log ( \Delta_L  T)}$ when $2<|\gamma| < T$. Next, using  the symmetry of zeros of $\zeta_L(s)$ about the critical line $\Re(s)=\frac{1}{2}$ and the fact that $x^\beta + x^{1 - \beta}$ takes its maximum value at its extremal values, we have
\begin{equation}\label{use1J5}
\begin{split}
  J^{(5)}(x,T) & \leq \frac{2 + \delta}{4}  \sum_{\overset{\varrho \in Z(\zeta_L)}{2 < |\gamma| < T}} \frac{x^\beta + x^{1-\beta}}{|\gamma|}   \\
  & \leq \frac{2 + \delta}{4} x^{1 - \frac{1}{R_{2,L} \log (\Delta_L  T)}} \left( 1 + x^{-1 + \frac{2}{R_{2,L} \log (\Delta_L  T)}} \right) \sum_{\overset{\varrho \in Z(\zeta_L)}{2 < |\gamma| < T}} \frac{1}{|\gamma|}.
  \end{split}
\end{equation}
We split the interval $[2,T)$ into blocks $[k,k+1)$ with $k$ ranging from $2$ to $\lceil T-1 \rceil$. For each block $[k,k+1)$, we use $|\gamma| \geq  k$ to obtain
\begin{equation}\label{divideasum}
\sum_{\overset{\varrho \in Z(\zeta_L)}{2 < |\gamma| < T}} \frac{1}{|\gamma|} \leq \sum_{k=2}^{\lceil T-1 \rceil} \frac{N_L(k+1)-N_L(k)}{k} \leq \frac{N_L(T+1)}{T-1} + \sum_{k=3}^{\lceil T-1 \rceil} \frac{N_L(k)}{k(k-1)} - \frac{N_L(2)}{2}.
\end{equation}
By \thmref{theohsw}, we obtain
\begin{equation}\label{NLuselow}
    N_L(2) \geq  \frac{2}{\pi} \log \Big( \frac{d_L}{(\pi e)^{n_L}} \Big)  - \alpha_1 n_L (\log 2) - \alpha_1 \log d_L - \alpha_2 n_L - \alpha_3. 
\end{equation}
Using the bound for $N_L$ defined in \eqref{NLuse2}, we have
\begin{multline}
\label{usebet}
\sum_{k=3}^{\lceil T-1 \rceil} \frac{N_L(k)}{k(k-1)} 
 \leq 
 \frac{1}{\pi} \log \Big( \frac{d_L}{(2\pi e)^{n_L}} \Big) \sum_{k=3}^{\lceil T-1 \rceil} \frac{1}{k-1}  
 + \frac{n_L}{\pi} \sum_{k=3}^{\lceil T-1 \rceil} \frac{\log k}{k-1} 
 + \alpha_1 n_L \sum_{k=3}^{\lceil T-1 \rceil} \frac{\log k}{k(k-1)} 
 \\ + (\alpha_1 \log d_L + \alpha_2 n_L + \alpha_3) \sum_{k=3}^{\lceil T-1 \rceil} \frac{1}{k(k-1)}.
\end{multline}
We also obtain
\begin{equation}
\begin{split}
\label{sumbo}
	& \log (T-1) - \log(2) \leq \sum_{k=3}^{\lceil T-1 \rceil} \frac{1}{k-1} \leq  \log (T-1), 
    \ \sum_{k=3}^{\lceil T-1 \rceil} \frac{\log k}{k-1} \leq \frac{(\log T)^2}{2} + 0.683, 
    \\& \sum_{k=3}^{\lceil T-1 \rceil} \frac{\log k}{k(k-1)} \leq 0.92,  
	\ \textnormal{ and } \sum_{k=3}^{\lceil T-1 \rceil} \frac{1}{k(k-1)} \leq \frac{1}{2},
\end{split}
\end{equation}
where the first two inequalities follow from comparing the sum to its corresponding lower and upper integral, and the next inequality follows from comparing the sum to the integral of $(\log t)/t$ and estimating the error using direct computation. The final two results are obtained using direct computation.

Using \eqref{NLuse2} for $N_L(T+1)$, \eqref{NLuselow}, \eqref{usebet} and the inequalities in \eqref{sumbo},
we have
\begin{align}\label{eqimp}
 & \frac{N_L(T+1)}{T-1} + \sum_{k=3}^{\lceil T-1 \rceil} \frac{N_L(k)}{k(k-1)} - \frac{N_L(2)}{2} 
 \leq \Bigg( \frac{\log (T -1)}{\pi} - \frac{1}{\pi} + \frac{\alpha_1 T}{T-1} + \frac{T+1}{\pi (T -1)} \Bigg) \log d_L \notag \\
 & + \Bigg( \frac{1}{\pi} \Big( \frac{(\log T)^2}{2} + 0.683 \Big) + \frac{(T+1) \log (T+1)}{\pi (T-1)} + \frac{\alpha_1 \log(T+1)}{T-1} +  0.92 \alpha_1 - \frac{\log(2 \pi e)}{\pi} \Big( \frac{T+1}{T-1} - \log 2 \Big)  \notag \\
 & + \frac{\log(\pi e)}{\pi} + \frac{\alpha_1 (\log 2)}{2} - \frac{\log (2 \pi e)}{\pi} \log (T-1) + \frac{\alpha_2 T}{T-1} \Bigg) n_L +  \frac{\alpha_3 T}{T-1}. 
\end{align}
Notice that, taking the factor $(\log T)^2$ out from the right-hand side above, the remaining term decreases in $T$ when $T > 4$. Thus, we complete the proof by combining \eqref{use1J5}, \eqref{divideasum}, and \eqref{eqimp} with the assumptions $x \geq x_0$, $\delta \leq \delta_0$, $T \geq T_0$, $n_0 \leq n_L \leq (\log d_L) \mathscr{M}$ and $\log \Delta_L \geq \frac{1}{\mathscr{M}}$.
\end{proof}
\subsection{Bounding $J^{(6)}$}
This preliminary step provides an upper bound for $J^{(6)}(x,T)$ as defined in \eqref{def-J6}, by expressing it as the product of a term that depends only on the smoothing weight and a sum over the zeros of $\zeta_L(s)$, which is independent of that weight.
\begin{lemma}\label{boundingxJ_6}
Let $n_L \geq n_0 \geq 2, \delta \leq \delta_0 < 1, m \geq 1, T_0 > 4$, and $x_0\geq  2$. For all $T \geq T_0$ and $x \geq x_0$, we have 
\begin{equation*}
J^{(6)}(x,T) \leq \frac{M(\delta, m)}{2 \delta^m} \bigg(  x^{\frac{1}{R_{2}n_L \log (\Delta_L  T)} } S_L(m,T,1) + x S_L(m,T,x) \bigg),
\end{equation*}
where $M(\delta,m)$ is defined in \eqref{Mrosser}, 
$R_2$ in Theorem \ref{thmR}, $\Delta_L$ in \eqref{def-DeltaL}, and
\begin{equation}\label{def-Sm}
    S_L(m,T,x)  = \sum_{\overset{\varrho \in Z(\zeta_L)}{|\gamma| \geq  T}} \frac{x^{-\frac{1}{R_{2,L}  \log (\Delta_L  |\gamma|)}}}{|\gamma|^{m+1}}.
\end{equation}
\end{lemma}
\begin{proof}
Using  \eqref{def-H(s)} for $k=m$, we have $|H(\varrho)| \leq \frac{M(\delta,m)}{\delta^m |\varrho|^{m+1}}$ with $M(\delta,m)$ defined in \eqref{Mrosser}. Thus using $|\varrho| = |\beta + i \gamma| \geq  |\gamma|$, we obtain
\begin{equation}\label{new1}
  J^{(6)}(x,T) \leq \frac{M(\delta,m)}{\delta^m} \sum_{\overset{\varrho \in Z(\zeta_L)}{|\gamma| \geq  T}} \frac{x^\beta}{|\gamma|^{m+1}}.
\end{equation}
For $\varrho = \beta + i \gamma$, \eqref{lee} implies $\frac{1}{R_{2,L} \log (\Delta_L |\gamma| )} \leq \beta \leq 1 - \frac{1}{R_{2,L} \log ( \Delta_L  |\gamma|)}$ when $|\gamma| > 2$. Using  the symmetry of zeros of $\zeta_L(s)$ about the critical line $\Re(s)=\frac{1}{2}$ and the fact that $x^\beta + x^{1 - \beta}$ takes its maximum value at its extremal values as given in Theorem \ref{thmR}, we have
\begin{equation}\label{new2}
    \sum_{\overset{\varrho \in Z(\zeta_L)}{|\gamma| \geq  T}} \frac{x^\beta}{|\gamma|^{m+1}} = \frac{1}{2} \sum_{\overset{\varrho \in Z(\zeta_L)}{|\gamma| \geq  T}} \frac{x^\beta + x^{1-\beta}}{|\gamma|^{m+1}} 
    \leq \frac{1}{2} \sum_{\overset{\varrho \in Z(\zeta_L)}{|\gamma| \geq  T}} \frac{x^{\frac{1}{R_{2,L} \log (\Delta_L  |\gamma|)}}}{|\gamma|^{m+1}}
     + \frac{1}{2}  \sum_{\overset{\varrho \in Z(\zeta_L)}{|\gamma| \geq  T}} \frac{x^{1-\frac{1}{R_{2,L} \log (\Delta_L  |\gamma|)}}}{|\gamma|^{m+1}}.
\end{equation}
Since $|\gamma| \geq  T > 2$ and $\Delta_L  \geq  \sqrt{3}$ \eqref{def-Minkowski}, therefore $x^{\frac{1}{R_{2,L} \log (\Delta_L  |\gamma|)}} \leq x^{\frac{1}{R_{2,L} \log (\Delta_L  T)}}$.
Finally, we conclude by combining \eqref{new1} and \eqref{new2}. 
\end{proof}
\subsection{A preliminary bound for the error term estimating $I_{L/K}(x)$}
We recall that $I_{L/K}$ is defined as in \eqref{newI}, that $\beta_0$ denotes the possible real exceptional zero of $\zeta_L(s)$, and that $a_{\beta0}=1$ or $2$ as defined in \eqref{abeta0}. 
\begin{proposition}\label{boundI-H}
Let $n_L \geq n_0 \geq 2$, $x_0\geq  4, \delta \leq \delta_0 < 1 - \frac{\sqrt{2}}{x_0}, m \geq 1$, and $T \geq T_0> 4$.
For all $x \geq x_0$, we have 
\begin{equation*}
\begin{split}
\frac{\left|I_{L/K}(x) - \frac{|C|}{|G|} x\right| }{\frac{|C|}{|G|} }
 \leq 
 &  \frac{x^{\beta_0 }}{\beta_0} 
 + \frac{\delta}{a_{\beta_0}} x 
 + \ell_1 (\log d_L) (\log x) 
 + \ell_2 (\log x) x^\frac{1}{2} 
 + \ell_3 (\log d_L)^2 x^\frac{1}{2} 
 \\& + \ell_4 (\log d_L) x^{1-\frac{1}{R_{1,L} \log (4 \Delta_L )}} 
+ \ell_5 (\log d_L) (\log T)^2 x^{1- \frac{1}{R_{2,L} \log(\Delta_L  T)}} 
\\&
+ \frac{M(\delta, m)}{2 \delta^m} \Big( x^{\frac{1}{R_{2,L} \log (\Delta_L  T)}} S_L(m,T,1) + x S_L(m,T,x) \Big) ,
\end{split}
\end{equation*}
where $\ell_1,\ell_2,\ell_3,\ell_4$ and $\ell_5$ are defined in \eqref{def-el1}, \eqref{def-el2}, \eqref{def-el3}, \eqref{def-el4} and \eqref{def-el5} respectively, $M(\delta, m)$ is defined in \eqref{Mrosser}, and $S_L(m,T,x)$ is defined in \eqref{def-Sm}.
\end{proposition}
Note that the bound obtained is valid for $I_{L/K}$ 
for both weights $h$ defined by $\alpha = 1 - \delta$ and $\alpha = 1$. 
\begin{proof}
In the estimate \eqref{def-|I|modified} of \corref{defji}, we isolate the main term $\frac{|C|}{|G|} x$, so that 
\begin{equation*}
\left|I_{L/K}(x) - \frac{|C|}{|G|} x \right|
\leq \left|I_{L/K}(x) -\frac{|C|}{|G|} x H(1) \right| + \frac{|C|}{|G|}x\left| H(1) -1\right|
\end{equation*}
As the term $x^{\beta_0}H(\beta_0)$ arises in the estimate of $ |I_{L/K}(x) - \frac{|C|}{|G|} x H(1)| $, we combine it with $x|H(1) -1|$. By \eqref{def-H(1)} and \eqref{def-H(s)} for $H$ and \eqref{Mrosser} for $M(\delta,0)$, we get
\begin{equation*}
x^{\beta_0-1} H(\beta_0) + |H(1) -1| 
\leq x^{\beta_0-1} \frac1{\beta_0} \left(1+\frac{\delta}2\right) + \frac{\delta}2 
= \frac{x^{\beta_0 -1}}{\beta_0} + \left( \frac{x^{\beta_0 -1}}{\beta_0} + 1 \right)\frac{\delta}{2}.
\end{equation*}
Finally, the fact that $\frac{x^{\beta_0-1}}{\beta_0} \leq 1$ is valid for $x \geq 4$ gives 
\begin{equation*}\label{uses1}
x^{\beta_0-1} H(\beta_0) + |H(1) -1| 
\leq \begin{cases}
\frac{x^{\beta_0 -1}}{\beta_0} + \delta
& \textrm{ if an exceptional }\beta_0\  \textrm{exists},\\
\frac{\delta}2
& \textrm{ otherwise.}
\end{cases}
\end{equation*}
We conclude by combining the above with \corref{defji}, \lmaref{bouJ0J1}, \lmaref{bouJ2}, \lmaref{boundingJ3}, \lmaref{boundingJ4}, \lmaref{boundingJ5} and \lmaref{boundingxJ_6}.
\end{proof}
\section{Estimating the sums $S_L(m,T,x)$ over the zeros of $\zeta_L(s)$}

Let $R_{2,L} $ be as given in the zero-free region for the Dedekind zeta functions $\zeta_L(s)$ stated in \thmref{thmR}.
We introduce the function
\begin{equation}\label{newphiu}
\phi_{m,x}(u) = \frac{x^{-\frac{1}{R_{2,L} (\log (\Delta_L  u))}}}{u^{m+1}} 
= \frac{1}{u^{m+1}} \exp \left( -\frac{ \log x}{R_{2,L}  \log (\Delta_L  u) } \right),
\end{equation}
and rewrite \eqref{def-Sm} as 
\begin{equation}\label{def-SLmTx}
S_L(m,T,x) = \sum_{\overset{\varrho \in Z(\zeta_L)}{|\gamma| \geq  T}} \phi_{m,x}(|\gamma|).
\end{equation}
In particular
\begin{equation*}\label{def-SmT1}
S_L(m,T,1) = \sum_{\overset{\varrho \in Z(\zeta_L)}{|\gamma| \geq  T}} \phi_{m,1}(|\gamma|),
\ \text{where}\  \phi_{m,1}(u) = \frac1{u^{m+1}}.
\end{equation*}
To study these sums over the zeros of $\zeta_L(s)$, we need  information on the density of these zeros. 
For $u \geq t \geq 1$, we consider $N_L(u) - N_L(t)$, where we recall that $N_L(T)$ is the number of zeros $\varrho = \beta + i \gamma$ of $\zeta_L(s)$ in the region $0 < \beta < 1$ and $|\gamma| \leq T$.
Explicit bounds for $N_L$ are stated in Theorem \ref{theohsw}, and we deduce
\begin{equation}
  \label{NLQ}
N_L(u) - N_L(t) \leq Q(u,t),
\end{equation}
where 
\begin{equation*}\label{def-Qut}
Q(u,t) = P_L(u) - P_L(t) + E_L(u) + E_L(t),
\end{equation*}
with $P$ and $E$ defined in \eqref{tteq}. 
A simplification yields
\begin{equation}
\label{Qtu}
 Q(u,t) 
= \frac{ n_L u }{\pi} \log \bigg(  \frac{\Delta_L  u}{2 \pi e}  \bigg) - \frac{ n_L t }{\pi} \log \bigg(  \frac{\Delta_L  t}{2 \pi e}  \bigg)  
+ 2 \alpha_1 n_L ( \log (\Delta_L  \sqrt{ut}) ) + 2 \alpha_2n_L + 2 \alpha_3
\end{equation}
where the $\alpha_i$'s are defined in \eqref{valalphai}. 
It will also be useful to write $Q(u,t)$ in the form
\begin{equation}
\label{Qtu2}
 Q(u,t) = \Big( \frac{u-t}{\pi} + 2 \alpha_1 \Big) \log d_L + \Big( \frac{u}{\pi} \log \Big( \frac{u}{2 \pi e} \Big) - \frac{t}{\pi} \log \Big( \frac{t}{2 \pi e} \Big) + \alpha_1 \log (u t) + 2 \alpha_2 \Big) n_L + 2 \alpha_3.   
\end{equation}
Combining partial summation with the facts $N_L(u) - N_L(T) \ll u \log u$ and $\phi_{m,x}(u) \ll  u^{-2}$ (as $m \geq 1$), leads to rewriting $ S_L(m,T,x)$ as the Stieltjes integral
\begin{equation*}\label{new9}
 S_L(m,T,x) 
 = \int_T^\infty \phi_{m,x}(u) \  d(N_L(u) - N_L(T)) 
     = \int_T^\infty (N_L(u) - N_L(T)) (-\phi'_{m,x}(u)) \ du.
\end{equation*}
A study of the sign of $\phi'_{m,x}(u)$ is necessary to study $S_L(m,T,x)$ further. 
Since
$$ 
\phi'_{m,x}(u) = \frac{\frac{\log x}{R_{2,L}(\log(\Delta_L  u))^2}-(m+1)}{u^{m+2}} \exp \bigg( \frac{- \log x}{R_{2,L} (\log (\Delta_L  u))} \bigg),
$$ 
then 
$\phi'_{m,1}(u)$ is always negative, while $\phi'_{m,x}(u) \geq 0$ for $u \leq W$ and that $\phi'_{m,x}(u) < 0$ for $u > W$, where we denote
\begin{equation*}
 \label{def-W} 
  W = \frac{1}{\Delta_L } \exp \bigg( \small\sqrt{\frac{\log x}{R_{2,L}(m+1)}} \bigg).
\end{equation*}
 Dropping the portion of the integral where $-\phi'_{m,x}(u) \leq 0$ and applying \eqref{NLQ}, we obtain
\begin{equation}\label{new10}
    S_L(m,T,x) \leq \int_{\max \{ T,W \}}^\infty Q(u,T) (- \phi'_{m,x}(u)) \ du.
\end{equation}
\subsection{Bounding \texorpdfstring{$S_L(m,T,1)$}{}}
At this point, we have enough information to establish a bound for $S_L(m,T,1)$ as given by \eqref{def-SLmTx}:
\begin{lemma}\label{bnd-SL1-exp}
Let $n_L \geq n_0 \geq 2,  m \geq 1, T_0 > 4$, and $x_0\geq  2$. For all $T \geq T_0$, we have 
\begin{equation*}
 S_L(m,T,1) \leq \ell_6 (\log d_L) \frac{\log T}{T^m},
 \end{equation*}
where
\begin{equation}\label{def-el6}
\ell_6 =\ell_6(m, \mathscr{M} , T_0) =
 \frac{1}{m \pi} \Big( \mathscr{M} + \frac{1}{\log T_0} \Big) + \frac{2 \mathscr{M} \alpha_1}{T_0} + \frac{ 2 \alpha_1 + \mathscr{M} \big( \tfrac{\alpha_1}{m+1} + 2 \alpha_2 + \alpha_3 \big) }{(\log T_0) T_0},
\end{equation}
where $\alpha_1, \alpha_2$ and $\alpha_3$ are as defined in \eqref{valalphai}.
\end{lemma}
\begin{proof}
We rewrite \eqref{new10} by inserting the definition \eqref{Qtu2} of $Q(u,t)$ to find that $    S_L(m,T,1) $ is bounded above by 
\begin{align*}
& (m+1) \bigg[
\bigg( \frac{1}{\pi} \int_T^\infty \frac{du}{u^{m+1}} + \Big( 2 \alpha_1 - \frac{T}{\pi} \Big) \int_T^\infty \frac{du}{u^{m+2}} \bigg) (\log d_L)  
   +
\bigg( \frac{1}{\pi} \int_T^\infty \frac{\log u}{u^{m+1}} du  
- \frac{\log (2 \pi e)}{\pi} \int_T^\infty \frac{du}{u^{m+1}} 
\\ \bigg. & \bigg. 
+ \Big( - \frac{T}{\pi} \log \Big( \frac{T}{2 \pi e} \Big) + \alpha_1 (\log T) + 2 \alpha_2 \Big) \int_T^\infty \frac{du}{u^{m+2}}  
 + \alpha_1 \int_T^\infty \frac{\log u}{u^{m+2}} du \bigg) n_L 
+ 2  \alpha_3 \int_T^\infty \frac{du}{u^{m+2}} 
\bigg].
\end{align*}
We directly calculate the integrals
$$    \int_T^\infty \frac{1}{u^{k+1}} \ du = \frac{1}{k T^k}, 
         \int_T^\infty \frac{\log u}{u^{k+1}} \ du = \frac{ \log T}{k T^k} + \frac{ 1}{k^2 T^k}
$$ 
and obtain, after re-ordering the terms, that
\begin{align*}
S_L(m,T,1) 
\leq & 
\frac{n_L}{\pi}  \frac{1}{m } \frac{ \log T}{T^m}
+ \Big(
\frac{ \log d_L}{m\pi} 
+ \frac{n_L}{m\pi} \Big( \frac{m+1}{m} - \log ( 2 \pi e ) \Big)
  \Big)  \frac{1}{T^{m}}
\\& + 2 \alpha_1 n_L\frac{\log T}{T^{m+1}}  
+ \Big( \frac{\alpha_1 n_L}{m+1} 
 +  2 \alpha_1 \log d_L + 2 \alpha_2 n_L+ 2 \alpha_3 \Big) \frac{1}{T^{m+1}}.
\end{align*}
We conclude with $n_L\leq \mathscr{M} \log d_L$ from \eqref{def-Minkowski} and $\frac{m+1}{m } - \log (2 \pi e) \leq 0$.
\end{proof}
\subsection{Bounding \texorpdfstring{$S_L(m,T,x)$}{}}
We recall the definition \eqref{def-SLmTx} for $S_L(m,T,x)$ and denote
\begin{equation}
 \label{def-XLmT}
  X_{L,m,T} = (m+1) R_{2,L} \log^2 ( \Delta_L  T).
\end{equation}
\begin{lemma}\label{pres2}
Let $n_L \geq n_0 \geq 2, m \geq 1$, and $T_0> 4$.
For all $T \geq T_0$ and $x \geq  2$, we have 
\begin{equation}
S_L(m,T,x) \leq Q(T_{1,L},T) \phi_{m,x}(T_{1,L}) + \int_{T}^\infty  \frac{\partial Q}{\partial u} (u,T) \phi_{m,x}(u) \ du ,
\end{equation}
where $\phi_{m,x}$ and $Q$ are defined in \eqref{newphiu} and \eqref{Qtu} respectively, 
\begin{equation}
\label{T1}
    T_{1,L} =  T_{1,L}(m,R_2,x) = \begin{cases}
    T & \textnormal{ if } \log x \leq X_{L,m,T} ,\\
    W = \frac{1}{\Delta_L } \exp \Big( \sqrt{\frac{\log x}{R_{2,L}(m+1)}} \Big) & \textnormal{ if } \log x > X_{L,m,T},
    \end{cases}
\end{equation}
and $\Delta_L $, $R_{2,L}$, and $X_{L,m,T}$ are defined in \eqref{def-DeltaL}, Theorem \ref{thmR},  and
 \ref{def-XLmT} respectively. 
\end{lemma}
\begin{proof}
Observe that we have
\begin{equation*}
  \label{equivalence}
   \log x > X_{L,m,T} \Longleftrightarrow W > T. 
\end{equation*}
So, with $T_{1,L}$ as defined above, we have that $\max \{ T,W \}=T_{1,L}$.
Integrating by parts \eqref{new10} yields
$$ 
    S_L(m,T,x) \leq Q(T_{1,L},T) \phi_{m,x}(T_{1,L}) + \int_{T_{1,L}}^\infty \frac{\partial Q}{\partial u} (u,T)   \phi_{m,x}(u) \ du  .
$$ 
Since 
\begin{equation}
  \label{partialQuT}
\frac{\partial Q}{\partial u} (u,T)  
= \frac{n_L}{\pi} \log \Big( \frac{\Delta_L  u}{2 \pi} \Big) + \frac{\alpha_1 n_L}{u} ,
\end{equation}
it follows from $\Delta_L  \geq \sqrt{3}$, $n_L \geq 2$,  and $u \geq T \geq T_0 \geq 4$, that 
$ \frac{\partial Q}{\partial u} (u,T) \geq \frac{2}{\pi} \log( \frac{4 \sqrt{3}}{2 \pi}) >0$. The last integrand being positive with $T \leq T_{1,L}$ concludes the proof.
\end{proof}
To complete the bound for $S_L(m,T,x)$, it remains to bound $Q(u,t),  \frac{\partial Q}{\partial u} (u,T)  $ and estimate the consequent integral 
$
\int_{T}^\infty \frac{\partial Q}{\partial u} (u,T)   \phi_{m,x}(u) du.
$
\subsubsection{Preliminary lemmas on $Q(u,t)$}
The next lemma provides a bound for $Q(u,t)$ similar to \cite[Lemma 2.15]{mb} for Dirichlet $L$-functions. 
\begin{lemma}\label{mlu}
Let $Q(u,t)$ be defined as in \eqref{Qtu} and $\omega_0 \geq 1$. Then there exists $t_0=t_0(\omega_0)$ such that for all t satisfying $t_0 \leq t \leq u$, then 
\begin{equation}
  \label{Qutinequality}
Q(u,t) < \omega_0 \frac{u n_L}{\pi} \log (\Delta_L  u).
\end{equation}
Table \ref{Table-t0-omega0} in Appendix \ref{appendixtables} provides admissible values for $ (\omega_0,t_0)$.
\end{lemma}
\begin{proof}
Let $ \nu(\omega,\Delta_L ,t,u) = \frac{\pi}{n_L} \Big( \omega \frac{u n_L}{\pi} \log ( \Delta_L  u) - Q(u,t) \Big) $.
We have 
$$ 
\frac{\partial \nu}{\partial u} 
= 
(\omega-1) \log ( \Delta_L  u) + \omega +\log (2 \pi)- \frac{\alpha_1 \pi }{u}
\geq 1+\log (2 \pi)- \frac{\alpha_1 \pi }{u}
\geq 0
$$ 
for 
$\omega \geq 1$ and $u\geq \frac{\alpha_1 \pi}{1+\log (2 \pi)}=0.2524\ldots$.
Thus, for $u\geq t$, 
\begin{align*}
 \nu(\omega,\Delta_L ,t,u) \geq \nu(\omega,\Delta_L ,t,t) 
&= \frac{\pi}{n_L} \Big( \omega \frac{t n_L}{\pi} \log ( \Delta_L  t) - Q(t,t) \Big) 
\\& = \frac{\pi}{n_L} \Big( \omega \frac{t n_L}{\pi} \log ( \Delta_L  t) -  2 \alpha_1 n_L ( \log (\Delta_L  t) ) - 2 \alpha_2n_L - 2 \alpha_3
 \Big) 
\\& \geq \pi  \Big( \log ( \sqrt3 t)
\Big( \omega \frac{t }{\pi}  -  2 \alpha_1 \Big) - 2 \alpha_2 -  \alpha_3 
 \Big) 
\end{align*}
which is non-negative for values of $(\omega,t)$ satisfying
\begin{equation}\label{cond-omega-t}
 \omega
\geq \frac{\pi}{t }  \Big( 2 \alpha_1 + \frac{2 \alpha_2 +  \alpha_3}{\log ( \sqrt3 t)} \Big) .
\end{equation}
Sample of values are displayed in Table \ref{Table-t0-omega0}.
From this, we deduce \eqref{Qutinequality}.  
\end{proof}
\begin{lemma}\label{MTT_0} 
Let $m \geq 1, T\geq T_0 \geq t_0> 4$, $(t_0,\omega_0)$ satisfying \lmaref{mlu}, and $x \geq  2$.
For $T_{1,L}$, $\phi_{m,x}$ and $Q$ as defined in \eqref{T1}, \eqref{newphiu} and \eqref{Qtu} respectively, we have 
\begin{equation*}
Q(T_{1,L},T) \phi_{m,x}(T_{1,L}) 
\leq \begin{cases}
\frac{2 E_L(T)}{T^{m+1}} e^{- \frac{ \log x}{R_{2,L} (\log \Delta_L  T)} }
 & \text{if}\ \log x \leq X_{L,m,T}, \\
 \frac{ \omega_0 n_L \Delta_L ^m}{\pi \sqrt{R_{2,L} (m+1)}} \,
 \sqrt{\log x}
 e^{ -(2m+1) \sqrt{\frac{\log x}{R_{2,L} (m+1)}} } & \text{if}\  \log x > X_{L,m,T},
 \end{cases}
\end{equation*}
where $X_{L,m,T}$ and $E_L(T)$ are defined in \eqref{def-XLmT} and \eqref{tteq} respectively.
\end{lemma}
\begin{proof}
In the case $\log x \leq X_{L,m,T}$, $T_{1,L} =T$ and thus $Q(T,T) =2E_L(T)$, giving the first inequality.
In the other case, $\log x > X_{L,m,T}$, then 
$T_{1,L} = W > T \geq   t_0$, and \lmaref{mlu} gives
\begin{equation}\label{p1}
    Q(T_{1,L},T) = Q(W,T) < \omega_0 \frac{W n_L}{\pi} \log ( \Delta_L  W) = \omega_0 
 \frac{n_L}{\pi \Delta_L } \sqrt{\frac{\log x}{R_{2,L} (m+1)}}  e^{ \sqrt{\frac{\log x}{R_{2,L} (m+1)}} }.
\end{equation}
Note that
\begin{equation}\label{p2}
    \phi_{m,x}(T_{1,L}) = \phi_{m,x}(W) = \frac1{W^{m+1}} e^{ -\frac{\log x}{R_{2,L} (\log \Delta_L  W )} } = \Delta_L ^{m+1} e^{ -2 \sqrt{\frac{(m+1) \log x}{R_{2,L}}} }.
\end{equation}
We conclude by combining \eqref{p1}, \eqref{p2}, and $2 \sqrt{m+1} - \frac{1}{\sqrt{m+1}} = \frac{2m +1}{\sqrt{m+1}}$.
\end{proof}
\subsubsection{Relating $S_L(m,T,x)$ to Bessel type integrals}

We introduce the following integral functions.
Given positive real numbers $n,m,\alpha, \beta$ and $l$, we define the Bessel-type integral
\begin{equation}\label{inm}
    I_{n,m} (\alpha, \beta; l) = \int_l^\infty \frac{(\log \beta u )^{n-1}}{u^{m+1}}  e^{ - \frac{\alpha}{\log \beta u} } \  du.
\end{equation} 
Moreover, given positive constants $n, z$, and $y$, we consider a type of Bessel integral  (see \cite[Page 376, equation 9.6.24]{ab}) given as the integral
\begin{equation}\label{def-K_nu}
    K_n(z, y) = \frac{1}{2} \int_y^\infty v^{n-1}  e^{ -\frac{z}{2} \left( v + \frac{1}{v} \right) } \ dv.
\end{equation}
These integrals are related through  the change of variable $v=(\log(\beta u)) \sqrt{\frac{m}{\alpha}}$:
\begin{equation*}\label{relation-IK}
I_{n,m} (\alpha, \beta; l) = 2 \beta^m \Big( \frac{\alpha}{m} \Big)^{n/2} K_n \Big( 2 \sqrt{\alpha m}, \sqrt{\frac{m}{\alpha}} \log (\beta l) \Big).
\end{equation*}
In particular, if $\alpha = \frac{\log x}{R_{2,L}}$, $\beta = \Delta_L $, and $l = T$, then
\begin{equation}\label{Inm}
    I_{n,m} \Big( \frac{\log x}{R_{2,L}}, \Delta_L ; T \Big) = 2 \Delta_L ^m \Big( \frac{\log x}{m R_{2,L}} \Big)^{n/2} K_n ( z_m , w_m ),
\end{equation}
where 
\begin{equation}\label{def-zm-wm}
z_m = 2 \sqrt{\frac{m \log x}{R_{2,L}}},
\ w_m = \sqrt{\frac{m R_{2,L}}{\log x}} \log (\Delta_L  T) .
\end{equation}
\begin{lemma}\label{usingbm}
Let $m\geq 1, T > 2,x \geq  2$, and $\phi_{m,x}$ and $Q$ be as defined in \eqref{newphiu} and \eqref{Qtu}.
We have
\begin{equation*}
   \int_{T}^\infty  \frac{ \partial Q}{ \partial u } (u,T) \phi_{m,x}(u) \ du 
    \leq  \frac{2}{\pi m R_{2} } \Delta_L ^m (\log x) K_2(z_m,w_m),
\end{equation*}
where $K_2(z_m,w_m)$ is defined in \eqref{def-K_nu} and \eqref{def-zm-wm}.
\end{lemma}
\begin{proof}
It follows from \eqref{partialQuT} and from $ \frac{\log ( 2 \pi)}{\pi} - \frac{\alpha_1 }{u} \geq 0$ for all $u \geq T > 2$, that 
$$
\frac{\partial Q}{\partial u} (u,T) 
\leq \frac{n_L}{\pi} \log (\Delta_L  u) .
$$ 
We observe that 
$$  \frac{n_L}{\pi} 
\int_T^\infty \log (\Delta_L  u) \phi_{m,x}(u) \ du
=\frac{n_L}{\pi} \, I_{2,m} \Big( \frac{\log x}{R_{2,L}}, \Delta_L  , T \Big) ,
$$ 
and complete the proof with \eqref{Inm}.
\end{proof}
Combining \lmaref{pres2}, \lmaref{MTT_0} and \lmaref{usingbm} leads to a bound for $S_L(m,T,x)$ as defined by \eqref{def-SLmTx} in terms of the Bessel-type integral $K_2$.
\begin{lemma}\label{bnd-B2}
Let $n_L \geq 2$ and let $(t_0,\omega_0)$ be as in \lmaref{mlu}.
Let $m\geq 1, T \geq T_0 \geq  t_0$ and  $X_{L,m,T}$ as defined in \eqref{def-XLmT}. 
For all $x$ satisfying $\log x > X_{L,m,T}$, we have
\begin{equation*}
S_L(m,T,x) \leq   \frac{\omega_0}{\pi \sqrt{R_2 (m+1)}} \, \Delta_L^m \sqrt{n_L} \, \sqrt{ \log x} e^{ -\frac{2m+1}{\sqrt{m+1}} \sqrt{\frac{\log x}{R_{2} n_L}} } 
+ \frac{2 }{\pi m R_2} \, \Delta_L^m \, (\log x) K_2(z_m,w_m),
\end{equation*}
where $R_2$ is satisfying Theorem \ref{thmR}, and $K_2(z_m,w_m)$ is defined in \eqref{def-K_nu} and \eqref{def-zm-wm}.
\end{lemma}
\subsubsection{Bounding $S_L(m,T,x)$ for $x$ is ``large enough"}

Recalling the definitions  \eqref{def-XLmT} and \eqref{def-zm-wm} for $X_{L,m,T}$ and $w_m$, 
we have that $\log x > X_{L,m,T}$ is equivalent to 
\begin{equation*}
0< w_m < \sqrt{\frac{m }{(m+1) }} . 
\end{equation*}
In this case, we can apply an exponentially decaying estimate  on so the called ``modified Bessel integrals of the second kind" $K_2(z,w)$ as established by Rosser and Schoenfeld in \cite[(2.35)]{RS75}  which is valid for large values of $z$ and values of $w$ close to $0$: \begin{equation}\label{bnd-K2-zw}
    K_2(z,w) \leq \sqrt{\frac{\pi}{2}} \frac{e^{-z}}{\sqrt{z}} \left( 1+\frac{15}{8z}+\frac{105}{128z^2} \right).
\end{equation}
This explicit bound for $K_2$ allows to finalize a bound for $S_L(m,T,x)$.
\begin{lemma}\label{bnd-SLT-exp}
Let $n_L \geq n_0 \geq 2$, $m\geq 1, T \geq T_0 \geq  t_0$ with $(t_0,\omega_0)$ as in \lmaref{mlu}, and  $X_{L,m,T}$ as defined in \eqref{def-XLmT}. 
For all $x$ satisfying $x \geq x_0$ and $\log x > X_{L,m,T}$, we have
\begin{equation*} 
S_L(m,T,x)\le
\ell_7  \Delta_L ^m \sqrt{n_L} (\log x)^{3/4} e^{-2 \sqrt{\frac{m \log x}{R_{2}n_L}} },
\end{equation*}
where $R_2$ is satisfying Theorem \ref{thmR} and 
\begin{multline}\label{def-el7}
 \ell_7  =  \ell_7(m, \mathscr{M}, R_2, t_0,T_0, \omega_0, x_0 ) = 
\frac{ \omega_0 }{\pi \sqrt{R_2 (m+1)}}  (\log x_0)^{-1/4}\, e^{ - \left(\frac{2m+1}{\sqrt{m+1}} - 2 \sqrt{m}  \right) \sqrt{(m+1) \left(\frac1{\mathscr{M}}+\log T_0\right)} } 
 \\ + \frac{2 n_0^{-1/4} }{\sqrt{\pi} m^{5/4} R_2^{3/4}}
  \left( 1+\frac{15 }{16 \sqrt{ m (m+1)} \left(\frac1{\mathscr{M}} + \log T_0\right) } + \frac{105 }{512 m (m+1) \left(\frac1{\mathscr{M}} + \log T_0\right)^2} \right)  .
\end{multline}
\end{lemma}
\begin{proof}
Inserting the expression \eqref{def-zm-wm} for $z_m$ in terms of $m,T$ and $x_0$ in \eqref{bnd-K2-zw} gives 
\begin{align*}
    K_2(z_m,w_m) 
& \leq \sqrt{\frac{\pi}{2}} \frac{e^{-2 \sqrt{\frac{m \log x}{R_{2}n_L}}}}{\sqrt{2 \sqrt{\frac{m \log x}{R_{2}n_L}}}} \left( 1+\frac{15}{82 \sqrt{\frac{m \log x}{R_{2}n_L}}}+\frac{105}{128(2 \sqrt{\frac{m \log x}{R_{2}n_L}}))^2} \right)
\\ 
& \leq  \frac{ \sqrt{\pi}  R_{2}^{1/4} n_L^{1/4}  e^{-2 \sqrt{\frac{m \log x}{R_{2}n_L}}} }{   m^{1/4} (\log x)^{1/4} } \left( 1+\frac{15\sqrt{R_{2}n_L}}{16 \sqrt{ m \log x }}+\frac{105 R_{2}n_L }{512 m \log x} \right).
\end{align*}
Since $\log x > X_{L,m,T}= (m+1) R_{2} n_L \log^2 ( \Delta_L  T)$,
we conclude
\begin{equation*}\label{bnd-K2}
 K_2(z_m,w_m) \leq 
 \frac{ \sqrt{\pi}  R_{2}^{1/4} }{   m^{1/4} } \Big( 1+\frac{15 }{16 \sqrt{ m (m+1)} \log ( \Delta_L  T) } + \frac{105 }{512 m (m+1) \log^2 ( \Delta_L  T)} \Big)
 \frac{ n_L^{1/4}  e^{-2 \sqrt{\frac{m \log x}{R_{2}n_L}} }}{ (\log x)^{1/4} } .
\end{equation*}
We now insert these bounds for $K_2$ into the bound for $S_L(m,Tx)$ established in \lmaref{bnd-B2}. We obtain
\begin{multline*}
S_L(m,T,x)\le
\Big[ \frac{ \omega_0 }{\pi \sqrt{R_2 (m+1)}}  (\log x)^{-1/4}\, e^{ - \left(\frac{2m+1}{\sqrt{m+1}} +2 \sqrt{m} \right) \sqrt{\frac{\log x}{R_{2}n_L}} }
 \Big. \\ \Big. + \frac{2 n_L^{-1/4} }{\sqrt{\pi} m^{5/4} R_2^{3/4}}
  \Big( 1+\frac{15 }{16 \sqrt{ m (m+1)} \log ( \Delta_L  T) } + \frac{105 }{512 m (m+1) \log^2 ( \Delta_L  T)} \Big) \Big] 
\sqrt{n_L} \Delta_L ^m (\log x)^{3/4} e^{-2 \sqrt{\frac{m \log x}{R_{2}n_L}} } .   
\end{multline*}
This bound is of size $(\log x)^{3/4} e^{-2 \sqrt{\frac{m \log x}{R_{2}n_L}} }$. We conclude by using $n_L \geq n_0, \log \Delta_L  \geq \frac1{\mathscr{M}}$, $T\geq T_0$, and 
$\frac{\log x}{R_2n_L}> \frac{X_{L,m,T}}{R_2n_L} > (m+1) \left(\frac1{\mathscr{M}}+\log T_0\right)$.
\end{proof}
\section{Explicit estimates for $\psi_C (x)$}
\label{section-estimate-EC(x)}
We recall the notations \eqref{Eh} for the error terms $E_C(x)$ and $\tilde{E}_C(x)$ corresponding to $\psi_C(x)$ and its weighted approximation $\widetilde{\psi}_C(x)$, respectively.
For the rest of the article, we are assuming the condition $$\log x > X_{L,m,T}.$$ 
Other cases are currently being investigated in a follow-up work. 
In this section, we first conclude a bound for $E_C (x)$ in terms of the parameters $T$ and $\delta$. We then choose $T$ and $\delta$ to balance different contributions of the error term and obtain a bound as small as possible. 
\subsection{Explicit estimates for $\psi_C (x)$
in terms of parameters \texorpdfstring{$T$}{} and \texorpdfstring{$\delta$}{}}
\label{Section71}
\begin{proposition}\label{prop-bnd-error-delta-L-m-T-x}
Let $n_L \geq n_0 \geq 2$ and $\zeta_L(s)$ be the associated degree and Dedekind zeta function.
Let $\beta_0$ be its possible exceptional real zero of $\zeta_L(s)$ with respect to the zero-free regions as defined in Theorem \ref{thmR} with the constants  $R_1$ and $R_2$. 
Let $(\omega_0,t_0)$ be selected from \lmaref{mlu}. 
Let $m \geq 1$ be an integer, let $x_0 \geq 4$, $\delta \leq \delta_0 \leq 1 - \frac{\sqrt{2}}{x_0}$, and $T \geq  T_0 \geq  t_0$.
For all $x$ satisfying $x \geq x_0$ and $\log x\geq X_{L,m,T}$ with $X_{L,m,T}$ as defined in \eqref{def-XLmT},we have
\begin{equation*}
  E_C (x) 
  \leq  \frac{x^{\beta_0 -1}}{\beta_0} + \varepsilon_L(\beta_0, \delta, T, x),
\end{equation*}
where 
\begin{equation}
\begin{split}
 \label{def-epsilon-1}
& \varepsilon_L(\beta_0, \delta, T, x) 
 =
\frac{\delta}{a_{\beta_0}} 
+ (\ell_0 + \ell_1) (\log d_L) (\log x)x^{-1} 
+ \ell_2 (\log x) x^{-\frac{1}{2}} 
+ \ell_3 (\log d_L)^2 x^{-\frac{1}{2}} 
\\&  + \ell_4 (\log d_L) x^{-\frac{1}{R_{1,L} \log (4 \Delta_L )}}  
+ \ell_5 (\log d_L) (\log T)^2 x^{- \frac{1}{R_{2,L} \log(\Delta_L  T)}} 
\\&  + \frac{\ell_6M(\delta, m)}{2 \delta^{m}}  \frac{\log T}{ T^m} (\log d_L)x^{-1+\frac{1}{R_{2,L} \log (\Delta_L  T)}}  
 + \frac{\ell_7 M(\delta, m)}{2 \delta^{m}} \Delta_L ^m \sqrt{n_L}  
  (\log x)^{3/4} e^{-2 \sqrt{\frac{m \log x}{R_{2}n_L}} } ,
  \end{split}
  \end{equation}
where
$M(\delta,m)$ is defined in \eqref{Mrosser} and the 
$\ell_i$'s are defined in \eqref{def-el0}, \eqref{def-el1}, \eqref{def-el2}, \eqref{def-el3}, \eqref{def-el4}, \eqref{def-el5}, \eqref{def-el6}, and \eqref{def-el7} respectively.
\end{proposition}
\begin{proof}
We first combine the bound \eqref{use1psiC} for 
$ \tilde E_C(x)$
with the bounds from \propref{boundI-H} for $\frac{|I_{L/K}(x) - \frac{|C|}{|G|} x|}{\frac{|C|}{|G|} x}$ 
and \lmaref{epstilbo} for $\frac{|I^{\text{ram}}_{L/K}(x)|}{\frac{|C|}{|G|} x}$.
We then bound $S_L(m,T,1)$ and $S_L(m,T,x)$ appearing in \propref{boundI-H}
with 
 \lmaref{bnd-SL1-exp} and \lmaref{bnd-SLT-exp} to obtain:
\begin{align*}
\tilde E_C (x)
 \leq 
& \frac{x^{\beta_0-1 }}{\beta_0} 
+ \frac{\delta}{a_{\beta_0}} 
+ (\ell_0 + \ell_1) (\log d_L) (\log x)x^{-1} 
+ \ell_2 (\log x) x^{-\frac{1}{2}} 
+ \ell_3 (\log d_L)^2 x^{-\frac{1}{2}} 
\\& + \ell_4 (\log d_L) x^{-\frac{1}{R_{1}n_{L} \log (4 \Delta_L )}}  
+ \ell_5 (\log d_L) (\log T)^2 x^{- \frac{1}{R_{2}n_{L} \log(\Delta_L  T)}} 
\\& + 
  \frac{\ell_6M(\delta, m)}{2 } (\log d_L) \frac{\log T}{\delta^m\,T^m} x^{-1+\frac{1}{R_{2}n_{L} \log (\Delta_L  T)}} 
\\ & 
+ \frac{\ell_7 M(\delta, m)}{2} \Delta_L ^m \sqrt{n_L}  \delta^{-m}
  (\log x)^{3/4} e^{-2 \sqrt{\frac{m \log x}{R_{2}n_L}} } .
\end{align*}
Note that the bound obtained is valid for both smoothing weights defined by $\alpha = 1 - \delta$ or $\alpha = 1$, so it is valid for both $\tilde E_C (x) =E^+(x)$ or $E^-(x)$, and thus is valid for $E_C (x)$ as $E_C (x) \leq \max \{ E^-(x), E^+(x) \}$.
\end{proof}
\subsection{Explicit estimates for \texorpdfstring{$\psi_C (x)$}{} independent of \texorpdfstring{$T$}{} but dependent on \texorpdfstring{$\delta$}{}}
\label{Section72}
We select here $T$ as a function of $x$ to remove the dependence in $T$ for $\varepsilon_L(\beta_0, \delta, T, x) $ from Proposition \ref{prop-bnd-error-delta-L-m-T-x} and to balance the main terms in $\varepsilon_L(\beta_0, \delta, T, x) $. 
Note that all of the $\ell_0$ to $\ell_6$ terms are of size no larger than
$e^{- \frac{\log x}{R_{2} n_L  \log(\Delta_L  T)}}$, whereas the $\ell_7$-term is of size
$e^{-2 \sqrt{\frac{m \log x}{R_{2}n_L}}}$.
By equating the exponents, we find that 
\begin{equation*}
   \label{eq-choiceT}
   \frac{\log x}{R_{2} n_L  \log(\Delta_L  T)} = 2 \sqrt{\frac{m \log x}{R_{2}n_L}} 
\end{equation*}
and solving for $T$ we obtain 
\begin{equation}\label{chooseT}
T = T_L(x):= \frac{1}{\Delta_L } \exp \bigg( \frac{1}{2\sqrt{m}} \sqrt{\frac{\log x}{R_{2} n_L }} \bigg).
\end{equation}

Further, as $m+1 \leq 4m$, this choice of $T$ ensures the condition 
$$\log x > X_{L,m,T}= (m+1) R_{2} n_L  \log^2 ( \Delta_L  T)=\frac{(m+1)}{4m} \log x$$ is satisfied.
In addition, note that the assumption $T\geq t_0$, with $t_0$ as in \lmaref{mlu} implies the lower bound for $\log x$:
\begin{equation}
    \label{cond-x-1}
    \log x > 4 m R_2 n_L \left(\log(t_0 \Delta_L )\right)^2.
\end{equation}
Inserting the expression \eqref{chooseT} for $T$ in terms of $x$ and rewrite Proposition \ref{prop-bnd-error-delta-L-m-T-x}.
\begin{proposition}
\label{prop-bnd-error-delta-L-m-x0-x}
Let $n_L \geq n_0 \geq 2$ and $\zeta_L(s)$ be the associated degree and Dedekind zeta function.
Let $\beta_0$ be its possible exceptional real zero of $\zeta_L(s)$ with respect to the zero-free regions as defined in Theorem \ref{thmR} with the constants  $R_1$ and $R_2$. 
Let $(\omega_0,t_0)$ be selected from \lmaref{mlu}. 
Let $m \geq 1$ be an integer, $T \geq  T_0 \geq  t_0$.
We define
\begin{equation}
    \label{cond-alpha-M}
\alpha  
= \max \left( 
4\frac{ R_1^2}{R_2}  ( (\log 4)\mathscr{M} + 1 )^2 , 4  R_2  \left( (\log t_0)\mathscr{M} +1\right)^2 \right),
\end{equation}
and 
\begin{equation}\label{define-x0}
\log x_0 = \frac{\alpha m n_0}{\mathscr{M}^2} .
\end{equation}
We assume 
\begin{equation}
\label{defining-x_0-delta_0}
\delta \leq \delta_0 \leq 1 - \frac{\sqrt{2}}{x_0}.
\end{equation}
For all $x$ satisfying 
\begin{equation}\label{Strong-logxcondition}
\log x \geq \alpha m n_L (\log \Delta_L )^2,
\end{equation}
we have
\begin{equation*}\label{main_error}
E_C (x) 
\leq  \frac{x^{\beta_0 -1}}{\beta_0} + \epsilon_L(\beta_0,\delta, x),
\end{equation*}
where 
\begin{align}
\label{Rtildepsi_2}
\epsilon_L(\beta_0,\delta, x) 
= & \frac{\delta}{a_{\beta_0}} +  A_L(x)   \delta^{-m} , 
\\%
\label{def-A-L-m-x0-x}
    A_L(x)  
    =&     A_L(\delta_0, m, \mathscr{M}, n_0, R_1, R_2, t_0,T_0, \omega_0, x_0,x )  
    = Y_0  \lambda_L (\log x)e^{-2 \sqrt{\frac{m \log x}{R_{2}n_L}} },
\\
    \label{def-lambdaL}
\lambda_L  
=& \lambda_L(m) 
= \max \left(  (\log \Delta_L)^2  n_L^2 ,  (\log \Delta_L) \Delta_L ^m \sqrt{ n_L} \right), 
\\
\label{def-Y-L-m-x}
Y_0  =
& Y_0(\delta_0, m, \mathscr{M}, n_0, R_1, R_2, t_0,T_0, \omega_0, x_0)
\\\notag
=&   \left( \frac{(\ell_0 + \ell_1)\mathscr{M}}{n_0} x_{0}^{-1} e^{2 \sqrt{\frac{m \log x_0}{R_{2}n_0}} }
 + \frac{\ell_2\mathscr{M}^2}{n_0^2} x_{0}^{-\frac{1}{2}} e^{2 \sqrt{\frac{m \log x_0}{R_{2}n_0}} } 
 + \ell_3   (\log x_0)^{-1}  x_{0}^{-\frac{1}{2}}  e^{2 \sqrt{\frac{m \log x}{R_{2}n_0}} } 
 \right. \\
 \notag  & \left.+ \frac{\ell_4\mathscr{M}}{n_0}  (\log x_0)^{-1} 
+ \frac{\ell_5 \mathscr{M}}{4 m R_{2} n_0^2 } 
\right)\delta_0^{m} 
+  \frac{\ell_6 M(\delta_0, m)}{4\sqrt{m R_{2}}}   (\log x_0)^{-1/2} x_{0}^{-1} e^{\frac72 \sqrt{\frac{m \log x_0}{R_{2}n_0}} } 
 \\ \notag
  &  + \frac{\ell_7  M(\delta_0, m)}{2}\mathscr{M} (\log x_0)^{-1/4},
\end{align}
where the  $\ell_i$'s are defined in \eqref{def-el0}, \eqref{def-el1}, \eqref{def-el2}, \eqref{def-el3}, \eqref{def-el4}, \eqref{def-el5}, \eqref{def-el6}, and \eqref{def-el7} respectively. 
\end{proposition}
\begin{proof}
We assume \eqref{cond-x-1} and implement our choice \eqref{chooseT} of $T$ in the $\ell_5$ and $\ell_6$ expressions of the definition  \eqref{def-epsilon-1} of $\varepsilon_L(\beta_0, \delta, T, x)$.
Since 
$$ 
\log T = -\log \Delta_L  + \frac{1}{2\sqrt{m}} \sqrt{\frac{\log x}{R_{2} n_L }} 
\leq -\frac1{\mathscr{M}}+\frac1{2\sqrt{m}} \sqrt{\frac{\log x}{R_{2} n_L }}
\leq \frac1{2\sqrt{m}} \sqrt{\frac{\log x}{R_{2} n_L }},  $$ 
it follows that the $\ell_5$-term is bounded by
$$ 
\ell_5 (\log d_L) (\log T)^2 x^{- \frac{1}{R_{2,L} \log(\Delta_L  T)}} 
\leq  \ell_5 (\log d_L) \frac1{4 m } \frac{\log x}{R_{2} n_L }  e^{-2 \sqrt{\frac{m \log x}{R_{2}n_L}} }
\leq  \frac{\ell_5}{4 m R_{2} }   \frac{(\log d_L)}{ n_L } (\log x) e^{-2 \sqrt{\frac{m \log x}{R_{2}n_L}} } .
$$ 
In addition, \eqref{chooseT} gives
$$ 
T^m = \frac{1}{\Delta_L ^m} e^{ \frac{\sqrt{m}}{2} \sqrt{\frac{\log x}{R_{2} n_L }} } 
\text{ and } 
 \frac{x^{\frac{1}{R_{2} n_L  \log (\Delta_L  T)}} }{T^m} 
 = \Delta_L ^m e^{ 3/2\sqrt{m} \sqrt{\frac{\log x}{R_{2} n_L }} } ,
  $$ 
so that the $\ell_6$-term becomes
\begin{align*}
 \frac{\ell_6M(\delta, m)}{2 \delta^{m} } (\log d_L) \frac{\log T}{T^m} x^{-1+\frac{1}{R_{2} n_L  \log (\Delta_L  T)}}  
& \leq 
\frac{\ell_6 M(\delta, m)}{4\sqrt{m R_{2}}\delta^{m} }  \frac{(\log d_L) \Delta_L ^m}{ \sqrt{ n_L}}(\log x)^{1/2} x^{-1}e^{ 3/2 \sqrt{\frac{m \log x}{R_{2}n_L}}}.
\end{align*}
We deduce the following bound for \eqref{def-epsilon-1},  which is now independent of $T$:
\begin{align*}
\varepsilon_L(\beta_0, \delta, T, x) \leq    
& \frac{\delta}{a_{\beta_0}} 
+ (\ell_0 + \ell_1) (\log d_L) (\log x)x^{-1} 
+ \ell_2 (\log x) x^{-\frac{1}{2}} 
+ \ell_3 (\log d_L)^2 x^{-\frac{1}{2}} 
\\& + \ell_4 (\log d_L) x^{-\frac{1}{R_{1,L} \log (4 \Delta_L )}}  
+ \frac{\ell_5}{4 m R_{2} }   \frac{(\log d_L)}{ n_L } (\log x) e^{-2 \sqrt{\frac{m \log x}{R_{2}n_L}} }
\\& + \frac{\ell_6 M(\delta, m)}{4\sqrt{m R_{2}}}  \frac{(\log d_L) \Delta_L ^m}{ \sqrt{ n_L}}\delta^{-m} (\log x)^{1/2} x^{-1}e^{ 3/2 \sqrt{\frac{m \log x}{R_{2}n_L}}}
\\&+ \frac{\ell_7 M(\delta, m)}{2} \Delta_L ^m \sqrt{n_L}  \delta^{-m}
  (\log x)^{3/4} e^{-2 \sqrt{\frac{m \log x}{R_{2}n_L}} } .
\end{align*} 
It follows that the leading term in the above right hand expression arises from $\ell_4, \ell_5$ and $\ell_7$, and is of size 
$$
\max\left(  x^{-\frac{1}{R_{1}n_L \log (4 \Delta_L )}}, (\log x) e^{-2 \sqrt{\frac{m \log x}{R_{2}n_L}} } \right)
=
(\log x) e^{-2 \sqrt{\frac{m \log x}{R_{2}n_L}} }$$
as soon as
\begin{equation}
    \label{cond-x-2}
\log x > 4\frac{m R_1^2}{R_2} n_L (\log 4 \Delta_L )^2 .
\end{equation}
It follows that all $\ell_i$ terms have at most size $(\log x) e^{-2 \sqrt{\frac{m \log x}{R_{2}n_L}} }$. 
Under this  assumption, we deduce the following upper bound for  $\frac{ \varepsilon_L(\beta_0, \delta, T, x) - \frac{\delta}{a_{\beta_0}} }{(\log x)e^{-2 \sqrt{\frac{m \log x}{R_{2}n_L}} } }  
$:
  \begin{multline}\label{prf-eps-bnd}
(\ell_0 + \ell_1) (\log d_L)  x^{-1} e^{2 \sqrt{\frac{m \log x}{R_{2}n_L}} } 
 + \ell_2  x^{-\frac{1}{2}}
 e^{2 \sqrt{\frac{m \log x}{R_{2}n_L}} } 
+ \ell_3 (\log d_L)^2  \frac{ x^{-\frac{1}{2}} e^{2 \sqrt{\frac{m \log x}{R_{2}n_L}} } }{\log x} 
+ \frac{\ell_4 (\log d_L) }{ \log x }
\\ + \frac{\ell_5}{4 m R_{2} }   \frac{(\log d_L)}{ n_L } 
+ \frac{\ell_6 M(\delta, m)}{4\sqrt{m R_{2}} \delta^{m}}  \frac{(\log d_L) \Delta_L ^m}{ \sqrt{ n_L}}  \frac{ x^{-1} e^{ 7/2 \sqrt{\frac{m \log x}{R_{2}n_L}} } }{(\log x)^{1/2}}
+ \frac{\ell_7 M(\delta, m)}{2\delta^{m}} \frac{\Delta_L^m \sqrt{n_L}}{(\log x)^{1/4}}  .
      \end{multline} 
Each of the functions of $x$ in \eqref{prf-eps-bnd} are of the form $x^{-a} e^{b \sqrt{\log x}}$,   
with each decreasing with $x$, as long as 
\begin{equation}
  \label{extraxcond}
\log x \geq  \frac{4m}{R_2 n_L}.
 \end{equation}  
We note that this is satisfied under the condition $x\geq x_0$ as defined in \eqref{define-x0} (see more details below).
Next, we bound each factor in \eqref{prf-eps-bnd} depending on the field $L$ with 
\begin{equation}\label{termL2bnd}
    \max \left( (\log d_L)^2 , \frac{(\log d_L) \Delta_L ^m}{ \sqrt{ n_L}} \right)
    = \max \left(  (\log \Delta_L)^2  n_L^2 ,  (\log \Delta_L) \Delta_L ^m \sqrt{ n_L} \right) =: \lambda_L. 
\end{equation}
Together with \eqref{prf-eps-bnd} and \eqref{termL2bnd}, we conclude by using the inequalities $\delta \leq \delta_0, M(\delta,m) \leq M(\delta_0,m)$,  $n_L\geq n_0$, and $\frac{n_L}{\log d_L} \leq \mathscr{M}$: 
\begin{align*} 
\frac{ \varepsilon_L(\beta_0, \delta, T, x) - \frac{\delta}{a_{\beta_0}} }{ \lambda_L \delta^{-m}  (\log x)e^{-2 \sqrt{\frac{m \log x}{R_{2}n_L}} }} 
  \leq &
 \left( \frac{(\ell_0 + \ell_1)\mathscr{M}}{n_0}  x_{0}^{-1} e^{2 \sqrt{\frac{m \log x_0}{R_{2}n_0}} }
 + \frac{\ell_2\mathscr{M}^2}{n_0^2}  x_{0}^{-\frac{1}{2}} e^{2 \sqrt{\frac{m \log x_0}{R_{2}n_0}} } 
 \right. \\  
& \left. + \ell_3   (\log x_0)^{-1}   x_{0}^{-\frac{1}{2}} e^{2 \sqrt{\frac{m \log x}{R_{2}n_0}} } 
+ \frac{\ell_4\mathscr{M}}{n_0}  (\log x_0)^{-1} 
+ \frac{\ell_5 \mathscr{M}}{4 m R_{2} n_0^2 } 
\right)\delta_0^{m}
\\ & +  \frac{\ell_6 M(\delta_0, m)}{4\sqrt{m R_{2}}}   (\log x_0)^{-1/2}  x_{0}^{-1} e^{  7/2 \sqrt{\frac{m \log x_0}{R_{2}n_0}} } 
  + \frac{\ell_7  M(\delta_0, m)}{2} \mathscr{M}  (\log x_0)^{-1/4}. 
\end{align*}
{\bf Clarifying the condition assumed on $x$:}
Note that the definition \eqref{cond-alpha-M} of $\alpha$,
$$
\alpha = \max \left( 4\frac{ R_1^2}{R_2} \left( (\log 4)\mathscr{M} + 1 \right)^2 ,\ 4 R_2 \left( (\log t_0)\mathscr{M} + 1 \right)^2 \right),
$$
together with the inequality $\frac{1}{\log \Delta_L} \leq \mathscr{M}$, ensures that
$$
\alpha \geq \frac{\max\left( 4 \frac{R_1^2}{R_2} (\log (4 \Delta_L))^2 ,\ 4 R_2 \left( \log(t_0 \Delta_L) \right)^2 \right)}{(\log \Delta_L)^2},
$$
and hence the assumption \eqref{Strong-logxcondition},
$
\log x \geq \alpha m n_L (\log \Delta_L)^2,
$
implies
$$ 
\log x \geq \max\left( 4 \frac{m R_1^2}{R_2} n_L (\log (4 \Delta_L))^2 ,\ 4 m R_2 n_L \left( \log(t_0 \Delta_L) \right)^2 \right),
$$
which in turn implies two of the conditions \eqref{cond-x-1} and \eqref{cond-x-2} required for $x$.
Recall $x_0$ is defined via \eqref{define-x0} as
$$
\log x_0 = \frac{\alpha m n_0}{\mathscr{M}^2}.
$$
We note that assuming \eqref{Strong-logxcondition} implies $\log x \geq \log x_0$, since
$$
\alpha m n_L (\log \Delta_L)^2 \geq \frac{\alpha m n_0}{\mathscr{M}^2}.
$$
In addition, $\log x \geq \log x_0$ implies that
$
\frac{\alpha m n_0}{\mathscr{M}^2} \geq \frac{4m}{R_2 n_L},
$
and thus that $x$ satisfies condition \eqref{extraxcond}. This follows from the estimate
$$
\frac{\alpha n_0 R_2 n_L}{4 \mathscr{M}^2} \geq \frac{\alpha R_2}{\mathscr{M}^2} \geq 4 R_1^2 \left( (\log 4) + \frac{1}{\mathscr{M}} \right)^2 \geq 4 \cdot 20^2 \left( (\log 4) + \frac{1}{1.82048} \right)^2 = 5994.47\ldots > 1.
$$
\end{proof}
\subsection{Explicit estimates for $\psi_C(x)$ of the classical exponential shape}
\label{Section73}
In this subsection, we modify \propref{prop-bnd-error-delta-L-m-x0-x} by making the error term $E_C (x)$ independent of the parameter $\delta$.  For the sake of exposition, 
we now provide full details for \thmref{main-thm-psi} with \thmref{thm-psi} and provide its proof. 
Note that prior to Proposition \ref{prop-bnd-error-delta-L-m-x0-x}, $x_0$ was a parameter.  However, from this point on $x_0$ is fixed and  
$x_0 =\exp(\frac{\alpha m n_0}{\mathscr{M}^2}) $ as given in \eqref{define-x0}.
\begin{theorem}\label{thm-psi}
Let $n_L \geq n_0 \geq 2$, $\Delta_L$ and $\mathscr{M}$ be as defined in \eqref{def-DeltaL} and \eqref{def-Minkowski} respectively.
Let $\zeta_L(s)$ be the associated Dedekind zeta function and let $\beta_0,R_1,R_2$ be as defined in Theorem \ref{thmR}. 
Let $\alpha, \delta_0,m, t_0, T_0,\omega_0$, and $x_0$ be as in \propref{prop-bnd-error-delta-L-m-x0-x}.  
For all $x$ satisfying \eqref{Strong-logxcondition}
$$\log x \geq \alpha m n_L (\log \Delta_L )^2,$$
we have
\begin{equation}\label{main_error}
  E_C (x) 
  \leq  \frac{x^{\beta_0 -1}}{\beta_0} + \mathscr{E}_L(\beta_0,m,R_2,x),
\end{equation}
where 
\begin{equation}
    \label{def-eps-L-m-x}
\mathscr{E}_L(\beta_0,m,R_2,x) \leq 
\max \left(\mathscr{E}_1(\beta_0)   , \mathscr{E}_2(\beta_0)  \right) 
\lambda_L  n_L^{\frac{m}{m+1} }
(\log x)^{\frac{1}{m+1}}
e^{-\frac{2}{\sqrt{R_2}} \frac{\sqrt{m}}{m+1} \sqrt{\frac{\log x}{n_L}} },
\end{equation}
with $\mathscr{E}_i(\beta_0) =  \mathscr{E}_i(\alpha, \beta_0, \delta_0, m, \mathscr{M}, n_0, R_1, R_2, t_0,T_0, \omega_0, x_0)$
and 
\begin{align}
    \label{def-eps1-L-m-x}
& \mathscr{E}_1(\beta_0) =
\frac{ m^{\frac{1}{m+1}} \mathscr{M}^{\frac{2m}{m+1}}  Y_0 ^{\frac{1}{m+1}} 
}{ 
a_{\beta_0}^{ \frac{m}{m+1}} 
n_0^{ \frac{3m}{m+1}} }
+ \frac{ (\alpha m)^{\frac{m}{m+1} } Y_0   }{ \delta_0^{m} \mathscr{M}^{ \frac{2m}{m+1} } 
e^{ 2 \frac{m^{2}}{m+1} \sqrt{\frac{\alpha  }{R_{2}\mathscr{M}^2}} } },
\\
\label{def-eps2-L-m-x}
& \mathscr{E}_2(\beta_0) 
=
\frac{ (m+1) 
\mathscr{M}^{\frac{2m}{m+1}} 
Y_0 ^{\frac{1}{m+1}} 
}{(a_{\beta_0}m)^{\frac{m}{m+1}} n_0^{\frac{3m}{m+1}} },
\end{align}
and with $a_{\beta_0}$, $\lambda_L $, and $Y_0 $ and  as defined in \eqref{abeta0}, \eqref{def-lambdaL}, and \eqref{def-Y-L-m-x} respectively. 
Further, we define
 \begin{equation}\label{def-curvyN0}    
\mathcal{N}_0
=\mathcal{N}_0(\alpha, \beta_0, \delta_0, m,\mathscr{M}, R_2, Y_0) 
=
\frac{
\delta_0^{\frac{m+1}3} \mathscr{M} e^{ \frac{m}{3 \mathscr{M}}
\left(  \frac{2 \sqrt{ \alpha}}{\sqrt{R_{2} }} - 1 \right) }
 }{ m (a_{\beta_0} c_0 \alpha  Y_0)^{1/3}  }
\end{equation}
where $c_0 = 0.354$. 
If 
\begin{equation}
    \label{cond-nL}
n_0 \leq  n_L  \leq \mathcal{N}_0,
\end{equation}
then we can replace \eqref{def-eps-L-m-x} with 
\begin{equation}
\label{def2-eps-L-m-x}
\mathscr{E}_L(\beta_0,m,R_2,x) \leq
\mathscr{E}_3(\beta_0) \lambda_L^{\frac{1}{m+1}} (\log x)^{\frac{1}{m+1}}
e^{-\frac{2}{\sqrt{R_2}} \frac{\sqrt{m}}{m+1} \sqrt{\frac{\log x}{n_L}} },  
\end{equation}
or with
\begin{equation}
\label{def3-eps-L-m-x}
\mathscr{E}_L(\beta_0,m,R_2,x) \leq
\tilde{\mathscr{E}_3}(\beta_0) \lambda_L  n_L^{\frac{m}{m+1} }
(\log x)^{\frac{1}{m+1}}
e^{-\frac{2}{\sqrt{R_2}} \frac{\sqrt{m}}{m+1} \sqrt{\frac{\log x}{n_L}} },
\end{equation}
where
\begin{equation}\label{def-epsilon3}
    \mathscr{E}_3(\beta_0)  
= \mathscr{E}_3(\beta_0, m, Y_0) 
= (a_{\beta_0}m)^{-\frac{m}{m+1}}(m+1)
Y_0 ^{\frac{1}{m+1}},
\end{equation}
\begin{equation}\label{def-epsilontilde3}
\tilde{    \mathscr{E}_3}(\beta_0)  
=\tilde{ \mathscr{E}_3}(\beta_0, m, Y_0) 
= \frac{\mathscr{E}_3(\beta_0)}{\sqrt{  n_0 \lambda_0}},
\end{equation}
and 
\begin{equation}
\label{def-lambda0}
\lambda_0  
= \max \Big(  \frac{n_0^2}{\mathscr{M}^2}   ,  \frac{\sqrt{ n_0}e^{\frac1{\mathscr{M}}}}{\mathscr{M}}  \Big) .
\end{equation}
\end{theorem}
\begin{proof}[\textbf{Proof of \thmref{thm-psi} / \thmref{main-thm-psi}}]
We shall apply Proposition \ref{prop-bnd-error-delta-L-m-x0-x} and recall \eqref{Rtildepsi_2}:
$$ 
\epsilon_L(\beta_0,\delta, x)
=f(\delta)=a \delta+b \delta^{-m}
$$  
where $a=\frac{1}{a_{\beta_0}}$ and $b=A_L(x) $. Note that $f(\delta)$ has a critical point at $
\delta_c= (\frac{mb}{a} )^{\frac{1}{m+1}}$
and $f(\delta_c)= a\big(\frac{m+1}{m}\big) \delta_c$.
Therefore the critical point for $\delta$ which minimizes $\epsilon_L(\beta_0,\delta, x)$ occurs at 
$\delta_L (m,x) = (a_{\beta_0} m A_L(x))^{\frac{1}{m+1}}$. Replacing 
 $A_L$ with its definition \eqref{def-A-L-m-x0-x}, we have  
\begin{equation}\label{def-delta-critical2}
    \delta_L (m,x) =
a_{\beta_0}^{\frac{1}{m+1}} m^{\frac{1}{m+1}} Y_0 ^{\frac{1}{m+1}} 
\lambda_L^{\frac{1}{m+1}}
(\log x)^{\frac{1}{m+1}}
e^{-2 \frac{\sqrt{m}}{m+1} \sqrt{\frac{\log x}{R_{2}n_L}} } .
\end{equation}
Our choice for $\delta$ depends on the location of $\delta_L(m,x)$ with respect to $\delta_0$ as $\delta_L(m,x)$ may or may not lie 
in the interval $(0,\delta_0)$. 
\\
If $\delta_0\geq \delta_L(m,x)$, we choose $\delta=\delta_L(m,x)$ and substitute this in \eqref{Rtildepsi_2} to obtain
\begin{equation}
\begin{split}
 \label{epsLid}
\epsilon_L(\beta_0,\delta_L (m,x), x)
& = \frac1{a_{\beta_0}} \Big(\frac{m+1}{m}\Big) \delta_L (m,x)
\\
& 
= (a_{\beta_0}m)^{-\frac{m}{m+1}}(m+1) 
Y_0 ^{\frac{1}{m+1}} \lambda_L^{\frac{1}{m+1}} (\log x)^{\frac{1}{m+1}}
e^{-\frac{2}{\sqrt{R_2}} \frac{\sqrt{m}}{m+1} \sqrt{\frac{\log x}{n_L}} } .
\end{split}
\end{equation}
This gives \eqref{def2-eps-L-m-x}.
Note that, from its definition \eqref{def-lambdaL}, 
\begin{equation*}
 \label{lambdaLmineq}
\lambda_L \geq (\log\Delta_L)^2n_L^2 > \left(\frac{n_0}{\mathscr{M}}\right)^2>1, 
\end{equation*}
where we have used $n_0 > \mathscr{M}$. This can be 
seen by comparing the $n_0$ and $\mathscr{M}$ columns  in Table \ref{n0d0}.
It follows that $\frac{n_0^{3m}}{\mathscr{M}^{2m}} \lambda_L < \lambda_L^{m+1}  n_L^{m}$ and 
\begin{equation}
    \label{comp-Lterms}
\lambda_L^{\frac{1}{m+1} } < \frac{\mathscr{M}^{\frac{2m}{m+1}}}{n_0^{\frac{3m}{m+1}}} \lambda_L  n_L^{\frac{m}{m+1} }.
\end{equation}
Using this last inequality in \eqref{epsLid} we find
$$ 
\epsilon_L(\beta_0,\delta_L (m,x), x)
\leq  (a_{\beta_0}m)^{-\frac{m}{m+1}}(m+1)
Y_0 ^{\frac{1}{m+1}} 
\frac{\mathscr{M}^{\frac{2m}{m+1}}}{n_0^{\frac{3m}{m+1}}} \lambda_L  n_L^{\frac{m}{m+1} }
(\log x)^{\frac{1}{m+1}}
e^{-\frac{2}{\sqrt{R_2}} \frac{\sqrt{m}}{m+1} \sqrt{\frac{\log x}{n_L}} } ,
$$  
which establishes \eqref{def-eps2-L-m-x}.

If $\delta_0<\delta_L(m,x)$, we choose $\delta=\delta_0$ and substitute in \eqref{Rtildepsi_2}:
$$ 
\epsilon_L(\beta_0,\delta_0, x)
= \frac{\delta_0}{a_{\beta_0}} +\frac{ Y_0  }{ \delta_0^{m}} \lambda_L (\log x)e^{-2 \sqrt{\frac{m}{R_{2}}} \sqrt{\frac{\log x}{n_L}} }  .
$$ 
We use the fact that $\delta_0<\delta_L(m,x)$ as defined in \eqref{def-delta-critical2} and obtain
\begin{align*}
\epsilon_L(\beta_0,\delta_0, x)
& \leq 
a_{\beta_0}^{-\frac{m}{m+1}} m^{\frac{1}{m+1}} Y_0 ^{\frac{1}{m+1}} 
\lambda_L^{\frac{1}{m+1}}
(\log x)^{\frac{1}{m+1}}
e^{-2 \frac{\sqrt{m}}{m+1} \sqrt{\frac{\log x}{R_{2}n_L}} } 
+\frac{ Y_0  }{ \delta_0^{m}} \lambda_L (\log x)e^{-2 \sqrt{\frac{m}{R_{2}}} \sqrt{\frac{\log x}{n_L}} }
\\& =
\Big( 
a_{\beta_0}^{-\frac{m}{m+1}} 
m^{\frac{1}{m+1}} Y_0 ^{\frac{1}{m+1}} 
\lambda_L^{\frac{1}{m+1}}
+\frac{ Y_0  }{ \delta_0^{m}} \lambda_L (\log x)^{\frac{m}{m+1} }
e^{-2 \frac{m^{3/2}}{m+1} \sqrt{\frac{\log x}{R_2 n_L}} }
\Big)
(\log x)^{\frac{1}{m+1}}
e^{-2 \frac{\sqrt{m}}{m+1} \sqrt{\frac{\log x}{R_{2}n_L}} } ,
\end{align*}
Note that the function of $x$ in the last line within the brackets is of the shape 
$$\phi_{a,b}(u) =u^ae^{-bu}$$ with $u=(\frac{\log x}{R_{2}n_L})^{\frac{1}{2}} \geq (\frac{\alpha m }{R_{2}\mathscr{M}^2})^{\frac{1}{2}}  $, $a=\frac{2}{m+1}$, and $b = 2 \frac{\sqrt{m}}{m+1}$.
By calculus,  this  function decreases for $u
>\frac{a}{b}= 
\frac{1}{\sqrt{m}}$.
It follows that 
\begin{align*}
\epsilon_L(\beta_0,\delta_0, x)
 \leq &
\Big( 
a_{\beta_0}^{-\frac{m}{m+1}} 
m^{\frac{1}{m+1}} Y_0 ^{\frac{1}{m+1}} 
\lambda_L^{\frac{1}{m+1}}
\Big.\\&
\Big.+\frac{ Y_0  }{ \delta_0^{m}} \lambda_L (\log x)^{\frac{m}{m+1} }
e^{-2 \frac{m^{3/2}}{m+1} \sqrt{\frac{\log x}{R_2 n_L}} }
\Big)
(\log x)^{\frac{1}{m+1}}
e^{-2 \frac{\sqrt{m}}{m+1} \sqrt{\frac{\log x}{R_{2}n_L}} } .
\end{align*}
The expression $a_{\beta_0}^{-\frac{m}{m+1}} m^{\frac{1}{m+1}} Y_0 ^{\frac{1}{m+1}} \lambda_L^{\frac{1}{m+1}}$ is bounded using \eqref{comp-Lterms}:
$$
a_{\beta_0}^{-\frac{m}{m+1}} m^{\frac{1}{m+1}} Y_0 ^{\frac{1}{m+1}} \lambda_L^{\frac{1}{m+1}} \leq a_{\beta_0}^{-\frac{m}{m+1}} 
m^{\frac{1}{m+1}} Y_0 ^{\frac{1}{m+1}}  \mathscr{M}^{\frac{2m}{m+1}} n_0^{-\frac{3m}{m+1}} \lambda_L n_L^{\frac{m}{m+1}}, 
$$
and we use that $(\log x)^{\frac{m}{m+1} } e^{-2 \frac{m^{3/2}}{m+1} \sqrt{\frac{\log x}{R_2 n_L}} }$ decreases with $x$, for $\log x \geq \alpha m \frac{n_L}{\mathscr{M}^2}$, to obtain the bound
\begin{align*}
\epsilon_L(\beta_0,\delta_0, x)
\leq & 
\Big( 
a_{\beta_0}^{-\frac{m}{m+1}} 
m^{\frac{1}{m+1}} Y_0 ^{\frac{1}{m+1}} 
\mathscr{M}^{\frac{2m}{m+1}} n_0^{-\frac{3m}{m+1}}
\Big.\\&
\Big.+ \delta_0^{-m} Y_0   (\alpha m)^{\frac{m}{m+1} }  \mathscr{M}^{-\frac{2m}{m+1} } 
e^{-2 \frac{m^{2}}{m+1} \sqrt{\frac{\alpha  }{R_{2}\mathscr{M}^2}} }
\Big)
\lambda_L  n_L^{\frac{m}{m+1} } (\log x)^{\frac{1}{m+1}}
e^{-2 \frac{\sqrt{m}}{m+1} \sqrt{\frac{\log x}{R_{2}n_L}} } .
\end{align*}
To end the proof we determine a bound on $n_L$, namely \eqref {cond-nL} which implies that $\delta_L(m,x) \leq \delta_0$. 
We now clarify what the condition $\delta_0\geq \delta_L(m,x)$ entails.
To simplify \eqref{termL2bnd}, we use condition \eqref{Strong-logxcondition} on $\Delta_L$, namely that 
$\log x \geq \alpha n_L (\log \Delta_L )^2$
gives 
\begin{equation*}
 \label{bnd-Delta}
\log \Delta_L \leq \sqrt{\frac{\log x}{\alpha m n_L}}
\ \text{and}\ 
\Delta_L \leq e^{ \sqrt{\frac{\log x}{\alpha m n_L}} } .
\end{equation*}
Thus, setting $X= \sqrt{\frac{m(\log x)}{ \alpha n_L}} $, 
\begin{align*}
 \lambda_L 
& \leq \max \Big( \frac{\log x}{\alpha m n_L} n_L^2 , \sqrt{\frac{\log x}{\alpha m n_L}}  e^{m \sqrt{\frac{\log x}{\alpha m n_L}}} \sqrt{ n_L} \Big)
= Xe^X \max \Big( Xe^{-X} \frac{n_L^2}{m^2}, \frac{\sqrt{ n_L}}m \Big) 
\\
&
 \leq Xe^X  n_L^2 \max \Big(  \frac{1}{e m^2 n_L^2}, \frac{1}{m n_{L}^{\frac{3}{2}}} \Big) 
\leq c_0 Xe^X n_L^2 ,
\end{align*}
where $c_0 = \max(1/4e, 1/2^{3/2})=0.353 \ldots$. 
Using this inequality in \eqref{def-delta-critical2}, we find 
\begin{equation}
\begin{split} 
 \label{deltaLineq}
\delta_L (m,x) 
& \leq a_{\beta_0}^{\frac{1}{m+1}} m^{\frac{1}{m+1}} Y_0 ^{\frac{1}{m+1}} 
\left( c_0 Xe^X n_L^2 
\right)^{\frac{1}{m+1}}
(\log x)^{\frac{1}{m+1}}
e^{-2 \frac{\sqrt{m}}{m+1} \sqrt{\frac{\log x}{R_{2}n_L}} }
\\ 
& =  c_{0}^{\frac{1}{m+1}}
a_{\beta_0}^{\frac{1}{m+1}} 
m^{\frac{3}{2(m+1)}}  
\alpha^{-\frac{1}{2(m+1)}} 
Y_0 ^{\frac{1}{m+1}} 
n_L^{\frac{3}{(m+1)}} 
\Big(\frac{\log x}{  n_L}\Big)^{\frac{3}{2(m+1)}}
e^{- \frac{\sqrt{m}}{m+1} \sqrt{\frac{(\log x)}{  n_L}}
\Big(  \frac{2}{\sqrt{R_{2} }} - \frac{1}{\sqrt{\alpha }}  \Big)}.
\end{split}
\end{equation}
The function of $x$ in 
\eqref{deltaLineq} is again of the shape $\phi_{a,b}(u) $, this time with $u=(\frac{\log x}{n_L})^{\frac{1}{2}}$, $a=\frac{3}{m+1}$, and  $b = \frac{\sqrt{m}}{m+1} \left(  \frac{2}{\sqrt{R_{2} }} - \frac{1}{\sqrt{\alpha }}  \right)$.  It follows from \eqref{Strong-logxcondition} and the lower bound \eqref{def-Minkowski} 
that $(\frac{\log x}{n_L})^{\frac{1}{2}} 
> ( \alpha m \mathscr{M}^{-2})^{\frac{1}{2}}$ and from  
assumption  \eqref{cond-alpha-M})  $\alpha > \left( \frac{3\mathscr{M} }{m} + 1 \right)^2 \frac{ R_{2} }{4}$ 
that
\begin{equation}
 \label{inequalities}
 \Big(\frac{\log x}{n_L} \Big)^{\frac{1}{2}} 
> ( \alpha m \mathscr{M}^{-2})^{\frac{1}{2}}
\geq 
\frac{3\sqrt{\alpha }}{ \sqrt{m} \left(  \frac{2\sqrt{\alpha }}{\sqrt{R_{2}}} - 1  \right)}.
\end{equation}
Since $\phi_{a,b}$ decreases for $u > \frac{a}{b} = \frac{3\sqrt{\alpha }}{ \sqrt{m} \left(  \frac{2\sqrt{\alpha }}{\sqrt{R_{2}}} - 1  \right)}$,  it follows that
$$ 
  \phi_{a,b}\Big(\Big( \frac{\log x}{  n_L} \Big)^{\frac{1}{2}} \Big)= 
\Big(\Big( \frac{\log x}{  n_L} \Big)^{\frac{1}{2}} \Big)^{\frac{3}{(m+1)}}
e^{- \frac{\sqrt{m}}{m+1} \sqrt{\frac{(\log x)}{  n_L}}
\left(  \frac{2}{\sqrt{R_{2} }} - \frac{1}{\sqrt{\alpha }}  \right)}
\leq 
\left( \frac{\alpha m}{\mathscr{M}^2} \right)^{\frac{3}{2(m+1)}}
e^{- \frac{ m }{m+1} \frac1{\mathscr{M}}
\left(  \frac{2 \sqrt{ \alpha}}{\sqrt{R_{2} }} - 1 \right)} 
$$  
and thus
$$ 
\delta_L(m,x) 
\leq c_{0}^{\frac{1}{m+1}}
a_{\beta_0}^{\frac{1}{m+1}} 
m^{\frac{3}{(m+1)}}  
\alpha^{\frac{1}{(m+1)}} 
Y_0 ^{\frac{1}{m+1}} 
n_L^{\frac{3}{(m+1)}} 
\mathscr{M}^{-\frac{3}{(m+1)}}
e^{- \frac{ m }{m+1} \frac1{\mathscr{M}}
\left(  \frac{2 \sqrt{ \alpha}}{\sqrt{R_{2} }} - 1 \right)}.
$$ 
From this, it follows that  if 
$$ 
 c_{0} 
a_{\beta_0}
m^{3}  
\alpha  
Y_0  
n_L^{ 3 } 
\mathscr{M}^{-3}
e^{-  \frac{m}{\mathscr{M}}
\left(  \frac{2 \sqrt{ \alpha}}{\sqrt{R_{2} }} - 1 \right)}
\leq \delta_0^{m+1},
$$ 
then $\delta_L(m,x) \leq \delta_0$.
Solving for $n_L$, we find 
$$ 
n_L  
\leq   \bigg( 
\frac{\delta_0^{m+1} \mathscr{M}^{3} e^{ \frac{m}{\mathscr{M}}
\big(  \frac{2 \sqrt{ \alpha}}{\sqrt{R_{2} }} - 1 \big)}
 }{ c_0 a_{\beta_0}
m^{3}  
\alpha  
Y_0  
}
\bigg)^{1/3}.
$$ 
The quantity on the right is defined to be $ \mathcal{N}_0(\alpha, \beta_0, \delta_0, m,\mathscr{M}, R_2, Y_0)$. This concludes the proof as 
we have shown that the inequality $n_L \leq \mathcal{N}_0(\alpha, \beta_0, \delta_0, m,\mathscr{M}, R_2, Y_0)$ implies that the bound 
\eqref{def2-eps-L-m-x} is valid in this case.
\\
Finally, to deduce \eqref{def3-eps-L-m-x}, we compare \eqref{def-eps-L-m-x} with \eqref{def2-eps-L-m-x}, by observing that 
$$
\mathscr{E}_3(\beta_0) \lambda_L^{\frac1{m+1}} (\log x)^{\frac1{m+1}}
e^{-\frac{2}{\sqrt{R_2}} \frac{\sqrt{m}}{m+1} \sqrt{\frac{\log x}{n_L}} }
\leq 
\frac{\mathscr{E}_3(\beta_0)}{ (n_0 \lambda_0)^{\frac{m}{m+1}}} 
\lambda_L  n_L^{\frac{m}{m+1} }
(\log x)^{\frac{1}{m+1}}
e^{-\frac{2}{\sqrt{R_2}} \frac{\sqrt{m}}{m+1} \sqrt{\frac{\log x}{n_L}} },
$$
since $\lambda_L\geq \lambda_0$, with $\lambda_L  $ and $\lambda_0$ defined in \eqref{def-lambdaL} and \eqref{def-lambda0} respectively. 
\end{proof}
\subsubsection{Calculations for \thmref{thm-psi} - \thmref{main-thm-psi} (see Table \ref{beta0-present-cor1.1})} \label{calcsThm35Thm1}
We observe that the exponent of $e^{-\frac2{R_2}\frac{\sqrt{m}}{m+1}\sqrt{\frac{\log x}{n_L}}}$ decreases with $m$. So, for the rest of the article, we set
$m=1$.
The choices for $R_1, R_2$ are given in \thmref{thmR} and we choose $\omega_0$
minimally with corresponding $t_0$ value from Table \ref{Table-t0-omega0}.
We summarize here the values chosen for the parameters $m, R_1,R_2, \alpha, t_0, T_0,$ and $\omega_0$:
\begin{equation}
 \label{parameters}
m=1, \, R_1 = 20, \, R_2= 12.2411,    \, T_0= t_0 = 40, \, \text{and}\  \omega_0 = 1.
\end{equation}
The values for $\alpha, \log x_0, $ and $\delta_0$ depend on the pair of values for $(n_0,\mathscr{M})$ as given in Table \ref{n0d0}.
The condition $\log x \geq \alpha n_L (\log \Delta_L)$ via \eqref{Strong-logxcondition} has $\alpha$ as in \eqref{cond-alpha-M}:
\begin{equation}
  \label{alphaspecific}
\alpha  
=  \max \left( 130.71 ( 1.39 \mathscr{M} + 1 )^2 , 48.97  \left( 3.69 \mathscr{M} +1\right)^2 \right).
\end{equation}
In addition, $\log x_0$ is as in \eqref{define-x0}
$$\log x_0 = \frac{\alpha n_0}{\mathscr{M}^2} .
$$
We rewrite the error terms $\mathscr{E}_L(\beta_0,m,R_2,x)$ from \eqref{def-eps-L-m-x} as follows:
\begin{equation} \label{def-eps-L-m=1-x}
\mathscr{E}_L(\beta_0,m,R_2,x) \leq 
 \max \left(\mathscr{E}_1(\beta_0)   , \mathscr{E}_2(\beta_0)  \right)
\lambda_L  n_L^{\frac12 }
(\log x)^{\frac12}
e^{-\frac{1}{\sqrt{R_2}} \sqrt{\frac{\log x}{n_L}} }.
\end{equation}
In addition, when $n_0 \leq n_L  \leq \mathcal{N}_0,$ we have \eqref{def2-eps-L-m-x}:
\begin{equation}
\label{def2-eps-L-m=1-x}
\mathscr{E}_L(\beta_0,m,R_2,x) \leq 
\mathscr{E}_3(\beta_0) \lambda_L^{\frac12} (\log x)^{\frac12}
e^{-\frac{1}{\sqrt{R_2}}  \sqrt{\frac{\log x}{n_L}} } .
\end{equation}
In those cases we also have \eqref{def3-eps-L-m-x}: 
\begin{equation} \label{def3-eps-L-m=1-x}
\mathscr{E}_L(x)
 \leq 
 \tilde{\mathscr{E}_3}(\beta_0)  \lambda_L  n_L^{\frac{1}{2} }
(\log x)^{\frac{1}{2}}
e^{-\frac{1}{\sqrt{R_2}} \sqrt{\frac{\log x}{n_L}} }.
\end{equation}
Thus, in those cases,
$$
\mathscr{E}_L(x)
 \leq \min\left( \max \left(\mathscr{E}_1(\beta_0)   , \mathscr{E}_2(\beta_0)  \right) , \tilde{\mathscr{E}_3}(\beta_0) \right)   \lambda_L  n_L^{\frac{1}{2} }
(\log x)^{\frac{1}{2}}
e^{-\frac{1}{\sqrt{R_2}} \sqrt{\frac{\log x}{n_L}} }.$$ 
Note that the constant in front of $\sqrt{\frac{\log x}{n_L}}$ in the exponent, as defined in \eqref{def-eps-L-m-x}, is 
\begin{equation}
  \label{exponentialconst}
-\frac{1}{\sqrt{R_2}}  = -0.285\ldots .
\end{equation}
We compute $Y_0 $ and $\lambda_0$ as defined in \eqref{def-Y-L-m-x} and \eqref{def-lambda0}.
Our goal is to minimize the constant in the error term
in the error term of shape $$\mathcal{O}\Big(\lambda_L  n_L^{\frac12 } (\log x)^{\frac12}
e^{-\frac{1}{\sqrt{R_2}} \sqrt{\frac{\log x}{n_L}} }
\Big), $$
namely to minimize $$\min\left( \max(\mathscr{E}_1, \mathscr{E}_2), \tilde{\mathscr{E}_3}(\beta_0) \right)$$ as given in \eqref{def-eps1-L-m-x} \eqref{def-eps2-L-m-x} \eqref{def-epsilontilde3}. 
We also aim for the refined bound \eqref{def2-eps-L-m=1-x} to apply to as many fields as possible, specifically to all fields of degree $n_L\in\{2,\ldots, \mathcal{N}_0\}$.   
Those considerations dictate our choice of $\delta_0$ under the condition \eqref{defining-x_0-delta_0}
$$
\delta_0 \leq 1 - \frac{\sqrt{2}}{x_0}.
$$
A first condition on $\delta_0$ is to make $\min \left( \max \left( \mathscr{E}_1, \mathscr{E}_2 \right), \tilde{\mathscr{E}_3}(\beta_0)\right)$ as small as possible (up to 4 digits of precision).
In addition, when $n_0 \in \{2,\ldots,20\}$, we choose $\delta_0$ such that $n_0\leq \mathcal{N}_0 < n_0+1$,  
and for $n_0=21$, we choose $\delta_0$ to make $\mathcal{N}_0$ as large as possible.
Note that we do these calculations are performed in both cases, whether $\beta_0$ exists or not. 
\subsubsection{Proof of \corref{largelogxcase1}} \label{proof-cor1.2}
The inequalities \eqref{cor11psidbda} \eqref{cor11psidbdb} and \eqref{cor11psidbdc} follow from using the values from Table \ref{beta0-present-cor1.1} for $\alpha$, $\tilde{\mathscr{E}_3}(\beta_0) $ and $\mathscr{E}_3(\beta_0)$ for $n_0 =2$. Note that these values are working for all $n_L\geq 2$ and are arising from the case when $\beta_0$ exists.
\subsection{Explicit estimates for \texorpdfstring{$\psi_C (x)$}{} with  $\log$-type error term}
\label{section-bnd-psi-log}
We deduce the following result from \thmref{main-thm-psi} where the error term for $\psi_C(x)$ is in the from $\frac{1}{ (\log x)^k }$: 
\begin{theorem}\label{form_logx^k}
Let $n_L \geq n_0 \geq 2$, $\Delta_L$ and $\mathscr{M}$ be as defined in \eqref{def-DeltaL} and \eqref{def-Minkowski} respectively.
Let $\zeta_L(s)$ be the associated Dedekind zeta function and let $\beta_0,R_2$ be as defined in Theorem \ref{thmR}. 
Let $\alpha, \delta_0,m, t_0, T_0,\omega_0$, and $x_0$ be as in \propref{prop-bnd-error-delta-L-m-x0-x}.  
Let $k$ be a non-negative integer satisfying 
\begin{equation}
 \label{krange}
 k \leq \frac{1}{2} \left(\frac{1}{\mathscr{M}} \sqrt{\frac{\alpha}{R_2}}  - 1 \right).
\end{equation}
For all $x$ satisfying \eqref{Strong-logxcondition} 
$$\log x \geq \alpha (\log \Delta_L )^2,$$
we have \eqref{main_error}
$$
  E_C (x) 
  \leq  \frac{x^{\beta_0 -1}}{\beta_0} + \mathscr{E}_L(x),
$$
with \begin{equation}
\label{log-gen}
\mathscr{E}_L(x)
 \leq 
 \mathscr{D}_{1,2} (\beta_0,k)
\, \lambda_L n_L^{k+1} \,  \frac{1}{(\log x)^k},
\end{equation}
where
\begin{equation}
\label{def-mathscrDi}
\mathscr{D}_{1,2} = \max(\mathscr{D}_{1},\mathscr{D}_{2}),\   
\mathscr{D}_i(\beta_0,k) 
= \mathscr{E}_i(\beta_0) 
 \left( \frac{\alpha }{ \mathscr{M}^2 }\right)^{k+\frac{1}{2}} e^{- \frac{1}{\sqrt{R_2}}  \frac{\sqrt{\alpha }}{\mathscr{M}}}
 \ \text{ for }i=1,2,
\end{equation}
and $\lambda_L$, $\mathscr{E}_1$, and $\mathscr{E}_2$ are defined in \eqref{def-lambdaL}, \eqref{def-eps1-L-m-x}, and \eqref{def-eps2-L-m-x} respectively.
\\
In addition, if 
   $n_0\leq  n_L \leq \mathcal{N}_0$,
where
$\mathcal{N}_0$ is given in \eqref{def-curvyN0},   
then, for all $x$ satisfying \eqref{Strong-logxcondition}, we have 
\begin{equation}\label{log-strong}
\mathscr{E}_L(x)
 \leq 
 \mathscr{D}_3(\beta_0,k)  \lambda_L^{\frac{1}{2}} n_L^{k+\frac{1}{2}} \frac{1}{(\log x)^k}
\end{equation}
where
\begin{equation}
\label{def-mathscrD3}
\mathscr{D}_3(\beta_0,k)  =  \mathscr{E}_3(\beta_0)  \left( \frac{\alpha }{ \mathscr{M}^2 }\right)^{k+\frac{1}{2}} e^{- \frac{1}{\sqrt{R_2}}  \frac{\sqrt{\alpha }}{\mathscr{M}}},
\end{equation}
where $\mathscr{E}_3$ is defined in \eqref{def-epsilon3}.
Values for $\mathscr{D}_{1,2}$ and $\mathscr{D}_{3}$ in the case $k=1$ may be found in Table \ref{beta0-present-logx-cor1.2}.
\end{theorem}
\begin{proof}[\textbf{Proof of \thmref{form_logx^k}}]
This result follows from  \thmref{thm-psi} / Theorem \ref{main-thm-psi}, with fixed $m=1$.
In particular, we shall bound the function of $x$ in \eqref{def-eps-L-m-x}.
Let $k \geq 1$ and write
\begin{equation}\label{logxform-deduce}
    (\log x)^\frac{1}{2} e^{ - \frac{1}{\sqrt{R_2}}  \sqrt{\frac{\log x}{n_L}} } = \Big( (\log x)^{k+\frac{1}{2}} e^{ - \frac{1}{\sqrt{R_2}}  \sqrt{\frac{\log x}{n_L}} } \Big) \frac{1}{(\log x)^k}.
\end{equation}
We now provide an upper bound for the term within the brackets
\begin{equation}\label{fnxk}
(\log x)^{k+\frac{1}{2}} e^{ - \frac{1}{\sqrt{R_2}}  \sqrt{\frac{\log x}{n_L}} } .
\end{equation}
This is of the shape $g(y)=y^B e^{ - C\sqrt{y} } $, with $B,C >0$, where $g(y)$ decreases when $y>\frac{4B^2}{C^2}$.
Thus, the expression \eqref{fnxk} decreases as long as 
 \begin{equation*}
   \label{condition}
 \log x \geq 4\Big(k+\frac{1}{2} \Big)^2  R_2 n_L .
 \end{equation*}
We note that this is satisfied under the assumption \eqref{Strong-logxcondition} since, by \eqref{krange}, $k$ satisfies  $k  \leq \tfrac{1}{2} ( \frac{1}{\mathscr{M}} \sqrt{\frac{\alpha}{R_2}} - 1 ) $, and thus
$$ 
  \log x    \geq  \alpha n_L (\log \Delta_L)^2 \geq \frac{\alpha  n_L }{\mathscr{M}^2} \geq4 \Big(k+\frac{1}{2} \Big)^2 R_2 n_L  .
 $$ 
As a result, we bound the expression \eqref{fnxk} by its value for $\log x = \frac{\alpha n_L}{ \mathscr{M}^2}$:
$$(\log x)^{k+\frac{1}{2}} e^{ - \frac{1}{\sqrt{R_2}}  \sqrt{\frac{\log x}{n_L}} } \leq \left( \alpha n_L (1/\mathscr{M})^2 \right)^{k+\frac{1}{2}} e^{- \frac{1}{\sqrt{R_2}} \frac{\sqrt{\alpha }}{\mathscr{M}}}.$$
Inserting this estimate in \eqref{logxform-deduce} and then using it in \eqref{def-eps-L-m-x} and \eqref{def2-eps-L-m-x} completes the proof.
\end{proof}
\subsubsection{Calculations for \thmref{form_logx^k} (Table \ref{beta0-present-logx-cor1.2})}
We fix $k=1$.
We use the same values for $m,R_1,R_2,\omega_0,t_0,T_0,\alpha, \delta_0, \mathcal{N}_0, \mathscr{E}_{1}, \mathscr{E}_{2}$ and $\mathscr{E}_3$ as calculated in Table \ref{beta0-present-cor1.1} to calculate $\mathscr{D}_{1,2}$ and $\mathscr{D}_{3}$ as defined in \eqref{def-mathscrDi} and \eqref{def-mathscrD3}.
In addition, if $  n_0 \leq n_L \leq \mathcal{N}_0$, we have \eqref{log-gen}:
$$
\mathscr{E}_L(x)
 \leq \mathscr{D}_{1,2} (\beta_0,1)
\, \lambda_L n_L^{2} \,  \frac{1}{(\log x)} .
$$ 
as well as \eqref{log-strong}: 
$$
\mathscr{E}_L(x)
 \leq \mathscr{D}_3(\beta_0,1)  \lambda_L^{\frac{1}{2}} n_L^{\frac{3}{2}} \frac{1}{(\log x)}.
 $$ 
Thus, in those cases,
$$
\mathscr{E}_L(x)
 \leq    \tilde{\mathscr{D}_3} (\beta_0,1)  \lambda_L n_L^{2} \frac{1}{(\log x)},$$
 where \begin{equation}
\tilde{\mathscr{D}_3} := \frac{\mathscr{D}_3 }{\sqrt{\lambda_0 n_0}} .
 \end{equation} 
\subsubsection{Proof of \corref{form_logx}}
The inequalities \eqref{cor12psidbdb} and \eqref{cor12psidbdc} follow from using the values for $\tilde{\mathscr{D}_3}(\beta_0,1)$ and $\mathscr{D}_3(\beta_0,1)$ from Table \ref{beta0-present-logx-cor1.2} when $n_0 =2$. Note that the values working for all $n_L\geq 2$ arise from the case when $\beta_0$ exists.
\subsection{Explicit estimates for \texorpdfstring{$\psi_C (x)$}{} of the classical exponential shape with absolute constants}
\label{Section75}
The next bounds follow from \thmref{thm-psi} after removing the dependency in $d_L$, by means of bound on $\log x$ combined with the Hermite-Minkowski bounds as given in Table \ref{n0d0}.
\begin{theorem}\label{mainthm2}
Let $n_L \geq n_0 \geq 2$, $\Delta_L$ and $\mathscr{M}$ be as defined in \eqref{def-DeltaL} and \eqref{def-Minkowski} respectively.
Let $\zeta_L(s)$ be the associated Dedekind zeta function and let $\beta_0,R_1,R_2$ be as defined in Theorem \ref{thmR}. 
Let $\alpha, \delta_0,m, t_0, T_0,\omega_0$, and $x_0$ be as in \propref{prop-bnd-error-delta-L-m-x0-x}.  
For all $x$ satisfying \eqref{Strong-logxcondition}
$ \log x 
\geq \alpha m \frac{(\log d_L )^2}{n_L}, 
$
we have \eqref{main_error}
$$
  E_C (x) 
  \leq  \frac{x^{\beta_0 -1}}{\beta_0} + \mathscr{E}_L(x),
$$
with 
\begin{equation}\label{Cor-eq1}
\mathscr{E}_L(\beta_0,m,R_2,x) \leq
\mathscr{C}_{1,2}(\beta_0) 
 n_L^{\frac{3}{2}+\frac{m}{m+1} }
(\log x)^{\frac{1}{2}+\frac{1}{m+1}}
e^{ - \Big( \frac{2}{\sqrt{R_2}} \frac{\sqrt{m}}{m+1} -
\frac{\sqrt{m}}{\sqrt{\alpha}}
\Big)
\sqrt{\frac{\log x}{n_L}} },
\end{equation}
where
\begin{equation}
\label{mathscrC12}
\mathscr{C}_{1,2}(\beta_0)=   \frac{\max \left(\mathscr{E}_1(\beta_0)   , \mathscr{E}_2(\beta_0)  \right)}{e \sqrt{\alpha m}} ,
\end{equation}
and $\mathscr{E}_1(\beta_0)   , \mathscr{E}_2(\beta_0) $ are given  by \eqref{def-eps1-L-m-x}, \eqref{def-eps2-L-m-x}.
\\
In addition, if 
   $n_L \leq \mathcal{N}_0$,
with
$\mathcal{N}_0$ as given in \eqref{def-curvyN0}, then
\begin{equation}\label{cond-no-dL}
\mathscr{E}_L(\beta_0,m,R_2,x) \leq\mathscr{C}_{3}(\beta_0)  n_{L}^{\frac{3}{2(m+1)}}  
 (\log x)^{\frac{3}{2(m+1)}} 
 e^{- \Big(\frac{2}{\sqrt{R_2}} \frac{\sqrt{m}}{m+1} 
 -  \frac{\sqrt{m}}{\sqrt{\alpha}(m+1)} \Big)
 \sqrt{\frac{\log x}{n_L}} } ,
 \end{equation}
 where 
\begin{equation}
\label{mathscrC3}
\mathscr{C}_{3}(\beta_0) = \frac{\mathscr{E}_3(\beta_0)}{(e \sqrt{\alpha m})^{\frac{1}{m+1}}} ,
\end{equation}
and $\mathscr{E}_3(\beta_0)$ is given by \eqref{def-epsilon3}.
Admissible values for $\mathscr{C}_{1,2}$ and $\mathscr{C}_3$ may be found in Table~\ref{beta0-present-no-dL-cor1.2}.
\end{theorem}
\begin{proof}[\textbf{Proof of \thmref{mainthm2}}]
Recall that
$
   \lambda_L = \max \left(  (\log \Delta_L)^2  n_L^2 ,  (\log \Delta_L) \Delta_L ^m \sqrt{ n_L} \right)$.
Note that we have 
$$ 
    (\log \Delta_L)^2  n_L^2  \leq   \frac{1}{e} n_{L}^{\frac{3}{2}} \cdot   (\log \Delta_L) \Delta_L ^m \sqrt{ n_L}
$$ 
 and thus 
$$ 
      \lambda_L \leq  \frac{1}{e} n_{L}^{2} \cdot   (\log \Delta_L) \Delta_L ^m.
$$ 
Since we are in the region $\log x \geq \alpha m n_L  (\log \Delta_L)^2$, it follows that  
\begin{equation*}
      \Delta_L  \le
      \exp \Big(  \sqrt{\frac{\log x}{\alpha m n_L}} \Big) 
  \text{ and }
    \log \Delta_L \leq  \sqrt{\frac{\log x}{\alpha m n_L}}.
\end{equation*}
These combine to give
$$ 
  (\log \Delta_L) \Delta_L ^m  
  \leq   \sqrt{\frac{\log x}{\alpha m n_L}} 
   \exp \Big( m \sqrt{\frac{\log x}{\alpha m n_L}} \Big)  
$$ 
and thus 
\begin{equation*}
  \label{lambdaLmineq}
 \lambda_L
 \leq    \frac{1}{e} n_{L}^{2}  \sqrt{\frac{\log x}{\alpha m n_L}} 
   \exp \Big( m \sqrt{\frac{\log x}{\alpha m n_L}} \Big)
   = \frac{1}{e \sqrt{\alpha m}} n_{L}^{\frac{3}{2}}  (\log x)^{\frac{1}{2}}
   \exp \Big( m \sqrt{\frac{\log x}{\alpha m n_L}} \Big).
\end{equation*}
It then follows from \eqref{def-eps-L-m-x} that 
\begin{equation*}
    \label{def-eps-L-m-xB}
\mathscr{E}_L(\beta_0,m,R_2,x) 
\le
\frac{\max \left(\mathscr{E}_1(\beta_0)   , \mathscr{E}_2(\beta_0)  \right)}{e \sqrt{\alpha m}}
 n_L^{\frac{3}{2}+\frac{m}{m+1} }
(\log x)^{\frac{1}{2}+\frac{1}{m+1}}
e^{ - \Big( \frac{2}{\sqrt{R_2}} \frac{\sqrt{m}}{m+1} -
\frac{\sqrt{m}}{\sqrt{\alpha}}
\Big)
\sqrt{\frac{\log x}{n_L}} } .
\end{equation*}
In addition, if $n_L  \leq \mathcal{N}_0(\alpha, \beta_0, \delta_0, m,\mathscr{M}, R_2, Y_0)$ where 
 $\mathcal{N}_0(\alpha, \beta_0, \delta_0, m,\mathscr{M}, R_2, Y_0)$ is defined by 
\eqref{def-curvyN0},
we have  
\begin{equation}
\label{def2-eps-L-m-x-no-dL}
\mathscr{E}_L(\beta_0,m,R_2,x) 
\leq   \frac{\mathscr{E}_3(\beta_0)}{(e \sqrt{\alpha m})^{\frac{1}{m+1}}} n_{L}^{\frac{3}{2(m+1)}}  
 (\log x)^{\frac{3}{2(m+1)}} 
 e^{- \Big(\frac{2}{\sqrt{R_2}} \frac{\sqrt{m}}{m+1} 
 -  \frac{\sqrt{m}}{\sqrt{\alpha}(m+1)} \Big)
 \sqrt{\frac{\log x}{n_L}} } 
\end{equation}
where $ \mathscr{E}_3(\beta_0)$ is given \eqref{def-epsilon3}.
\end{proof}
\subsubsection{Calculations for \thmref{mainthm2} (Table~\ref{beta0-present-no-dL-cor1.2})}
We use the same values for $m,R_1,R_2,\omega_0,t_0,T_0,\alpha, \delta_0, $ and $\mathcal{N}_0, \mathscr{E}_{1}, \mathscr{E}_{2}$ and $\mathscr{E}_3$ as in Table \ref{beta0-present-cor1.1} to calculate $\mathscr{C}_{1,2}$ and $\mathscr{C}_{3}$ as defined in \eqref{mathscrC12} and \eqref{mathscrC3}.
In particular, for $m=1$, \eqref{Cor-eq1} and \eqref{cond-no-dL} become respectively
\begin{equation}\label{error-C12} 
\mathscr{E}_L(\beta_0,m,R_2,x) \leq 
\mathscr{C}_{1,2}(\beta_0) 
 n_L^{2}
(\log x)
e^{ - \Big( \frac{1}{\sqrt{R_2}}  -
\frac{1}{\sqrt{\alpha}}
\Big)
\sqrt{\frac{\log x}{n_L}} },
\end{equation}
and 
\begin{equation*} 
\mathscr{E}_L(\beta_0,m,R_2,x) \leq 
\mathscr{C}_{3}(\beta_0)  n_{L}^{\frac{3}{4}}  
 (\log x)^{\frac{3}{4}} 
 e^{- \Big(\frac{1}{\sqrt{R_2}}  
 -  \frac{1}{2\sqrt{\alpha}} \Big)
 \sqrt{\frac{\log x}{n_L}} } .
 \end{equation*}
In addition, we note that 
\begin{equation*}
\mathscr{C}_{3}(\beta_0)  n_{L}^{\frac{3}{4}}  
 (\log x)^{\frac{3}{4}} 
 e^{- \Big(\frac{1}{\sqrt{R_2}}  
 -  \frac{1}{2\sqrt{\alpha}} \Big)
 \sqrt{\frac{\log x}{n_L}} }
\leq 
\mathscr{C}_{3}(\beta_0)  n_{L}^{-\frac{5}{4}}  
 (\log x)^{-\frac{1}{4}} 
 e^{ - \frac{1}{2\sqrt{\alpha}} 
 \sqrt{\frac{\log x}{n_L}} }
 \Big(
  n_L^{2}
(\log x)
e^{ - \big( \frac{1}{\sqrt{R_2}}  -
\frac{1}{\sqrt{\alpha}}
\big)
\sqrt{\frac{\log x}{n_L}} }
\Big),
\end{equation*}
where
\begin{equation*}
\label{factor-adj-C3}
n_{L}^{-\frac{5}{4}}  
 (\log x)^{-\frac{1}{4}} 
 e^{ - \frac{1}{2\sqrt{\alpha}} 
 \sqrt{\frac{\log x}{n_L}} }
 \le
n_{L}^{-\frac{5}{4}}  
 \Big(\alpha \frac{(\log d_L )^2}{n_L}\Big)^{-\frac{1}{4}} 
 e^{ - \frac{1}{2\sqrt{\alpha}} 
 \sqrt{ \alpha \frac{(\log d_L )^2}{n_L^2} } }
\leq 
 \alpha^{-\frac{1}{4}} 
 n_0^{-3/2} 
  \sqrt{ \mathscr{M} } 
 e^{ - \frac{1}{2 \mathscr{M}}} 
\end{equation*}
since
$\log x\geq \alpha \frac{(\log d_L )^2}{n_L}$.
For sake of comparison, we calculate 
\begin{equation}
\label{def-tildeC3}
\tilde{\mathscr{C}_3}(\beta_0) :=  \mathscr{C}_3(\beta_0)  
 \alpha^{-\frac{1}{4}} 
 n_0^{-3/2} 
  \sqrt{ \mathscr{M} } 
 e^{ - \frac{1}{2 \mathscr{M}}} 
\end{equation}
in order to compare it to $\mathscr{C}_{1,2}(\beta_0)$.
Thus, for $2\leq  n_0\leq  n_L \leq \mathcal{N}_0$, we have 
\begin{equation}\label{error-tildeC3}
\mathscr{E}_L(x)
 \leq \tilde{\mathscr{C}_3}(\beta_0) \, n_L^{2}
(\log x)
e^{ - \Big( \frac{1}{\sqrt{R_2}} -
\frac{1}{\sqrt{\alpha}}
\Big)
\sqrt{\frac{\log x}{n_L}} }. 
\end{equation}
\subsubsection{Proof of \corref{largelogxcase2}}\label{proofCor13}
We fix $m=1$. 
We use here the values calculated in Table \ref{beta0-present-no-dL-cor1.2}. 
Note that for $n_0=2, \ldots20$, $n_0\le n_L\le \mathcal{N}_0$ actually gives $n_L=n_0$, and thus the bound \eqref{error-tildeC3} is only valid for the values of $\tilde{\mathscr{C}_3}$ and  $\frac{1}{\sqrt{R_2}} - \frac{1}{\sqrt{\alpha}} $ at $n_L=n_0$.
On the other hand, the inequality \eqref{error-C12} is valid for values of $ \mathscr{C}_{1,2}$ and $\frac{1}{\sqrt{R_2}} - \frac{1}{\sqrt{\alpha}} $ for all $n_L \geq n_0$.
As the degree changes, the exponential power $\frac{1}{\sqrt{R_2}} - \frac{1}{\sqrt{\alpha}} $ also changes, so in this Table, we can not directly compare values of $\tilde{\mathscr{C}_3}$ and $ \mathscr{C}_{1,2}$.
Thus, the inequalities \eqref{cor13psidbda} and \eqref{cor13psidbdb} follow from using the values of $\mathscr{C}_{1,2}$ and $\mathscr{C}_3(\beta_0)$ when $n_0 =2$.

\subsubsection{From  \thmref{mainthm2} to  \corref{corpsi-main2}}
\label{cor35.1}
The following is a detailed version of \corref{corpsi-main2}.
\begin{corollary}\label{corpi-main2-gen}
We assume the same conditions as \thmref{mainthm2} and that there exists positive constants $\alpha, A,B,\mathcal{C}, \mathcal{D}$, such that 
for all $x$ satisfying
$$ \log x 
\geq  \alpha   \frac{(\log d_L )^2}{n_L}, 
$$
we have
\begin{equation}
  \label{ECxbd}
    E_C(x)
     \leq 
           \frac{x^{\beta_0 -1}}{\beta_0} +  \mathcal{C}   \  n_L^A (\log x)^B
e^{-  \mathcal{D}   \sqrt{\frac{\log x}{n_L}}}.
\end{equation}
Thus, for all $x$ satisfying 
$$
\log x 
\geq  c_0   n_L (\log d_L )^2 , 
$$
and for any $0< b_0 < \mathcal{D}  $, we have
\begin{equation}
  E_C(x) \leq \frac{x^{\beta_0 -1}}{\beta_0} + a_0  e^{ -   b_0  \sqrt{\frac{\log x}{n_L}} },
\end{equation}
for all $n_L \geq n_0$, 
where
\begin{equation}\label{def-c0-a0}
 c_0  = \frac{ \alpha  }{n_0^2} \ \text{and}\ 
 a_0  =  \frac{ \mathcal{C}   \mathscr{M}^{\frac23A}}{ c_0  ^{\frac{A}3}}
  \max_{\log x \geq \frac{ c_0   n_0^3}{\mathscr{M}^{2}} }  \Big(   (\log x)^{\frac{A}{3}+B} e^{-  \frac{( \mathcal{D}   -  b_0 )  c_0  ^{1/6}}{ \mathscr{M}^{1/3}} (\log x)^\frac{1}{3} } \Big).
\end{equation}
Values for $(a_0,b_0,c_0)$ can be found in Tables \ref{beta0-present-cor2.4-for-nL>=n0} and \ref{beta0-present-cor2.4}. 
\end{corollary}
\begin{proof}[\textbf{Proof of \corref{corpi-main2-gen}}]
Under the conditions of \thmref{mainthm2}, there exist positive constants $\alpha, A, B, \mathcal{C} ,  \mathcal{D}  $ such that, if 
\begin{equation}
  \label{F0range}
  \log x \geq  \alpha   \frac{(\log d_L)^2}{n_L}, 
\end{equation}
then
\begin{equation}
  \label{ECxbd}
    E_C(x)
     \leq 
           \frac{x^{\beta_0 -1}}{\beta_0} +  \mathcal{C}   \  n_L^A (\log x)^B
e^{-  \mathcal{D}   \sqrt{\frac{\log x}{n_L}}}.
\end{equation}
For $m=1$, \thmref{mainthm2} establishes a bound of this type with admissible constants $(A,B) = (2,1)$ or $(A,B) = (\tfrac{3}{4}, \tfrac{3}{4})$ (depending on whethere $n_L < \mathcal{N}_L$ or not). 
Thus, if  
\begin{equation}
 \label{weakerbd}
\log x \geq  c_0   n_L (\log d_L)^2,  \ \text{ with } c_0   = \frac{ \alpha  }{n_0^2}, 
\end{equation} 
then the condition on $x$ \eqref{F0range} holds.  We now establish a bound for the error term 
in \eqref{ECxbd} 
of the shape $a e^{-b\sqrt{\frac{\log x}{n_0}}}$ for absolute constants $a$ and $b$
which are independent of $n_L$.
Combining the Minkowski bound  \eqref{def-Minkowski} with \eqref{weakerbd}, we have that 
\begin{equation}\label{logx-lowerbound}
    \log x \geq  c_0   n_L ( n_L / \mathscr{M})^2 = \frac{ c_0  }{\mathscr{M}^{2}} n_L^3.
\end{equation}
This implies
\begin{equation} 
\label{nLbdb}
n_L \leq  \Big( \frac{\mathscr{M}^2 \log x}{ c_0  }  \Big)^{\frac{1}{3}}
\ \text{ and }\ 
\sqrt{ \frac{\log x}{n_L} } \geq   \frac{ c_0^{\frac{1}{6}} (\log x)^{\frac{1}{3}} }{\mathscr{M}^{\frac{1}{3}} } .
\end{equation}
We assume $ b_0  <  \mathcal{D}  $, and combine these bounds with \eqref{ECxbd}. We deduce
\begin{equation}
\begin{split} 
\label{inequalities}
     \mathcal{C}   \ n_{L}^A (\log x)^B  e^{ -   \mathcal{D}   \sqrt{\frac{\log x}{n_L}} } 
   & \leq  \mathcal{C}   \Big( \frac{\mathscr{M}^2 \log x}{ c_0  } \Big)^{\frac{A}{3}} (\log x)^B 
    e^{ -  ( \mathcal{D}   -  b_0 )  \frac{ c_0^{\frac{1}{6}} (\log x)^{\frac{1}{3}} }{\mathscr{M}^{\frac{1}{3}} }  } e^{ -   b_0  \sqrt{\frac{\log x}{n_L}} } \\
    & \leq  a_0    \exp \Big( -   b_0  \sqrt{\frac{\log x}{n_L}} \Big)
\end{split}
\end{equation}
where 
\begin{equation*}
  \label{H2}
   a_0  =   \frac{ \mathcal{C}   \mathscr{M}^{\frac23A}}{ c_0  ^{\frac{A}3}}
  \max_{\log x \geq \frac{ c_0   n_0^3}{\mathscr{M}^{2}} }  \Big(   (\log x)^{\frac{A}{3}+B} e^{-  \frac{( \mathcal{D}   -  b_0 )  c_0  ^{1/6}}{ \mathscr{M}^{1/3}} (\log x)^\frac{1}{3} } \Big).
\end{equation*}
Since $-( \mathcal{D}   -  b_0)$ is negative , the above function possesses an absolute
maximum and thus $a_0$ is an absolute constant. 
\end{proof}
\subsubsection{Calculations for \corref{corpi-main2-gen}}\label{subsection-7.5.3} (Tables \ref{beta0-present-cor2.4} and \ref{beta0-present-cor2.4-for-nL>=n0}) 
It follows from \thmref{corpi-main2-gen} that, for each $n_L\geq 2$, if $\log x \geq  c_0   n_L (\log d_L)^2$, then 
$$E_C(x)
 \leq \frac{x^{\beta_0 -1}}{\beta_0} + \mathscr{E}_L(x)
,\ \text{ with }\ 
\mathscr{E}_L(x) \leq a_0  e^{ -   b_0  \sqrt{\frac{\log x}{n_L}} } , $$
where we fix an admissible value for $b_0$ and calculate $(a_0,c_0)$ as defined by \eqref{def-c0-a0} using Table \ref{beta0-present-no-dL-cor1.2}'s values for $(\alpha,\mathcal{N}_0, \mathscr{C}_{1,2}(\beta_0), \mathscr{C}_3(\beta_0), \tilde{\mathscr{C}_3}(\beta_0))$.
\begin{enumerate}
\item {\it Calculations for Table \ref{beta0-present-cor2.4}:}\\
For each $n_L = n_0 \in \{2, \ldots, 20 \}$, and for $n_L$ in the range $21$ to $\mathcal{N}_0$,
we fix $$( \mathcal{D}  , b_0 ) = \Big( \frac{1}{\sqrt{R_2}} -  \frac{1}{2\sqrt{\alpha}}, 0.25 \Big)$$
and we calculate $(a_0,c_0)$ using 
\begin{equation}\label{values1}
(A,B, \mathcal{C}   ) = \Big(\frac34,\frac34,\mathscr{C}_3(\beta_0)  \Big).
\end{equation}
Note that the choice $b_0=0.25$ satisfies the condition $b_0<\mathcal{D} $ as $\mathcal{D} = \frac{1}{\sqrt{R_2}}  -  \frac{1}{2\sqrt{\alpha}}\ge 0.258\ldots$.
In addition, we separate the cases whether $\beta_0$ exists or not as $\mathcal{N}_0$ takes different values, namely $654.65$ and $519.59$ respectively.
\newline
Finally, for $n_L \geq  \mathcal{N}_0$, we use the values
\begin{equation}\label{values2}
(A,B, \mathcal{C}  , \mathcal{D}  , b_0 ) = \Big(2,1,\mathscr{C}_{1,2}(\beta_0), \frac{1}{\sqrt{R_2}}  -  \frac{1}{\sqrt{\alpha}}=0.23122\ldots , 0.23 \Big).
\end{equation}
\item {\it Calculations for Table \ref{beta0-present-cor2.4-for-nL>=n0}:}\\
We use the values \eqref{values2} to calculate $(a_0,c_0)$ for all $n_L\ge n_0$ with $n_0=2,\ldots, 21.$
\end{enumerate}

\subsubsection{Proof of \corref{corpsi-main2}.} \label{proofCor14}
The bound \eqref{cor14psidbda} follows from the case $n_0 =2$ in Table \ref{beta0-present-cor2.4-for-nL>=n0}, while \eqref{cor14psidbdb} follows from taking the maximum of $a_0$ and $c_0$ in the range $n_L \leq 519$ from Table \ref{beta0-present-cor2.4}. 

\ \newpage
\bibliographystyle{plain}
\bibliography{mybib.bib}
\ \newpage 
\appendix
\section{Appendix of Tables}
\label{appendixtables}
All numerical results are rounded up to either 3 or 4 digits of precision.
\begin{table}[h!]
\centering
\caption{ \small Numerical bounds for $N_L(T)$ from \thmref{thm-bnd-nchiNL} and from \thmref{theohsw}: 
\begin{displaymath} 
N_L(T) \leq \alpha_0(T)( \log d_L) \ \text{ and } \ N_L(T) \leq \alpha_0'(T) (\log d_L).
\end{displaymath} 
}
\label{Table-NL-1-2} 
\begin{tabular}{|r||r|r|r|r|r|}
        \hline
$n_0$  & $\alpha_0(1/2)$ & $\alpha_0(1)$ & $\alpha_0(2)$ & $\alpha_0'(1)$ & $\alpha_0'(2)$
         \\ \hline
2	&	36.0416	&	40.1778	&	52.4347	&	45.0838	&	44.8487\\
3	&	18.5920	&	20.7192	&	27.0288	&	23.2331	&	23.2605 \\
4	&	15.9830	&	17.8055	&	23.2168	&	20.1463	&	20.2091 \\
5	&	13.0111	&	14.4971	&	18.9071	&	16.1955	&	16.3079 \\
6	&	12.4404	&	13.8590	&	18.0713	&	15.5438	&	15.6635 \\
7	&	11.1248	&	12.3959	&	16.1678	&	13.7372	&	13.8800 \\
8	&	10.9182	&	12.1647	&	15.8643	&	13.5161	&	13.6612 \\
9	&	10.2616	&	11.4346	&	14.9148	&	12.6055	&	12.7623 \\
10	&	10.1957	&	11.3602	&	14.8164	&	12.5560	&	12.7132 \\
11	&	9.6703	&	10.7765	&	14.0581	&	11.8096	&	11.9765 \\
12	&	9.6569	&	10.7608	&	14.0364	&	11.8212	&	11.9877 \\
13	&	9.2696	&	10.3307	&	13.4778	&	11.2655	&	11.4392 \\
14	&	9.2793	&	10.3410	&	13.4902	&	11.3018	&	11.4750 \\
15	&	8.9781	&	10.0066	&	13.0560	&	10.8663	&	11.0452 \\
16	&	9.0005	&	10.0310	&	13.0870	&	10.9163	&	11.0945 \\
17	&	8.7555	&	9.7590    &	12.7340	&	10.5595	&	10.7423 \\
18	&	8.7833	&	9.7896	&	12.7731	&	10.6141	&	10.7961 \\
19	&	8.5795	&	9.5634	&	12.4795	&	10.3158	&	10.5017 \\
20	&	8.6100       &	9.5970     &	12.5226	&	10.3720	&	10.5571 \\
21	&	8.5458	&	9.5256	&	12.4297	&	10.2832	&	10.4694 \\
\hline 
    \end{tabular}
\end{table}
\ \newline
We observe that, when $n_0=2$, the bound $\alpha'$ decreases and then increases, and that $\alpha_0'(2)<\alpha_0'(1) $. 
Thus we can use the value $\alpha_0'(2)=44.8487$ to bound $N_L(1)$.

\begin{table}[h!]
\caption{\small Values for $ \omega_0$ and $t_0$ such that \eqref{cond-omega-t} is satisfied (\lmaref{mlu}).}
\label{Table-t0-omega0}
\begin{tabular}{cc}
\begin{tabular}{|c|c|}
    \hline
    $\omega_0$ & $t_0$  
    \\
    \hline
     1 & $39.217$ \\ 
    2 & $22.479$ \\ 
    3 & $16.364$ \\
    4 & $13.116$ \\
    5 & $11.077$ \\
    6 & $9.666$ \\
    7 & $8.625$ \\
    8 & $7.823$ \\
    9 & $7.183$ \\
    10 & $6.660$ \\
    \hline
\end{tabular}
&
\begin{tabular}{|c|c|}
    \hline
    $t_0$ & $\omega_0$ 
    \\
    \hline
    1 & $291.601$ \\ 
    2 & $64.860$ \\
    3 & $32.718$ \\
    4 & $20.945$ \\
    5 & $15.053$ \\
    6 & $11.586$ \\
    7 & $9.330$ \\
    8 & $7.758$ \\
    9 & $6.607$ \\
    10 & $5.732$ \\
    \hline 
\end{tabular}
\end{tabular}
\end{table}
For the remaining Tables, we give numerical bounds to $\mathscr{E}_L(x)$ where 
$$      E_C (x) := \frac{\left| \psi_C(x) - \frac{|C|}{|G|} x \right|}{\frac{|C|}{|G|} x} 
  \leq  \frac{x^{\beta_0 -1}}{\beta_0} + \mathscr{E}_L(x).
$$
\begin{table}[h!]
\setlength\tabcolsep{4pt} 
\centering
\caption{ \small Numerical bounds for $\mathscr{E}_L(x)$ from \thmref{main-thm-psi} and \ref{thm-psi}:  
If $n_L\geq n_0$ and if  $\log x \geq \alpha n_L (\log \Delta_L )^2,$ then 
\begin{displaymath} \mathscr{E}_L (x)   \leq \max \left( \mathscr{E}_1 , \mathscr{E}_2 \right) \, \lambda_L  n_L^{\frac12 } \, (\log x)^{\frac12} e^{-\frac{1}{\sqrt{R_2}} \sqrt{\frac{\log x}{n_L}} } .  \end{displaymath}
In addition, if $ n_0 \leq n_L  \leq \mathcal{N}_0, $ then 
\begin{displaymath} \mathscr{E}_L  (x)   \leq \mathscr{E}_3 \, \lambda_L^{\frac{1}{2}} \, (\log x)^{\frac{1}{2}} e^{-\frac{1}{\sqrt{R_2}} \sqrt{\frac{\log x}{n_L}} } \ \text{ and }\  \mathscr{E}_L  (x)  \leq  \tilde{\mathscr{E}_3} \,  \lambda_L  n_L^{\frac12 } \, (\log x)^{\frac12} e^{-\frac{1}{\sqrt{R_2}} \sqrt{\frac{\log x}{n_L}} }. \end{displaymath}
}
\label{beta0-present-cor1.1}
    \resizebox{\linewidth}{!}{%
      \begin{tabular}{|c|c|c||c|c|c|c|c||c|c|c|c|c|}
\toprule
\multicolumn{3}{|c||}{} 
      &\multicolumn{5}{c||}{When $\beta_0$ exists} 
      &\multicolumn{5}{c@{}|}{When $\beta_0$ does not exist}\\
\midrule
$n_0$	&	$\alpha$	&	$\log x_0$		
& $\delta_0$	&	$\max \left( \mathscr{E}_1 , \mathscr{E}_2 \right)$	&	$\mathcal{N}_0$	&	$\mathscr{E}_3$	& $\tilde{\mathscr{E}_3}$ &	$\delta_0$	&	$\max \left( \mathscr{E}_1 , \mathscr{E}_2 \right)$	&	$\mathcal{N}_0$	&	$\mathscr{E}_3$ &  $\tilde{\mathscr{E}_3}$ 	\\
\midrule
2	&	2914.82	&	1759	&	2.26E-03	&	0.28649	&	2.003	&	0.44511	& 0.27134 &	3.20E-03	&	0.20275	&	2.003	&	0.31501	& 0.19203 \\
3	&	1004.56	&	3292	&	1.92E-03	&	0.05216	&	3.005	&	0.28090	&	0.05172 & 2.71E-03	&	0.03695	&	3.000	&	0.19872	& 0.03659 \\
4	&	822.49	&	4663.1	&	2.26E-03	&	0.02557	&	4.009	&	0.24256	& 0.02547 &	3.19E-03	&	0.01812	&	4.003	&	0.17160	& 0.01802 \\
5	&	599.17	&	6532.6	&	2.02E-03	&	0.01225	&	5.010	&	0.20219	& 0.01225 &	2.85E-03	&	8.6637 E-03	&	5.001	&	0.14304	& 8.6637 E-03 \\
6	&	569.3	&	8004.2	&	2.35E-03	&	8.4069E-03	&	6.002	&	0.18914	& 8.4069E-03 &	3.33E-03	&	5.9476E-03	&	6.007	&	0.13381 & 5.9476E-03	\\
7	&	479.55	&	1.0073 E4	&	2.17E-03	&	5.2678E-03	&	7.021	&	0.16900	& 5.2678E-03 &	3.06E-03	&	3.7266E-03	&	7.006	&	0.11956	& 3.7266E-03	\\
8	&	470.92	&	1.1611 E4	&	2.48E-03	&	4.0823E-03	&	8.016	&	0.16217	&	4.0823E-03 & 3.50E-03	&	2.8882E-03	&	8.003	&	0.11473	&	2.8882E-03\\
9	&	428.75	&	1.3681 E4	&	2.41E-03	&	2.9687E-03	&	9.015	&	0.15093	&	2.9687E-03 & 3.41E-03	&	2.1003E-03	&	9.015	&	0.10678	&	2.1003E-03 \\
10	&	427.67	&	1.5221 E4	&	2.73E-03	&	2.4617E-03	&	10.002	&	0.14686	&	2.4617E-03 & 3.87E-03	&	1.7417E-03	&	10.014	&	0.10391	&	1.7417E-03 \\
11	&	394.15	&	1.7480 E4	&	2.61E-03	&	1.8846E-03	&	11.026	&	0.13805	& 1.8846E-03 &	3.69E-03	&	1.3333E-03	&	11.019	&	0.09767	& 1.3333E-03\\
12	&	395.45	&	1.9035 E4	&	2.93E-03	&	1.6252E-03	&	12.026	&	0.13531	&	1.6252E-03 & 4.14E-03	&	1.1499E-03	&	12.014	&	0.09574	& 1.1499E-03\\
13	&	371.19	&	2.1360 E4	&	2.81E-03	&	1.3049E-03	&	13.009	&	0.12868	&	1.3049E-03 & 3.98E-03	&	9.2322E-04	&	13.017	&	0.09104	& 9.2322E-04\\
14	&	373.33	&	2.2928 E4	&	3.13E-03	&	1.1547E-03	&	14.028	&	0.12669	& 1.1547E-03& 	4.42E-03	&	8.1702E-04	&	14.009	&	0.08964	& 8.1702E-04\\
15	&	354.71	&	2.5305 E4	&	3.02E-03	&	9.5867E-04	&	15.004	&	0.12146	&	9.5867E-04 & 4.28E-03	&	6.7832E-04	&	15.019	&	0.08594 & 6.7832E-04	\\
16	&	357.25	&	2.6879 E4	&	3.33E-03	&	8.6431E-04	&	16.002	&	0.11995	&	8.6431E-04 & 4.72E-03	&	6.1159E-04	&	16.019	&	0.08488	& 6.1159E-04\\
17	&	342.25	&	2.9300 E4	&	3.24E-03	&	7.3537E-04	&	17.019	&	0.11567	& 7.3537E-04 &	4.58E-03	&	5.2033E-04	&	17.007	&	0.08185	& 5.2033E-04\\
18	&	344.86	&	3.0882 E4	&	3.55E-03	&	6.7207E-04	&	18.031	&	0.11448	& 6.7207E-04	& 5.02E-03	&	4.7558E-04	&	18.021	&	0.08101	& 4.7558E-04\\
19	&	336.07	&	3.3696 E4	&	3.26E-03	&	5.8200E-04	&	19.035	&	0.11073	&	5.8200E-04 & 4.61E-03	&	4.1182E-04	&	19.025	&	0.07835	& 4.1182E-04\\
20	&	337.66	&	3.5193 E4	&	3.60E-03	&	5.3776E-04	&	20.009	&	0.10980	&	5.3776E-04 & 5.10E-03	&	3.8054E-04	&	20.022	&	0.07770	& 3.8054E-04\\
\midrule
21	&	335.48	&	3.7351 E4	&	0.99999	&	7.6933E-04	&	654.650	&	0.17047	& 7.6933E-04 &	0.99999	&	5.4400 E-04	&	519.59	&	0.12054	& 5.4400 E-04\\
\bottomrule
\end{tabular}
}
\end{table}

\begin{table}[h!]
\setlength\tabcolsep{4pt} 
\centering
\caption{ \small Numerical bounds for $\mathscr{E}_L(x)$ from \thmref{form_logx^k}:
If $n_L\geq n_0$ and if $\log x \geq \alpha n_L (\log \Delta_L )^2,$
then
\begin{displaymath}
\mathscr{E}_L (x)  \leq   \mathscr{D}_{1,2} \lambda_L n_L^2 (\log x)^{-1} .
\end{displaymath}
In addition, if 
$
n_0 \leq n_L  \leq \mathcal{N}_0,
$
then 
\begin{displaymath}
 \mathscr{E}_L (x) 
  \leq   \mathscr{D}_{3} \, \lambda_L^{\frac12} n_L^{\frac32} \,  (\log x)^{-1}
\ \text{ and }\  
 \mathscr{E}_L (x) 
  \leq
 \tilde{\mathscr{D}_{3}} \,  \lambda_L n_L^2 \,  (\log x)^{-1} .
\end{displaymath}
}
\label{beta0-present-logx-cor1.2}
    \resizebox{\linewidth}{!}{%
      \begin{tabular}{|c|c||c|c|c|c||c|c|c|c|}
\toprule
\multicolumn{2}{|c||}{} 
      &\multicolumn{4}{c||}{When $\beta_0$ exists} 
      &\multicolumn{4}{c@{}|}{When $\beta_0$ does not exist}\\
\midrule
$n_0$ & $\alpha$ & $\mathscr{D}_{1,2}$ & $\mathcal{N}_0$ & $\mathscr{D}_3 $  & $ \tilde{\mathscr{D}_{3}}$ & $\mathscr{D}_{1,2}$ & $\mathcal{N}_0$ & $\mathscr{D}_3 $  & $  \tilde{\mathscr{D}_{3}}$ \\
\midrule
2	& 2914.82 &	1.5568	& 2.003 & 2.4187	&	1.4744	&	1.1018 & 2.003	&	1.7118	&	1.0435 \\
3	& 1004.56 &	1.4654 E-1 & 3.005 &	 7.8910 E-1	&	1.4530 E-1	&	1.0379 E-1	& 3.000 &	5.5825 E-1	&	1.0279 E-1 \\
4	&	822.49 & 5.8817 E-2	& 4.009 & 	5.5787 E-1	& 5.8573 E-2	&	4.1663 E-2	& 4.003 & 3.9468 E-1	&	4.1439 E-2 \\
5	&	599.17 & 1.8856 E-2	&	5.010 & 3.1131 E-1	& 1.8856 E-2	&	1.3339 E-2	&	5.001 & 2.2023 E-1	&	1.3339 E-2 \\
6	&	569.30 & 1.1985 E-2	&	6.002 & 2.6964 E-1	&	1.1985 E-2	&	8.4791 E-3	& 6.007 &	1.9076 E-1	&	8.4791 E-3 \\
7	&	479.55 & 5.6230 E-3	& 7.021 &	1.8040 E-1	&	5.6230 E-3	&	3.9779 E-3	&	7.006 & 1.2762 E-1	&	3.9779 E-3 \\
8	&	470.92 & 4.2129 E-3	&	8.016 & 1.6736 E-1	&	4.2129 E-3	&	2.9805 E-3	&	8.003 & 1.1840 E-1	&	2.9805 E-3 \\
9	&	428.75 & 2.5453 E-3	&	9.015 & 1.2940 E-1	& 2.5453 E-3	&	1.8007 E-3	& 9.015 & 	9.1550 E-2	&	1.8007 E-3 \\
10	& 427.67 & 2.0996 E-3	& 10.002 &	1.2526 E-1	&	2.0996 E-3	&	1.4855 E-3	&	10.014 & 8.8623 E-2	&	1.4855 E-3 \\
11	& 394.15 & 1.3451 E-3	&	11.026 & 9.8535 E-2	&	1.3451 E-3		&	9.5167 E-4	&	11.019 & 6.9713 E-2	&	9.5167 E-4 \\
12	& 395.45 & 1.1687 E-3	&	12.026 & 9.7304 E-2	&	1.1687 E-3	&	8.2693 E-4	&	12.014 & 6.8847 E-2	&	8.2693 E-4 \\
13	& 371.19 & 	8.0811 E-4	& 13.009 &	7.9693 E-2	&	8.0811 E-4	&	5.7176 E-4	& 13.017 & 	5.6385 E-2	&	5.7176 E-4 \\
14	&	373.33 & 7.2519 E-4	&	14.028 & 7.9565 E-2	&	7.2519 E-4	&	5.1312 E-4	&	14.009 & 5.6298 E-2	&	5.1312 E-4 \\
15	&	354.71 & 5.2964 E-4	&	15.004 & 6.7102 E-2	&	5.2964 E-4	&	3.7475 E-4	&	15.019 & 4.7479 E-2	&	3.7475 E-4 \\
16	&	357.25 & 4.8636 E-4	&	16.002 & 6.7498 E-2	&	4.8636 E-4	&	3.4415 E-4	&	16.019 & 4.7762 E-2	&	3.4415 E-4 \\
17	&	342.25 & 3.6968 E-4	&	17.019 & 5.8148 E-2	&	3.6968 E-4	&	2.6158 E-4	&	17.007 & 4.1145 E-2	&	2.6158 E-4 \\
18	&	344.86 & 3.4481 E-4	&	18.031 & 5.8734 E-2	&	3.4481 E-4	&	2.4400 E-4	&	18.021 & 4.1562 E-2	&  2.4400 E-4 \\
19	&	336.07 & 2.5749 E-4	&	19.035 & 4.8988 E-2	&	2.5749 E-4	&	1.8220 E-4	&	19.025 & 3.4663 E-2	&	1.8220 E-4 \\
20	&	337.66 & 2.4644 E-4	&	20.009 & 5.0320 E-2	&	2.4644 E-4	&	1.7439 E-4	& 20.022 & 	3.5608 E-2	&	1.7439 E-4 \\
\midrule
21	&	335.48 & 3.3591 E-4	& 654.65 &	7.4434 E-2	&	3.3591 E-4	&	2.3753 E-4	&	519.59 & 5.2633 E-2	&	2.3753 E-4 \\
\bottomrule
\end{tabular}
}
\end{table}
\begin{table}[h!]
\setlength\tabcolsep{4pt} 
\centering
\caption{ \small Numerical bounds for $\mathscr{E}_L(x)$ from \thmref{mainthm2}: 
If $n_L\geq n_0$ and if $\log x \geq \alpha n_L (\log \Delta_L )^2,$
then
\begin{displaymath}
\mathscr{E}_L (x)  \leq \mathscr{C}_{1,2} \, n_L^{2}  \, (\log x) e^{ - \Big( \frac{1}{\sqrt{R_2}} - \frac{1}{\sqrt{\alpha}} \Big)
\sqrt{\frac{\log x}{n_L}} }.
\end{displaymath}
In addition, if 
$n_0 \leq n_L  \leq \mathcal{N}_0,$
then 
\begin{displaymath}
 \mathscr{E}_L (x) 
  \leq  \mathscr{C}_{3}  \, n_{L}^{\frac{3}{4}}  \, (\log x)^{\frac{3}{4}} e^{- \big(\frac{1}{\sqrt{R_2}} -  \frac{1}{2\sqrt{\alpha}} \big) \sqrt{\frac{\log x}{n_L}} } 
 \text{ and }
 \mathscr{E}_L (x) 
  \leq
   \tilde{\mathscr{C}_{3}}  \, n_L^{2}  \,  (\log x) e^{ - \big( \frac{1}{\sqrt{R_2}} - \frac{1}{\sqrt{\alpha}} \big) \sqrt{\frac{\log x}{n_L}} } .
\end{displaymath}
}
\label{beta0-present-no-dL-cor1.2}
    \resizebox{\linewidth}{!}{%
      \begin{tabular}{|c|c|c|c||c|c|c|c||c|c|c|c|}
\toprule
\multicolumn{4}{|c||}{} 
      &\multicolumn{4}{c||}{When $\beta_0$ exists} 
      &\multicolumn{4}{c@{}|}{When $\beta_0$ does not exist}\\
\midrule
$n_0$ 	&	$\alpha$	&	$\frac{1}{\sqrt{R_2}}  -  \frac{1}{\sqrt{\alpha}} $	&	$ \frac{1}{\sqrt{R_2}} -  \frac{1}{2\sqrt{\alpha}}  $	
&	$\mathcal{N}_0$	&	$\mathscr{C}_{1,2} $	&	$\mathscr{C}_3 $	&	$\tilde{\mathscr{C}_3} $	
&	$\mathcal{N}_0$	&	$\mathscr{C}_{1,2} $	&	$\mathscr{C}_3 $	&	$\tilde{\mathscr{C}_3} $	\\
\midrule
2	&	2,914.82	&	0.26730	&	0.27656	&	2.003	&	1.952 E-3	&	3.674 E-2	&	1.813 E-3	&	2.003	&	1.382 E-3	&	2.600 E-2	&	1.283 E-3	\\
3	&	1,004.56	&	0.25427	&	0.27004	&	3.005	&	6.055 E-4	&	3.026 E-2	&	6.001 E-4	&	3.000	&	4.289 E-4	&	2.141 E-2	&	4.245 E-4	\\
4	&	822.49	&	0.25095	&	0.26838	&	4.009	&	3.280 E-4	&	2.747 E-2	&	3.241 E-4	&	4.003	&	2.324 E-4	& 1.944 E-2	&	2.293 E-4	\\
5	&	599.17	&	0.24497	&	0.26539	&	5.010	&	1.841 E-4	&	2.479 E-2	&	1.762 E-4	&	5.001	&	1.302 E-4	& 1.754 E-2	&	1.247 E-4	\\
6	&	569.30	&	0.24391	&	0.26486	&	6.002	&	1.296 E-4	&	2.349 E-2 &	1.230 E-4	&	6.007	&	9.170 E-5	&	1.662 E-2	&	8.701 E-5	\\
7	&	479.55	&	0.24015	&	0.26299	&	7.021	&	8.850 E-5	&	2.190 E-2	&	8.076 E-5	&	7.006	&	6.260 E-5	&	1.550 E-2	&	5.714 E-5	\\
8	&	470.92	&	0.23974	&	0.26278	&	8.016	&	6.921 E-5	&	2.111 E-2	&	6.285 E-5	&	8.003	&	4.896 E-5	&	1.494 E-2	&	4.446 E-5	\\
9	&	428.75	&	0.23752	&	0.26167	&	9.015	&	5.274 E-5	&	2.012 E-2	&	4.654 E-5	&	9.015	&	3.731 E-5	&	1.423 E-2	&	3.293 E-5	\\
10	&	427.67	&	0.23746	&	0.26164	&	10.002	&	4.380 E-5	&	1.959 E-2	&	3.861 E-5	&	10.014	&	3.098 E-5	&	1.386 E-2	&	2.732 E-5	\\
11	&	394.15	&	0.23545	&	0.26063	&	11.026	&	3.492 E-5	&	1.879 E-2&	2.990 E-5	&	11.019	&	2.471 E-5	&	1.323 E-2	&	2.115 E-5	\\
12	&	395.45	&	0.23553	&	0.26067	&	12.026	&	3.007 E-5	&	1.840 E-2	&	2.577 E-5	&	12.014	&	2.127 E-5	&	1.302 E-2	&	1.824 E-5	\\
13	&	371.19	&	0.23391	&	0.25987	&	13.009	&	2.492 E-5	&	1.778 E-2	&	2.081 E-5	&	13.017	&	1.763 E-5	&	1.258 E-2	&	1.472 E-5	\\
14	&	373.33	&	0.23406	&	0.25994	&	14.028	&	2.199 E-5	&	1.748 E-2	&	1.841 E-5	&	14.009	&	1.556 E-5	&	1.237 E-2	&	1.302 E-5	\\
15	&	354.71	&	0.23272	&	0.25927	&	15.004	&	1.873 E-5	&	1.698 E-2	&	1.532 E-5	&	15.019	&	1.325 E-5	&	1.201 E-2	&	1.084 E-5	\\
16	&	357.25	&	0.23291	&	0.25937	&	16.002	&	1.682 E-5	&	1.673 E-2	&	1.381 E-5	&	16.019	&	1.190 E-5	&	1.184 E-2	&	9.773 E-6	\\
17	&	342.25	&	0.23176	&	0.25879	&	17.019	&	1.462 E-5	&	1.631 E-2	&	1.176 E-5	&	17.007	&	1.035 E-5	&	1.154 E-2	&	8.321 E-6	\\
18	&	344.86	&	0.23197	&	0.25889	&	18.031	&	1.331 E-5	&	1.611 E-2	&	1.075 E-5	&	18.021	&	9.421 E-6	&	1.140 E-2	&	7.605 E-6	\\
19	&	336.07	&	0.23127	&	0.25854	&	19.035	&	1.168 E-5	&	1.569 E-2	&	9.254 E-6	&	19.025	&	8.264 E-6	&	1.110 E-2	&	6.548 E-6	\\
20	&	337.66	&	0.23140	&	0.25861	&	20.009	&	1.077 E-5	&	1.554 E-2	&	8.565 E-6	&	20.022	&	7.619 E-6	&	1.099 E-2	&	6.061 E-6	\\
\midrule
21	&	335.48	&	0.23122	&	0.25852	& 654.650	&	1.545 E-5	&	2.416 E-2	&	1.222 E-5	&	519.590	&	1.093 E-5	& 1.708 E-2 &	8.644 E-6	\\
\bottomrule
\end{tabular}
}
\end{table}

\begin{table}[h!]
\setlength\tabcolsep{4pt} 
\centering
\caption{ \small Numerical bounds for $\mathscr{E}_L(x)$ from \corref{corpi-main2-gen}:\\
If $\displaystyle n_L\ge 2$ and 
$\displaystyle
\log x \geq  c_0   n_L (\log d_L )^2 , 
$
then 
$\displaystyle
  \mathscr{E}_L (x) \leq a_0  e^{ -   b_0  \sqrt{\frac{\log x}{n_L}} }.
$
}
\label{beta0-present-cor2.4}
      \begin{tabular}{|c||c|c|c||c|c|c|}
\toprule
\multicolumn{1}{|c||}{} 
      &\multicolumn{3}{c||}{When $\beta_0$ exists} 
      &\multicolumn{3}{c@{}|}{When $\beta_0$ does not exist}\\
\midrule
$n_L$	&	$ a_0 $	&	$ b_0 $	&	$ c_0 $ 	&	$ a_0 $	&	$ b_0 $	&	$  c_0  $ \\
\midrule
2	&	46.1831	&	0.25	&	728.705 	&	32.6846 	&	0.25	&	728.705 \\
3	&	137.697	&	0.25	&	111.618 	&	97.4145	&	0.25	&	111.618 \\
4	&	238.328	&	0.25	&	51.4056 	&	168.612 	&	0.25	&	51.4056 \\
5	&	470.197	&	0.25	&	23.9668 	&	332.625 	&	0.25	&	23.9668 \\
6	&	640.325	&	0.25	&	15.8139 	&	453.008 	&	0.25	&	15.8139 \\
7	&	1066.09	&	0.25	&	9.78673 	&	754.188	&	0.25	&	9.78673 \\
8	&	1309.41 	&	0.25	&	7.35812 	&	926.389	&	0.25	&	7.35812 \\
9	&	1886.57 	&	0.25	&	5.29321 	&	1334.69 	&	0.25	&	5.29321 \\
10	&	2166.54	&	0.25	&	4.27670 	&	1532.86	&	0.25	&	4.27670 \\
11	&	3045.99 	&	0.25	&	3.25744 	&	2155.04	&	0.25	&	3.25744 \\
12	&	3364.38 	&	0.25	&	2.74618 	&	2380.43	&	0.25	&	2.74618 \\
13	&	4520.51	&	0.25	&	2.19639 	&	3198.42 	&	0.25	&	2.19639 \\
14	&	4868.41 	&	0.25	&	1.90474 	&	3444.76	&	0.25	&	1.90474 \\
15	&	6321.82 	&	0.25	&	1.57649 	&	4473.06 	&	0.25	&	1.57649 \\
16	&	6680.25	&	0.25	&	1.39551 	&	4726.91	&	0.25	&	1.39551 \\
17	&	8456.72	&	0.25	&	1.18426 	&	5983.82	&	0.25	&	1.18426 \\
18	&	8822.43	&	0.25	&	1.06438 	&	6243.00	&	0.25	&	1.06438 \\
19	&	10248.5 	&	0.25	&	0.93095 	&	7251.71	&	0.25	&	0.93095 \\
20	&	10779.4	&	0.25	&	0.84415 	&	7628.01 	&	0.25	&	0.84415 \\
\midrule
21 to 519	&     &		&	&	13051.2   	&	0.25	&	0.76073 \\
\midrule
21 to 654	&   18457.3  &	0.25	&	0.76073	&	&	&	 \\
\midrule
$\geq$ 520	&   	&	&	&	1.047 E10 &	0.23	&	0.76073 \\
\midrule
$\geq$ 655	&  1.480 E10  &   0.23	&	0.76073	&	&	&	 \\
\bottomrule
\end{tabular}
\end{table}

\begin{table}[h!]
\setlength\tabcolsep{4pt} 
\centering
\caption{ \small Numerical bounds for $\mathscr{E}_L(x)$ from \corref{corpi-main2-gen}:\\
If $\displaystyle n_L\ge n_0\geq 2$ and 
$\displaystyle
\log x \geq  c_0   n_L (\log d_L )^2 , 
$
then 
$\displaystyle
  \mathscr{E}_L (x) \leq a_0  e^{ - 0.23  \sqrt{\frac{\log x}{n_L}} }.
$
}
\label{beta0-present-cor2.4-for-nL>=n0}
      \begin{tabular}{|c||c|c||c|c|}
\toprule
\multicolumn{1}{|c||}{} 
      &\multicolumn{2}{c||}{when $\beta_0 $ exists} 
      &\multicolumn{2}{c@{}|}{when $\beta_0 $ does not exist}\\
\midrule
$n_0$	&	$ a_0 $	&	$  c_0  $ 	&	$ a_0 $	&	$  c_0  $ \\
\midrule
2	&	174.707	  & 728.705 	&	123.643 	&	728.705 \\
3	&	1125.19	&	111.618 	&	797.005	&	111.618 \\
4	&	2752.61	&	51.4056 	&	1949.78 	&	51.4056 \\
5	&	13667.2	&	23.9668 	&	9668.48 	&   23.9668 \\
6	&	23258.1	&	15.8139 	&	16454.4 	&   15.8139 \\
7	&	1.085 E5	&	9.78673 	&	76773.3	&	9.78673 \\
8	&	1.542 E5 &	7.35812 	&	1.091 E5	& 	7.35812 \\
9	&	5.669 E5 &	5.29321 	&	4.010 E5 	&	5.29321 \\
10	&	6.712 E5	 &	4.27670 	&   4.749 E5	&	4.27670 \\
11	&	3.219 E6 &	3.25744 	&	2.278 E6 	&	3.25744 \\
12	&	3.345 E6 	&   2.74618 	&	2.367 E6	&	2.74618 \\
13	&	1.884 E7 	&	2.19639 	&	1.333 E7	&	2.19639 \\
14	&	1.733 E7 	&	1.90474 	&	1.226 E7	&	1.90474 \\
15	&	1.285 E8 	&   1.57649 	&	9.094 E7	&	1.57649 \\
16	&   1.008 E8 	&	1.39551 	&	7.133 E7	&	1.39551 \\
17	&	1.238 E9 	&   1.18426 	&	8.758 E8	&	1.18426 \\
18	&	7.774 E8 	&	1.06438 	&	5.501 E8	&	1.06438 \\
19	&	6.863 E9 	&	0.93095 	&	4.856 E9	&	0.93095 \\
20	&	4.601 E9	&	0.84415 	&	3.256 E9 	&	0.84415 \\
21	&   1.480 E10	&	0.76073 	&	1.047 E10 	&	0.76073 \\
\bottomrule
\end{tabular}
\end{table}

 \end{document}